\documentclass[12pt,reqno]{amsart}

\usepackage{latexsym}
\usepackage{amssymb}

\renewcommand{\epsilon}{\varepsilon}

\newcommand{\R}{{\mathbb R}}

\newcommand{\Z}{{\mathbb Z}}

\newcommand{\Ss}{{\mathbb S}}

\renewcommand{\phi}{\varphi}

\newcommand{\acal}{\mathcal{A}}

\newcommand{\ccal}{\mathcal{C}}
\newcommand{\dcal}{\mathcal{D}}
\newcommand{\ecal}{\mathcal{E}}
\newcommand{\fcal}{\mathcal{F}}
\newcommand{\gcal}{\mathcal{G}}
\newcommand{\hcal}{\mathcal{H}}

\newcommand{\ocal}{\mathcal{O}}

\newcommand{\rcal}{\mathcal{R}}

\newcommand{\tcal}{\mathcal{T}}

\newcommand{\zcal}{\mathcal{Z}}

\newcommand{\Om}{\Omega}

\newtheorem{theo}{{\sc Theorem}}[section]
\newtheorem{maintheo}{{\sc Theorem}}

\newtheorem{maincor}[maintheo]{{\sc Corollary}}
\newtheorem{maindefn}[maintheo]{{\sc Definition}}
\newtheorem{lem}[theo]{{\sc Lemma}}
\newtheorem{mainprop}[maintheo]{{\sc Proposition}}
\newtheorem{prop}[theo]{{\sc Proposition}}
\newtheorem{mainprob}[maintheo]{{\sc Problem}}

\newenvironment{rem}{\medskip\noindent{\it Remark:\/} }{\medskip}
\newtheorem{defn}[theo]{{\sc Definition}}

\title[Quantum ergodic restriction]
{Quantum ergodic restriction theorems, I:  interior  hypersurfaces
in domains with ergodic billiards}

\author{John  A. Toth}
\address{Department of Mathematics and Statistics, McGill University, Montreal, CANADA } \email{jtoth@math.mcgill.ca} \thanks{Research partially supported by NSERC grant \# OGP0170280 and a William Dawson Fellowship}

\author{Steve Zelditch}
\address{Department of Mathematics, Johns Hopkins University, Baltimore,
MD 21218, USA} \email{zelditch@math.jhu.edu}
\address{Department of Mathematics, Northwestern  University,
Evanston, IL 60208-2370, USA} \email{
zelditch@math.northwestern.edu}

\thanks{Research partially supported by NSF grants \# DMS-0603850 and  and  \# DMS-0904252.}

\date{\today}

\begin{document}

 \maketitle

 \begin{abstract}  Quantum ergodic restriction (QER) is the
 problem of finding conditions on a hypersurface $H$ so that
 restrictions $\phi_j |_H$ to $H$  of $\Delta$-eigenfunctions of Riemannian
 manifolds $(M, g)$ with ergodic geodesic flow are quantum ergodic
 on $H$. We prove two kinds of results: First (i) for any smooth
 hypersurface $H$, the Cauchy data $(\phi_j|H, \partial \phi_j|H)$
 is quantum ergodic if the Dirichlet and Neumann data are weighted
 appropriately. Secondly (ii) we give conditions on $H$ so that
 the Dirichlet (or Neumann) data is individually quantum ergodic. The
 condition involves the almost nowhere equality of left and right
 Poincar\'e maps for $H$. The proof involves two further novel
 results: (iii) a local Weyl law for boundary traces of
 eigenfunctions, and (iv) an `almost-orthogonality' result for
 Fourier integral operators whose canonical relations almost
 nowhere commute with the geodesic flow.
 \end{abstract}

 \tableofcontents

\section{Introduction}

\medskip

This is the first in a series of articles on ``quantum ergodic
restriction" (or QER) theorems. Quantum ergodicity (in the
Riemannian setting) is the study of asymptotics of eigenfunctions
of the Laplacian $\Delta_g$ on Riemannian manifolds $(M, g)$ with
ergodic geodesic flow.  The QER question is whether restrictions
of $\Delta_g$-eigenfunctions to hypersurfaces $H \subset M$ are
quantum ergodic along the hypersurface.
 More precisely, the aim
is to find conditions on $H$ so that the Dirichlet data
$\phi_{\lambda}|_H$, or  Neumann data $\partial_{\nu}^H
\phi_{\lambda}|_H$ or full Cauchy data $(\phi_{\lambda}|_H,
\partial^H_{\nu} \phi_{\lambda}|_H)$  is quantum ergodic along
 $H$. Here, $\partial_{\nu}^H$ is a fixed
choice of unit normal to $H$.

  In this article, we consider
the case of smooth  hypersurfaces $H \subset \Omega$ in piecewise
smooth Euclidean domains $\Omega \subset \R^N$ with `totally'
ergodic billiards, i.e. such  that all powers $\beta^k: B^*
\partial \Omega \to B^* \partial \Omega$ of the billiard map are
ergodic. We first prove that in a suitably weighted sense, the
Cauchy data is quantum ergodic for any hypersurface. This is a
general phenomenon that holds for all Riemannian manifolds with
boundary  (further results are proved in  \cite{CHTZ}). We then
consider the subtler question of when the Dirichlet data (or
Neumann data) is itself quantum ergodic. It is simple to find
examples of $H$ for which $\phi_j |_H$ are not quantum ergodic.
But in  Theorem \ref{maintheorem} we give a generic sufficient
condition on
  $H$ in
terms of a `dynamical' condition on the `transfer map
$\tau^{\partial \Omega}_H: B^* H \to B^ *
\partial \Omega$ from tangent co-vectors to $H$ to the boundary (see Definitions \ref{TRANSFERMAP} and \ref{ANC}).
The proof is  based on a  seemingly new principle in quantum
dynamics: the almost orthogonality, $\langle F(\lambda_j) \phi_j,
\phi_j \rangle \to 0$,  of $\Delta$- eigenfunctions
$\phi_{\lambda}$ (and their Cauchy data) to their images
$F(\lambda) \phi_{\lambda}$ under a semi-classical Fourier
integral operator whose underlying symplectic map almost nowhere
commutes with the geodesic flow (or billiard map).

To our knowledge, the only prior quantum ergodic restriction
theorem concerns the Cauchy data along the boundary of
eigenfunctions of the Laplacian on bounded domains satisfying
Dirichlet or generalized Neumann boundary conditions
\cite{GL,HZ,Bu}. However, heuristic and numerical work of
Hejhal-Rackner \cite{HR} appears to show that restrictions of
eigenfunctions on finite area hyperbolic surfaces to closed
horocycles are quantum ergodic.  We restrict to Euclidean domains
in this article because there are special techniques available for
them and there are many interesting examples.  At the end of the
introduction, we discuss related work (and work in progress) on
quantum ergodic restriction theorems in more general settings.

To state our results, we introduce some notation (See \S
\ref{NOTATION} for a notational index).  We consider the
eigenvalue problem in a piecewise smooth bounded Euclidean domain
$\Omega \subset \R^n$  with
  Neumann (resp. Dirichlet) boundary conditions:
\begin{equation} \label{EIG}  \left\{ \begin{array}{ll} - \Delta \phi_{\lambda_j}  = \lambda_j^2
\phi_{\lambda_j} \;\; \mbox{in}\;\; \Omega, & \;\;
\\ \\\partial_{\nu} \phi_{\lambda_j}
 = 0 \;\;(\mbox{resp.} \; \phi_{\lambda_j} = 0) \; \mbox{on} \;\; \partial \Omega.
\end{array} \right. \end{equation} Here, $\Delta$ is the Euclidean
Laplacian and  $\partial_{\nu}$ is the interior unit normal to
$\Omega$. We denote by $\{\phi_{\lambda_j}\}$ an orthonormal basis
of eigenfunctions of the boundary value problem corresponding to
the eigenvalues $\lambda_0^2 < \lambda_1^2 \leq \lambda_2^2
\cdots$ enumerated according to multiplicity. We denote by $\beta:
B^*
\partial \Omega \to B^*
\partial \Omega$ the billiard map of $\Omega$ (see \S \ref{BMAP}).

The  quantum ergodic restriction problem for Dirichlet data is to
determine the limits of the diagonal matrix elements
\begin{equation} \label{ME} \rho_j^H(Op^H_{\lambda}(a)): =  \langle Op_{\lambda_j}(a)
\phi_{\lambda_j}|_{H},\phi_{\lambda_j}|_{H} \rangle_{L^{2}(H)},
\;\; Op_{\lambda}^H(a) \in \Psi^0_{sc}(H))
\end{equation}
of (semi-classical) pseudo-differential operators
$Op_{\lambda}^H(a)$ along $H$. That is, $a \in S^{0,0}_{cl}(T^*H
\times (0,\lambda_0^{-1}] )$ lies in  the class of zeroth-order
semiclassical symbols with polyhomogeneous expansions in
$\lambda^{-1}$  with $\lambda \in[\lambda_0, \infty)$  We recall  that \cite{DS}
\begin{equation} \label{scsymbol} \begin{array}{ll}
S^{0,0}_{cl}(T^*H \times (0,\lambda_0^{-1}] ):= \{ a(s,\tau;\lambda) \in C^{\infty}(T^*H \times
 (0,\lambda_0^{-1}]); a \sim_{\hbar \rightarrow 0} \sum_j a_j \lambda^{-j}, \\ \\ 
    |\partial^{\alpha}_{s} \partial^{\beta}_{\tau} a_j(s,\tau)| \leq C_{j,\alpha,\beta} \langle \tau \rangle^{-j - |\beta|} \}. \end{array} \end{equation}
  For simplicity, we denote the corresponding space of semiclassical pseudodifferential operators by
  $$ \Psi^{0}_{sc}(H):= Op_{\lambda}( S^{0,0}_{cl}(T^*H \times [0,\lambda_0^{-1}] ).$$

 The sequence of Dirichlet data
$\phi_{\lambda_j}|_{H}$ is called quantum ergodic on $L^2(H)$ if
almost all of the matrix elements (\ref{ME}) tend to the integral
$\omega_{\infty}^H(a) = \int_{B^* H} a_0 d\mu_{\infty}$ of the
principal symbol $a_0$ against the  quantum limit measure
$d\mu_{\infty}$ where $d\mu_{\infty}$ is the image under the
transfer map $\tau_H^{\partial \Omega}$ (see Definition
\ref{TRANSFERMAP}) to $B^* H$ of the limit measure $d\mu_B$ of
\cite{HZ} with boundary conditions $B$.
 The quantum ergodic restriction problem for Neumann data is
 similar, while that for
Cauchy data is to study the matrix elements  $\langle {\bf A}_2
CD_H(\phi_{\lambda_j}), CD_H(\phi_{\lambda_j}) \rangle$, where
${\bf A}_2$ is a $2 \times 2$ matrix of pseudo-differential
operators on $H$. As  can be seen  in the table  below, we need to weight the
Dirichlet and Neumann data in a special way and to choose special
${\bf A}_2$ to obtain an almost unique weak limit.

To put the quantum ergodic restriction problem into context,  we
recall that the Dirichlet or Neumann eigenfunctions are quantum
ergodic on $M$ (i.e. in the `interior') whenever the geodesic flow
of $(M, g)$  (resp. billiard flow, if $\partial M \not\emptyset$) is ergodic. In this article we restrict our attention
to bounded domains in $\R^n$, and   refer to \cite{GL,ZZ} for the
relevant interior quantum ergodic theorems.  It  is shown in
\cite{GL,HZ,Bu} that $(M,
\partial M)$ satisfies QER if the classical billiard flow of $M$
is ergodic, with $\omega_{\infty}^H$ depending on the choice of
boundary conditions (which we suppress in the notation). The
boundary quantum ergodic limit measures are as follows \cite{HZ}:

\bigskip
\noindent \hskip 100pt { \begin{tabular}{|c|c|c|c|c|}
 \hline \multicolumn{4}{|c|}{\bf Boundary traces and quantum limits} \\
\hline
 B & $Bu$  &   $u^{b}$ & $d\mu_B$   \\
\hline
 Dirichlet &  $u|_{Y}$ &  $\lambda^{-1}\; \partial_{\nu} u |_{Y}$ & $\gamma(q) d\sigma $ \\
\hline
 Neumann &   $\partial_{\nu} u |_{Y}$  &  $u|_{Y}$  & $\gamma(q)^{-1}
 d\sigma$\\
\hline
 \end{tabular}}
\bigskip

Here, we denote by  $d \sigma = dq d \eta $  the natural
symplectic volume measure on $B^*\partial \Omega$. The standard
cotangent projection is
 denoted by $\pi: T^* \partial \Omega \rightarrow \partial \Omega.$ We also define
the function $\gamma(q, \eta)$ on $B^*\partial \Omega$ by
\begin{equation}
\gamma(q, \eta) = \sqrt{1 - |\eta|^2}.
\label{a-defn}\end{equation}

 The set $S^*M |_H$ of unit vectors to
$M$  along $H$ is a kind of cross-section to the geodesic or
billiard flow and one might expect that on both the classical and
quantum levels, the ambient ergodicity induces ergodicity on the
cross section. However, this heuristic idea is based on the
assumption  that the Cauchy data (or Dirichlet data) on a
hypersurface is a suitable notion of quantum cross section. In
effect we are investigating the extent to which that is true (the
abstract microlocal notion of quantum cross section is studied in
\cite{CHTZ2}).

\subsection{Quantum ergodic restriction theorem for Cauchy data}

The difference  in the quantum ergodic restriction problems for
Cauchy data versus  Dirichlet data is visible in  the
 one-dimensional case where $M = S^1$ or where $M = [0, \pi]$.  Modulo
time-reversal, the classical systems are obviously ergodic and
consist of just one orbit. The real-valued eigenfunctions $\sin k
x, \cos k x$  do not individually have quantum ergodic
restrictions to a hypersurface (i.e. a point) $x = x_0$. Indeed,
they oscillate as $k \to \infty$. On the other hand, if we weight
the real Neumann data by $\frac{1}{k}$ then the formula $\cos^2 kx
+ \sin^2 k x = 1$ shows that the  Cauchy data is quantum ergodic.

A somewhat more revealing example is that  of eigenfunctions of
semi-classical Schr\"odinger operators $h^2 D^2 + V$ in one
dimension and with connected level sets of the energy $\xi^2 +
V(x) = E$. The real valued eigenfunctions have the WKB form
$A_{\hbar} \cos (\frac{1}{\hbar} S(x))$ and we see that again the
Dirichlet data at a point fails to be quantum ergodic but that the
Cauchy data is quantum ergodic if  the Dirichlet data is weighted
by $S'(x)$ and the Neumann data by $\hbar$.

As these examples show, we should study the {\it normalized Cauchy
data} \begin{equation} \label{CAUCHYDATA} CD_H(\phi_{\lambda}): (\phi_{\lambda}|_H, \lambda^{-1}
\partial_{\nu} \phi_{\lambda}|H ). \end{equation} Henceforth, Cauchy data will
always refer to this normalization.

To state our result, we introduce some further notation. We assume
that $H \subset \Omega^o$, although with some extra technical
complications, we could allow $H$ to intersect $\partial \Omega$
(e.g. the midline of the Bunimovich stadium).  We denote by $s
\mapsto q_H(s) \in H$ a smooth parameterization of the
hypersurface $ H \subset \Omega$ with $s \in U \subset
\mathbb{R}^{n-1}$ an open set and $ U \in y \mapsto q(y) \in
 \partial \Omega$  a local smooth parametrization of the boundary of $ \Omega.$ The dual  coordinates are
  $\tau \in T_{s}^*(H) $ and $\eta \in T_{y}^*(\partial \Omega)$
  respectively. We sometimes abuse notation and use $(s, \tau)$ as
  coordinates on $B^* H$, resp. $(y, \eta)$ as coordinates on $B^*
  \partial \Omega$.

The examples above  indicate that we can only expect QER for
rather special  matrices ${\bf A}_2$ of pseudo-differential
operators.  We say  that ${\bf A}_2(\hbar)$ is a {\it normalized
scalar} matrix of semi-classical pseudo-differential operators if,

\begin{itemize}

\item (i)\; $A_{21}(\hbar) = A_{12}(\hbar) = 0$; \medskip

\item (ii) $\sigma_{A_{11}}(s, \tau) =  (1 - |\tau|^2) \sigma_{
A_{22}}(s, \tau)$ for all $(s, \tau) \in B^* H$.

\end{itemize}

\begin{maintheo} \label{CDTHM} Assume that $\Omega \subset \R^n$ is a
piecewise smooth Euclidean plane domain with totally ergodic
billiards, and let $H \subset \Omega^o$ be any smooth
hypersurface. Also, let ${\bf A}$ be a normalized scalar matrix of
zeroth order pseudo-differential operators on $H$.  Then, there
exists a density-one subset $S$ of ${\mathbb N}$ such that
   the (normalized) Cauchy
data of  Dirichlet or Neumann eigenfunctions satisfies
$$ \lim_{\lambda_j \rightarrow \infty; j \in S}  \langle {\bf A}_2 CD_H(\phi_{\lambda_j}), CD_H(\phi_{\lambda_j}
\rangle = \int_{B^* H} (\sigma_{A_{11}} +
\sigma_{A_{22}})d\mu_{\infty},
$$ where $d\mu_{\infty} = (\tau_{\partial \Omega}^H)_* d\mu_B$ is the  push-forward by
$\tau_{\partial \Omega}^H$ (Definition \ref{TRANSFERMAP}) to $B^*
H$ of the limit measure $d\mu_B$ in the table above.

\end{maintheo}
In other words, there exists a measure $d\mu_{\infty}$ on $B^*H$
so that along a subsequence of eigenvalues of density one we have,
\begin{equation}\label{TOPROVE} \begin{array}{c} \langle Op_{\lambda}((1 - |\tau|^2) a(s,\tau))
\phi_{\lambda} |_H, \phi_{\lambda}|_H \rangle + \lambda^{-2}
\langle Op_{\lambda}(a(s, \tau)) \partial_{\nu_H} \phi_{\lambda}
|_H, \partial_{\nu_H} \phi_{\lambda}|_H \rangle \\ \\ \to
\int_{B^* H} 2 (1 - |\tau|^2)
 a d\mu_{\infty}. \end{array}
\end{equation}

When  the correspondence $\tau_{\partial \Omega}^{H}$ has two
branches, so that  $\Gamma_{\tau^{H}_{\partial \Omega}
}\Gamma_{\tau^{H;1}_{\partial \Omega}} \cup
\Gamma_{\tau^{H;2}_{\partial \Omega}} $
 where $\tau^{H;j}_{\partial \Omega}; j=1,2,$ are single-valued maps (see section
 \ref{transferformulas}), the push-forward is given by
$$( \tau_{\partial \Omega}^{H} )_{*} := ( \tau_{\partial \Omega}^{H;1} )_{*} + ( \tau_{\partial \Omega}^{H;2} )_{*}.$$
Similarly if there are more intersection points.

Although we do not do so here, Theorem \ref{CDTHM} can also be
proved by using interior quantum ergodicity \cite{ZZ} and  a
Rellich-type commutator argument. As a consequence, we show in
\cite{CHTZ2}
 that this theorem can be naturally interpreted as  a quantum ergodicity result for
  the quantum flux norm of the Cauchy data along  the hypersurface, $H \subset \Omega.$
Moreover, the latter argument  is insensitive to whether or not
$\partial  \Omega \neq 0.$
  Thus, our result in \cite{CHTZ2} extends Theorem \ref{CDTHM} to  interior hypersurfaces $H \subset M$
   of arbitrary compact ergodic manifolds $(M,g)$ with or without boundary.

\subsection{Quantum ergodic restriction for Dirichlet data}

The one-dimensional case raises doubts that there can exist a
quantum ergodic restriction theorem for the Dirichlet data alone.
However, it is not hard to show that quantum ergodic restriction
does hold for  random waves on general smooth  hypersurfaces in
dimensions $n \geq 2$. In this section, we give a sufficient
condition on $H$ so that QER is valid.

To motivate the condition, we  consider another special case of
the QER problem: i.e.
 the case  of Neumann eigenfunctions of a
Bunimovich stadium $ S \subset \R^2 $,  with the curve $H$ given
by the vertical midline of $S$. We note that $S$ has a left-right
symmetry $\sigma_L$ (and an up-down symmetry). The midline $H$ is
pointwise fixed by $\sigma_L$. `Half' of the Neumann
eigenfunctions are even with respect to $\sigma_L$ and half are
odd (in the sense of spectral density). Obviously, the odd
eigenfunctions restrict to zero on $H$. On the other hand, the
even ones are quantum ergodic on $H$: in the quotient domain
$S/\Z_2$ by $\sigma_L$,   $H$ becomes a boundary component and
even Neumann eigenfunctions of $S$ are Neumann eigenfunctions of
$S/\Z_2$. Hence a full density subsequence of their restrictions
to $H$ have the boundary quantum limit $d\mu_B$  given  in the
table above.  Thus, there exist two subsequences of eigenfunctions
of density $1/2$ with different quantum limits, and therefore QER
(quantum ergodic restriction) of the Dirichlet fails for $(S,H)$.
Consistently with Theorem \ref{CDTHM}, the Cauchy data is quantum
ergodic. Indeed, the Neumann data is quantum ergodic on the
midline by the arguments of \cite{HZ,Bu} applied to the
half-stadium with Neumann boundary conditions on all but the
midline and with Dirichlet boundary conditions on the midline. A
more general example of the same kind is a curve $H$ which
intersects the midline in a segment of the midline. The same kind
of example exists on the fixed point set of any manifold (with or
without boundary) that possesses a $\Z_2$ symmetry.

We now give a sufficient condition for restricted quantum
ergodicity of the Dirichlet data which, while difficult to verify
in many  cases, is generic for smooth  hypersurfaces, at least for
some classes of domains.  As above, we let $H \subset
 \Omega^o$ (the interior of $\Omega$)
denote a  smooth hypersurface.  The key objects associated to $H$
are the transfer maps $\tau_{\partial \Omega}^H$ and
$\tau_H^{\partial \Omega}$. They are somewhat complicated because
of the possible number of intersection points of a billiard
trajectory with $H$. In \S \ref{BMAPSH}, we will give a more
detailed definition of the transfer maps (see Definition
\ref{SECONDTRANSFER}).

 Let $(y,\eta) \in B^* \partial \Omega$ and
let $q(y) + t \zeta(y, \eta)$ denote its billiard trajectory where
(see \S \ref{BMAP})
\begin{equation} \label{zetayeta} \zeta(y,\eta) := \eta + \sqrt{1-|\eta|^{2}} \nu_y \end{equation}
is the lift of $(y, \eta)$ to an interior unit vector to $\Omega$
at $y$.

\begin{maindefn}\label{TRANSFERMAP}  The   transfer map   $\tau_{\partial \Omega}^H: B^* \partial \Omega \to B^* H$
is defined as follows: We say that $(y, \eta)$ is in the domain of
$\tau_{\partial \Omega}^H$,  (ie. $(y, \eta) \in
\dcal(\tau_{\partial \Omega}^H)$)  if its billiard trajectory
intersects $H$. Denote by $E(t_H(y,\eta), y, \eta) \in H$ any
intersection point (see  section \ref{transferformulas}) and by
$\pi^T_{E(t_H(y,\eta),y,\eta)} \zeta(y, \eta)$ the tangential
projection of the terminal velocity at $E(t_H(y,\eta),y,\eta)$ to
$T_{E(t_H(y,\eta),y,\eta)} H$.  The graph of transfer map is by
definition the canonical relation,
\begin{multline} \label{TRANSFERRELATION} \Gamma_{\tau_{\partial
\Omega}^H}: =\{(y, \eta, s(y,\eta),\pi^T_{E(t_H(y,\eta),y,\eta)}
\zeta(y, \eta)) \in B^* \partial \Omega \times B^* H,  \\
(y,
\eta) \in \dcal(\tau_{\partial \Omega}^H), \,  q_H(s(y,\eta)) = E(t_H(y,\eta),y,\eta) \}.
\end{multline}

In the inverse direction, we have the  double-valued transfer map
(or correspondence)
\begin{equation} \label{TRANSFERH} \tau_H^{\partial \Omega}: B^*H \to B^* \partial \Omega, \end{equation}
 defined by taking $(s, \tau) \in B^* H$ to the two unit
covectors  $\xi_{\pm}(s, \tau) \in S^*_{y(s,\tau)} \Omega$  which
project to it, following each of their trajectories until they hit
the boundary and then projecting
 each terminal velocity vector to $B^* \partial \Omega$. The   graph of $ \tau_H^{\partial \Omega}$  is the canonical
relation,
\begin{equation} \label{TRANSFERHH} \Gamma_{\tau_H^{\partial
\Omega}}: =\{(s, \tau, y(s,\tau), \pi^T_{q(y(s,\tau))} \zeta(s,
\tau)) \in B^*H \times B^* \partial \Omega \}. \end{equation}
\end{maindefn}

We emphasize that  $\tau_{H}^{\partial \Omega}$ is  double-valued
due to the fact that each $(s,\tau) \in B^* H$ lifts to two unit
covectors \begin{equation} \label{xi+-} \xi_{\pm}(s,\tau) = \tau +
\sqrt{1 - |\tau|^2} \nu_{\pm}  \in S^*_s \Omega \end{equation}
which project to $\tau$. More precisely,  the normal bundle $N^*
H$ to an orientable hypersurface  $H$ decomposes into two $\R_+$
bundles $N^*_{\pm}$, which we view as the two infinitesimal sides
of $N^* H$. Then $\xi_{\pm}(s,\tau) \in N^*_{\pm} H$ are the lifts
of $\tau $ to unit covectors on the two sides. The lifts are
analogous to (\ref{zetayeta}) but only the interior side of the
boundary is  relevant in (\ref{zetayeta}).  The two lifts in
(\ref{xi+-})  give rise to  two branches
\begin{equation} \label{TAUPM} \tau_{\pm}: B^*H \to B^* \partial \Omega \end{equation}
of $\tau_H^{\partial \Omega}$ by taking $\tau_{\pm}(s, \tau)$ to
be the element of $\tau_H^{\partial \Omega}(s, \tau)$ coming from
the trajectory defined by  the unit vector $\xi_{\pm}(s, \tau) \in
S^*_{q_H(s,\tau)} \Omega$. This double-valued-ness is independent
of the shape of $H$ and is an essential feature of the quantum
restriction problem.

On the other hand, $\tau_{\partial \Omega}^H$ is multi-valued due
to the shape of $H$, specifically to the many possible
intersection points of a billiard trajectory with $H$. This  is an
unavoidable feature of the QER problem which  causes notational
inconveniences more than essential analytical problems. For
expository simplicity, we will assume that $H$ is weakly convex,
so that a generic line intersects $H$ in at most two points. This
saves us from tedious notations and indices for the branches of
$\tau_{\partial \Omega}^H$ caused by the multiple intersection
points. In the course of the proof we indicate the modifications
necessary to deal with general smooth hypersurfaces. In Definition
\ref{SECONDTRANSFER}, we introduce notation  for the multiple
branches.

Another inevitable and tedious complication is that  the
correspondences have singularities at directions where billiard
trajectories intersect $H$ or $\partial \Omega$ tangentially or
which run into a corner of $\partial \Omega$. It is not apriori
clear that such tangential rays  cause only  technical
complications. Without a detailed analysis, it is possible that
eigenfunction mass could get concentrated microlocally in
tangential directions to $H$. In the case $H = \partial \Omega$ it
was proved in \cite{HZ} that no such concentration occurs. The
same method shows that no such concentration occurs on $H$ either.
  In \S \ref{extensions} we give a self-contained  pointwise Weyl law
 argument to show that for a full density of $u_{\lambda}^{H}$'s, no such tangential
  mass concentration occurs. This method works equally well for manifolds,
  $M$, with or without boundary.

Another important symplectic correspondence associated to $H$ is
the once-broken  {\it transmission billiard map} (or
correspondence) through $H$, $\beta_H: B^*
\partial \Omega \to B^* \partial \Omega$,
defined by
\begin{equation} \label{TRANS} \beta_H (y, \eta)
= \left\{ \begin{array}{ll} \tau_- \tau_+^{-1} (y, \eta), &\;\;\;
 (y, \eta) \in \mbox{range}\;\; \tau_+,\\ & \\
\tau_+ \tau_-^{-1}(y, \eta), & (y, \eta) \in \mbox{range} \;
\tau_-.
\end{array} \right.
\end{equation}
 Equivalently, $\beta_H$ maps  the two endpoints of the rays
through $\xi_{\pm}(s, t)$ with $(s, t) \in B^* H$ to each other.
Thus, $\beta_H$ follows the trajectory of  $(y, \eta) \in B^*
\partial \Omega$ backward  to its intersection point  $ q_H(s(y,\eta)) \in H$ (assuming it
intersects $H$),
 breaks at $q_H(s(y,\eta))$ by the law of equal angles,  and then proceeds along the second link
 in the forward direction to  $\partial
\Omega$. Since the trajectory defined by $(s,\tau) \in B^* H$ may
have multiple intersection points with $H$,  this map is also a
multi-valued correspondence. With the simplifying assumption that
$H$ is weakly convex, a generic trajectory intersects $H$ in at
most two points, and we denote the corresponding  branches of
$\beta_H$ by $\beta^k_H; \, k=1,2$. See \S \ref{betas} for further
details. Note that the trajectory {\it breaks only once} even if
the trajectory intersects $H$ twice.

\begin{rem} \label{SIGMA} $\beta_H$  is not  the same as the broken billiard flow with
a break on $H$ because the trajectory travels backwards on the
first link. Starting  at $(y, \eta) \in B^* \partial \Omega$, the
map first time reverses the $\eta \to - \eta$, then proceeds on
the broken trajectory and then projects the terminal velocity.
Hence, $\beta_H (y, \eta)$ differs by a time-reversal from the
usual broken billiard trajectory.  As a check, let us note that in
the case of the midline of the stadium, the two unit vectors
projecting to $v \in B^*_s H$ are related by $\sigma_L$. Hence
their rays, and the projections of the terminal velocity vectors,
are related by $\sigma_L$. But the two-link ray starting at $v \in
T_y \partial \Omega$  between them is not so related: the terminal
velocity is the time reversal of $\sigma_L(v)$.
\end{rem}

  Given $A \subset B^*
\partial \Omega$ in the following   we
denote the symplectic volume $|dy d\eta|$- measure of $A$ by
$|A|.$

\begin{maindefn} \label{ANC} We say that  $(\beta_H, \beta)$ {\it almost
nowhere  commute} if  the (commutation) sets $$\ccal \ocal_{p,k}: =
\{(y, \eta) \in B^*\partial \Omega: \beta_H \beta^k (y, \eta) =
\beta^p \beta_H (y,\eta) \} $$
       have symplectic volume measure zero, $|\ccal \ocal_{p,k}| = 0$, for all $p, k = 1, 2, 3,
       \dots$
 \end{maindefn}

The almost nowhere commutativity (ANC) condition may be re-formulated in
terms of `left' and `right' return maps to $B^* H$.  Given
covector $(s, \tau) \in B^*H$, one may lift it in two ways
(\ref{xi+-})  to a unit covector $(s, \xi_{\pm}) \in S^*_H\Omega$
which projects to $\tau$.  We then follow the two geodesics
$\gamma^{\pm}(t)$ with initial conditions $(s, \xi_{\pm})$ until
they return to $H$ on the same side as $(s, \xi_{\pm})$ and then
project the terminal co-vector back to $B^*H$. The almost nowhere
commutativity condition  is equivalent to the condition that these
two oriented return maps do not coincide on a set of positive
measure in $B^*
 \partial \Omega$. The equivalence of the definitions is easily seen by
re-writing the equation  $ \beta_H \beta^k (y, \eta) = \beta^p
\beta_H (y,\eta)$ (see (\ref{TRANS}))  as
$$ \tau_- \tau_+^{-1} \beta^k (y, \eta) = \beta^p \tau_-
\tau_+^{-1} (y,\eta) \iff  \tau_+^{-1} \beta^k
\tau_+(s,\tau) = \tau_- ^{-1} \beta^p \tau_- (s, \tau),
 $$
 with $(y,\eta) = \tau_+(s,\tau).$
We observe that  $ \tau_{\pm}^{-1} \beta^k \tau_{\pm}(s,\tau)$ is
the $\pm \to \pm$ sided return map to $B^* H$ after $k$ bounces
off $\partial \Omega$. The union over $k$ gives the full one-sided
return map. The fact that $k$ and $p$ may differ in Definition
\ref{ANC} then seems rather natural, since the equality of the
one-sided return maps should not depend on the number of bounces
off $\partial \Omega$.

  Our
main result on Dirichlet data is:

   \begin{maintheo} \label{maintheorem} Let $\Omega \subset {\mathbb R}^{n}$ be a
     piecewise-smooth billiard with totally ergodic billiard flow and let $H \subset
     int(\Omega)$ be a  smooth interior hypersurface.  Let $\phi_{\lambda_j}; j=1,2,...$ denote the
       $L^{2}$-normalized Neumann eigenfunctions in $\Omega$. Then,
       if $(\beta_H, \beta)$  almost nowhere  commute,
 there exists a  density-one subset $S$ of ${\mathbb N}$ such that
  for $\lambda_0 >0$ and  $a(s,\tau;\lambda) \in S^{0,0}_{cl}(T^*H \times (0,\lambda_{0}^{-1}]),
  $
$$ \lim_{\lambda_j \rightarrow \infty; j \in S} \langle Op_{\lambda_j}(a_0)
 \phi_{\lambda_j}|_{H},\phi_{\lambda_j}|_{H} \rangle_{L^{2}(H)} = c_n \int_{B^{*}H} a( s, \tau )   \,
 \rho_{\partial \Omega}^{H}(s,\tau) \, ds d\tau. $$
 Here, $c_{n} = \frac{4}{ vol( {\mathbb S}^{n-1}) \,
vol(\Omega) }$ and $ \rho_{\partial \Omega}^{H} ds d\tau : =
(\tau_{\partial \Omega}^{H})_{*} ( \gamma^{-1} dy d\eta). $
Similarly for Dirichlet data on $H$ of   Dirichlet eigenfunctions,
except that $\gamma(y,\eta)= (1-|\eta|^{2})^{-1/2}$, and for
Neumann data if one multiplies the above by $(1 - |\tau|^2)$.
\end{maintheo}

We summarize the conclusion in the following table.

\bigskip
\noindent \hskip 100pt { \begin{tabular}{|c|c|c|c|c|}
 \hline \multicolumn{4}{|c|}{\bf H traces and quantum limits} \\
\hline
B&  Trace on H &   $u^{H}$ & $d\mu_{\infty}$    \\
\hline Dirichlet &  Dirichlet &  $u|_{H}$  & $ (\tau_{\partial
\Omega}^H)_*\gamma(q) d\sigma $ \\ \hline Dirichlet &  Neumann & $
\lambda^{-1} \partial_{\nu_H} u
|_{H}$  & $ (1 - |\tau|^2) (\tau_{\partial \Omega}^H)_*\gamma(q) d\sigma $ \\
\hline  Neumann &  Dirichlet & $u|_{H}$ & $ (\tau_{\partial
\Omega}^H)_*\gamma(q)^{-1} d\sigma $
\\ \hline Neumann &  Neumann & $\lambda^{-1} \partial_{\nu_H} u
|_{H}$ & $ (1 - |\tau|^2) (\tau_{\partial \Omega}^H)_*
\gamma(q)^{-1}
 d\sigma$\\
\hline
 \end{tabular}}
\bigskip

There is a  stronger but more technical version of this result
which uses a quantitative refinement of the almost nowhere
commuting condition (see the definition of quantitative almost nowhere commutation (QANC) condition in Definition \ref{ANCQ}  of subsection \ref{quant condition}).

\begin{maintheo} \label{maintheorema}   Theorem \ref{maintheorem}
 is valid
       as long as  $(\beta_H, \beta)$ quantitatively  almost nowhere
       commute.
\end{maintheo}

To illustrate the result, we reconsider the midline $H$  of the
Bunimovich stadium $S$. In this case, for $(y, \eta) \in
B^*_{q(y)} H$, the two unit vectors   $\xi_{\pm}(y, \eta) \in
S^*_q \Omega $  projecting to $\eta$ are $\sigma_L$-related, i.e.
images of each other under $\sigma_L$. Hence their trajectories,
and the projections to $T \partial \Omega$ of their terminal
velocities, are $\sigma_L$ related. It follows that $\beta_H(y,
\eta) = \sigma_L(y, \eta)$. Hence $\beta^k \sigma_L = \sigma_L
\beta^k$ for all $k$,  and $(S, H)$ violates the condition of
quantitative almost nowhere vanishing. Thus, our condition rules
out this `counterexample' to QER and the original ANC condition. At this time, we are not aware
of any other kind of `counter-example'. Indeed, the equality of
the left and right sided return maps is a kind of left/right
symmetry condition on $H$. It might be equivalent to the existence
of an involutive  isometry $\sigma: M \to M$ fixing a positive
measure of $H$, but at this time we do not know how to prove that.

Unfortunately, it is often difficult to check whether a given
hypersurface  $H$ in the billiard setting satisfies the  almost
nowhere commutativity condition. The next result shows that  the
condition holds  for generic $H$ in a reasonably broad billiard
setting. In \S \ref{fixedpoint} we prove:

\begin{mainprop} \label{GENERICintro} Let $\Omega$ be a two-dimensional dispersing billiard table (i.e. with
hyperbolic billiards). Then for generic convex curves,  $H \subset
\Omega$, $(\beta_H, \beta)$ almost never commute.
\end{mainprop}

This result is just an example of the genericity of the condition.
We only use  hyperbolicity of the billiards to prevent the
existence of a positive measure of self-conjugate directions, and
we restrict to dimension two to simplify one step in the argument.
A much more general result on genericity of the condition should
be possible,  and probably  could be proved by the methods of
\cite{PS1,PS2}.  The non-commutativity condition is also similar
to the condition that Poincar\'e maps of periodic reflecting rays
of generic domains do not have roots of unity as eigenvalues; see
\cite{PS1,PS2,S}.  But that would take us too far afield in this
article and we only give the simpler result to convince the reader
that our non-commutativity condition is far from vacuous. In the
boundaryless case, we believe that the condition can be proved to
hold  for closed geodesics of hyperbolic surfaces which are not
fixed by isometric involutions \cite{TZ2} and possibly for closed
horocycles.

We also remark that the  `total ergodicity assumption' is not
overly resrictive: all of the known examples of ergodic billiards
that we are aware of are totally ergodic. These include: the
Bunimovitch stadium, the cardioid and generic polygonal billiards
\cite{CFS, P}. The first two examples are known to satisfy
stronger mixing conditions (the stadium is in fact Bernuollii
\cite{BU}). It is apparently open as to whether generic polygons
are mixing, but Troubetzkoy \cite{Tr} has recently proved that
these billiards are also totally ergodic.

\subsection{Convexity of $H$ and QER}

As mentioned above, we assume $H$ is convex for simplicity of
exposition, but the assumption is not necessary for the quantum
ergodic restriction to hold. Indeed, we may localize (and
microlocalize) the problem by considering symbols supported on
small regions of $H$. On a sufficiently small region of $B^* H$,
concave or convex hypersurfaces are essentially the same. Flat
regions are different, but since tangential intersections only
occur for a zero measure of rays, it suffices to prove that only a
sparse (density zero)  subsequence of  eigenfunctions can
concentrate microlocally on tangential rays.  We prove such a
result in Lemma \ref{tanmass} using a pointwise local Weyl law. In
\cite{HZ} the analogous  proof is in Section 7 (see Step 2 and
Lemma 7.1), Section 9 (see Lemma 9.2) and in Appendix 12 (see
especially Lemma 12. 5 in the Dirichlet case).

 The fact that QER  is independent of the curvature
 of the hypersurface  may seem surprising in view of
the results of Burq-G\'erard-Tzetkov \cite{BGT} (see also
\cite{R,KTZ,To,So}) on $L^p$ norms of restrictions of
eigenfunctions. In these results, the $L^p$ norm of the restricted
eigenfunctions depends on whether the curve (or hypersurface) is
geodesically convex or not. There is a dramatic difference, for
example, between norms of restrictions to distance circles and to
closed geodesics.

However, these results do not make use of the global assumption of
ergodicity of the geodesic flow. They are sharp in the case of the
round sphere, but need not be sharp for negatively curved surfaces
(for instance). The results of \cite{HZ,Bu} amply indicate that
geodesic curvature of $H$ is irrelevant to quantum ergodic
restriction theorems: Indeed, when $H = \partial \Omega$, QER is
valid for non-convex $H$ and for $H$ with flat sides (e.g. the
Bunimovich stadium).

Since the validity of the results for general smooth hypersurfaces
is possibly surprising, and since we assume convexity at some
points to simplify the notation, we signal in each section where
and when we use the convexity assumption.

\subsection{Applications}
Under the assumptions of Theorem \ref{maintheorem},
 the $L^2$-restriction bounds along $H \subset \Omega$  of a quantum ergodic sequence of eigenfuncitons are uniformly bounded above and below.
Indeed, choosing $a(s, \tau) =1$ in Theorem \ref{maintheorem} gives the following
\begin{maincor} \label{restriction bounds}
Under the same assumptions as in Theorem \ref{maintheorem},
$$ \lim_{\lambda_j \rightarrow \infty; j \in S}  \| \phi_{\lambda_j}|_{H} \|_{L^2(H)}^{2} = c_n  \, \int_{B^*H} \rho_{\partial \Omega}^{H}(s,\tau) ds d\tau. $$
\end{maincor}

 Except for  possible exceptional sparse subsequences of eigenfunctions,  in the ergodic case, the asymptotics in Corollary \ref{restriction bounds}  substantially improve on  some
recent bounds of $L^p$ norms of restrictions of eigenfunctions to
submanifolds in the articles \cite{R,BGT,To,T,So,BR}).

 In subsection \ref{final},  we give an application of Theorem \ref{maintheorem} to
 the study of asymptotic nodal structure of eigenfunctions of ergodic planar
 billiards $\Omega \subset \mathbb{R}^{2}.$ Consider either Dirichlet or
  Neumann eigenfunctions $\phi_\lambda$ with $-\Delta_{\Omega} \phi_\lambda
  = \lambda^{2} \phi_{\lambda}$. In \cite{TZ} we show that for general
  piecewise-analytic domains, the  eigenfunction nodal set
   ${\mathcal N}_{\phi_\lambda} = \{ x \in \Omega; \phi_{\lambda}(x) = 0 \}$
    (where we omit the boundary itself) for both Neumann and Dirichlet boundary conditions satisfies
\begin{equation} \label{nodal1}
\text{card} \, ( {\mathcal N}_{\phi_\lambda} \cap \partial \Omega
) = {\mathcal O}(\lambda).
\end{equation}
Here, (\ref{nodal1}) requires no dynamical assumptions. In Theorem
6 of \cite{TZ} we proved a similar bound for intersections of
nodal lines with interior analytic curves under an additional technical  assumption of `goodness' of the
curve $C$ (see the proof of the corollary). The quantum ergodicity
result in Theorem \ref{maintheorem} yields the following

 \begin{maincor} \label{nodalthm} Let $\Omega \subset {\mathbb R}^{2}$ be a piecewise-smooth, totally ergodic planar
  billiard and $C \subset \Omega$ a real-analytic interior
  curve. Assume $(\beta_C, \beta)$ almost never commute.
   Then, for a full-density of (Dirichlet or Neumann) eigenfunctions,
$$ \text{card} \, ( {\mathcal N}_{\phi_\lambda} \cap C ) = {\mathcal O}(\lambda).$$
\end{maincor}

\subsection{\label{SKETCH} Sketch of proof }

We now outline the proofs of Theorems \ref{CDTHM} and
\ref{maintheorem}. This outline is an integral part of the proofs
of the Theorems; later sections fill in the considerable number of
technical details, but rigidly adhere to the following outline.

The first step is to  use the
 `quantization' of the transfer map $\tau_{\partial \Omega}^H :
B^* \partial \Omega \to B^* H$ to  transfer the Dirichlet data
$u_{\lambda}^H : = \phi_{\lambda}
 |_H$ of the eigenfunctions on $H$ to the Dirichlet data
 $u_{\lambda}^b: = \phi_{\lambda} |_{\partial \Omega}$ on $\partial
 \Omega$.
As in \cite{HZ,TZ}), the semi-classical quantization of the
transfer map is a layer potential integral operator. By Green's
formula, we have
  \begin{equation} \label{boundarylayer}
   u_{\lambda}^{H}(q_H) =  [N_{\partial \Omega}^{H}(\lambda) u_{\lambda}^b](q_H): = \int_{\partial \Omega} N_{\partial \Omega}^{H}(\lambda, q_H,q) \,  u_{\lambda}^{b}(q)
   d\sigma(q),
   \end{equation}
  where  $q\in \partial \Omega, q_H \in H$ and where $N_{\partial \Omega}^H(\lambda)$ is a boundary integral operator defined in Definition \ref{NDEF} and (\ref{NBDYH}). There is a similar operator $N_{\partial \Omega}^{H, \nu}$
  that transfers the Dirichlet data on the boundary to Neumann data $u_{\lambda}^{H, \nu} = \partial_{\nu}^H \phi_{\lambda}|_H$
  on $H$.  By a modification of  a similar argument in \cite{HZ} (see
  \S
\ref{billiard maps}), we show
  that the operators $N_{\partial \Omega}^{H}(\lambda), N_{\partial \Omega}^{H, \nu}(\lambda): C^{\infty}(\partial \Omega) \rightarrow C^{\infty}(H)$
 are  semiclassical
  Fourier Integral operators  quantizing   $\tau_{\partial \Omega}^H : B^* \partial \Omega \rightarrow
  B^*H$.

We then consider matrix elements (\ref{ME})  of semi-classical
pseudo-differential operators on $H$, i.e. operators with symbols
$a \in S^{0,0}_{cl}(T^* H).$ From (\ref{boundarylayer}) it follows that for any $a \in
S^{0,0}_{cl}(T^* H),$
  \begin{equation} \label{CTOB} \langle Op_{\lambda}(a) u_{\lambda}^{H}, u_{\lambda}^{H} \rangle_{L^2(H)}
   = \langle N^{H}_{\partial \Omega}(\lambda)^* Op_{\lambda}(a) N_{\partial \Omega}^{H}(\lambda) u_{\lambda}^{b},
    u_{\lambda}^{b} \rangle_{L^2(\partial \Omega)}. \end{equation}
Similarly for Neumann data.  The formula (\ref{CTOB}) is an
advantage to working in the boundary setting; it allows us to use
results of \cite{HZ} to avoid many technicalities.

At first sight, it appears that this transfer operator fully
reduces our problem to the known results in \cite{HZ} on $\partial
\Omega$. However, this is not the case. The key point is that
$N^{H}_{\partial \Omega}(\lambda)^* Op_{\lambda}(a) N_{\partial
\Omega}^{H}(\lambda)$ is not a pseudo-differential operator.
Rather, in  \S \ref{splitting}, we show that
\begin{equation} \label{decomposition}
N^{H}_{\partial \Omega}(\lambda)^* Op_{\lambda}(a) N_{\partial \Omega}^{H}(\lambda)
 = Op_{\lambda}( (\tau^{H}_{\partial \Omega})^* a   \cdot \rho_{\partial \Omega}^H) + F_{2}(\lambda), \end{equation}
where $\rho_{\partial \Omega}^{H}  \in C^{\infty}(B^*H)$ is a
smooth density and where $F_2(\lambda)$ is a semi-classical
Fourier integral operator which quantizes the transmission map
(\ref{TRANS}). Thus, one of the main contributions of the present
article is the analysis of the matrix elements of $F_2(\lambda)$.

The proof of Theorem \ref{CDTHM} is then rather simple: The
Neumann data produces a second Fourier integral operator
$F_2^{\nu}(\lambda)$. The key observation is that the composition
of $F_2(\lambda)$ on the left by an operator with symbol $1 -
|\eta|^2$   cancels $F_2^{\nu}(\lambda)$ in the sum of the matrix
elements with respect to both the Dirichlet and Neumann data.
Thus, the Fourier integral operator terms cancel,  leaving only a
pseudo-differential matrix element whose limit was calculated in
\cite{HZ}.

For the Dirichlet data alone, there is nothing to cancel
$F_2(\lambda)$. Therefore, an important  step of the proof of
Theorem \ref{maintheorem} is to determine the weak* limits of the
matrix elements $\langle F_2(\lambda) u_{\lambda}^b, u_{\lambda}^b
\rangle. $  This is a special case of the more general problem of
determining limits of matrix elements of Fourier integral
operators relative to eigenfunctions, which is discussed in
\cite{Z2} and will be investigated more thoroughly in \cite{TZ3}.
A significant  additional complication is that  the matrix
elements are relative to  boundary restrictions of eigenfunctions
rather than to the eigenfunctions themselves.

In the case of pseudo-differential matrix elements, the first step
in obtaining the limits of matrix elements relative to $\Delta$-
eigenfunctions  is to show that they are invariant measures for
the billiard flow. In the case of matrix elements of Fourier
integral operators,  it is not clear  what kind of invariance
properties the limit has, viewed as a functional of the symbol of
$F_2(\lambda)$, when the canonical relation of $F$ fails to be
invariant under the billiard flow. Under the condition of
Definition (\ref{ANC}),  the canonical relation is, in fact, almost
nowhere  invariant under the billiard flow. We turn this to our
advantage by using a rather surprising averaging argument to show
that the limits must be zero under the almost nowhere commuting
condition.

The starting point of this argument is  the classical  boundary integral equation
\begin{equation} \label{boundarylayerb}
   u_{\lambda}^{b}(q) = \int_{\partial \Omega} N^{\partial \Omega}_{\partial \Omega}(\lambda)(q,q') \,  u_{\lambda}^{b}(q') d\sigma(q')
   \end{equation}
   for  the Dirichlet   data of $\phi_{\lambda}$ on $\partial
   \Omega$. The operator $N^{\partial \Omega}_{\partial
   \Omega}(\lambda)$ is studied in detail in \cite{HZ} as a
   semi-classical Fourier integral operator quantizing the
   billiard map $\beta: B^* \partial \Omega \to B^*\partial \Omega$.
   The classical formula,  $N^{\partial \Omega}_{\partial \Omega}(\lambda) u_{\lambda}^b = u_{\lambda}^b$
    was the key input to the quantum ergodicity result of \cite{HZ}. In particular, it shows
    that when $H = \partial \Omega$, the weak* limits of
    (\ref{ME}) are invariant under $\beta$.

    We apply this quantum invariance to  the $F_2(\lambda)$ term.
For all $m$, we have
   \begin{equation} \label{CTOBa} \langle Op_{\lambda}(a) u_{\lambda}^{H},
    u_{\lambda}^{H} \rangle_{L^2(H)} = \langle N^{\partial \Omega}_{\partial \Omega}(\lambda)^m
     N^{H}_{\partial \Omega}(\lambda)^* Op_{\lambda}(a) N_{\partial \Omega}^{H}(\lambda) N^{\partial \Omega}_{\partial \Omega}(\lambda)^m
     u_{\lambda}^{b}, u_{\lambda}^{b} \rangle_{L^2(\partial \Omega)}. \end{equation}
At this point,  we  cannot apply Egorov's theorem to the
conjugation because of  the almost nowhere commuting condition in
Definition \ref{ANC}. Under that condition,  the canonical
relation underlying $N^{\partial \Omega}_{\partial
\Omega}(\lambda)^m
     N^{H}_{\partial \Omega}(\lambda)^* Op_{\lambda}(a) N_{\partial \Omega}^{H}(\lambda) N^{\partial \Omega}_{\partial
     \Omega}(\lambda)^m$
 changes with each $m$, and it appears that this  destroys any invariance
 property of the quantum limits.  What we show is that it forces
 the matrix element $\langle F_2(\lambda) u_{\lambda}^b,
 u_{\lambda}^b \rangle \to 0$ on average  as  $\lambda \rightarrow \infty$. Thus,
 $F_2(\lambda_j) u_{\lambda_j}^{b}$  is `almost orthogonal' to $u_{\lambda_j}^{b}$ under the almost nowhere commuting condition.

The proof is based on a special case of a novel boundary local
Weyl law for semi-classical Fourier integral operators.

  \begin{maintheo} \label{weyl2}
Let $\Omega \subset {\mathbb R}^{n}$ be a Euclidean domain with piecewise-smooth boundary and $u_{\lambda_{j}}^{b}
 = \phi_{\lambda_j}|_{\partial \Omega} $ (resp. $u_{\lambda_{j}}^{b} =   \frac{1}{\lambda_j} \partial_{\nu}\phi_{\lambda_j}|_{\partial \Omega} $)
  be the Neumann  (resp. Dirichlet)  eigenfunction boundary traces. Let $F(\lambda) : C^{\infty}(\partial \Omega) \rightarrow C^{\infty}(\partial  \Omega)$
  be a zeroth-order $\lambda$-FIO with canonical relation $\Gamma_{F(\lambda)}:= \text{graph}
   ( \kappa_{F})  \subset B^{*}(\partial \Omega) \times B^{*}(\partial \Omega)$ and let $$\Sigma_{F(\lambda)}:=    \bigcup_{m \in \mathbb{Z}} \{ (y,\eta) \in B^{*} \partial \Omega; \beta^{m} (y,\eta) =  \kappa_F(y,\eta) \}.$$
    Then, under the assumption that  $| \Sigma_{F(\lambda)} | = 0,$

   $$ \frac{1}{N(\lambda)} \sum_{\lambda_j \leq \lambda} \langle F(\lambda_j) u_{\lambda_j}^{b}, u_{\lambda_j}^{b} \rangle = o(1).$$
 \end{maintheo}

In \cite{TZ3}, more general local Weyl laws for semi-classical
Fourier integral operators are proved. The general {\it boundary
local Weyl law}  is that there exists a density $\rho_F \in
C^{\infty}( B^*
\partial \Omega)$ such that
 \begin{equation} \label{weyl lemma}
\lim_{\lambda \rightarrow \infty}\frac{1}{N(\lambda)} \sum_{j:
\lambda_j \leq \lambda} \langle F(\lambda_j) u_{\lambda_j}^{b},
u_{\lambda_j}^{b} \rangle = \int_{\Sigma_{F(\lambda)} }
  \sigma(F(\lambda))(y,\eta)  \,   \cdot  \rho_F(y,\eta)  \,  dy d\eta. \end{equation}
    Moreover, $\rho_{F} dy d\eta$ can be explicitly computed as a sum of
     certain measures associated with the various multi-link contributions to the fixed-point set, $\Sigma_{F(\lambda)}.$  The
proof is valid for  arbitrary Riemannian manifolds with
piecewise-smooth
   boundaries. We present only the narrow version above since it is the only statement we need
   for Theorem \ref{maintheorem}
   and since  the
   proof of the general result is  lengthy.  Indeed,
   the proof of the special case in Theorem \ref{weyl2} is already
   the lengthiest part of the proof of Theorem \ref{maintheorem}  due to the complicated nature of the wave
   group on a domain with boundary.

Using Theorem \ref{weyl2}, we reduce Theorem \ref{maintheorem} to
the following {\it almost orthogonality} theorem, which has an
independent interest.

\begin{maintheo} \label{VARIANCE} With the same notation as in Theorem \ref{weyl2}, assume that
$F(\lambda)$ is a zeroth order $\lambda$-FIO and assume that its
underlying symplectic map  $\kappa_F$ almost nowhere commutes with
$\beta$ in the sense of Definition \ref{ANC}.  Then
$$ \frac{1}{N(\lambda)} \sum_{\lambda_j \leq \lambda}
   | \langle F(\lambda)
    u_{\lambda_j}^{b}, u_{\lambda_j}^b \rangle  |^{2} = o(1) \,\,\, \text{as} \,\, \lambda \rightarrow
    \infty.$$
    \end{maintheo}

In Theorem \ref{MEFIO} we restate the result in the form we need
for the proof of Theorem \ref{maintheorem}. However, the proof is
valid with no essential modification in  the general case.
     An interesting point is that the result  does not
use ergodicity of the billiard map. It is a completely general
almost orthogonality result for pairs of operators satisfying the
almost nowhere commutativity condition.

    The proof of Theorem \ref{VARIANCE} is based on a somewhat curious
    averaging argument.   By quantum invariance, we may (and do) replace
  $ F(\lambda)$
 by the operator
 \begin{equation} \label{AVE} \langle F(\lambda) \rangle_M: =   \frac{1}{M} \sum_{m = 0}^M [N^{\partial \Omega}_{\partial \Omega}(\lambda)^{*}]^m
 F(\lambda) [N^{\partial \Omega}_{\partial \Omega}(\lambda)]^m. \end{equation}
 Then, by the  Schwartz inequality,
 \begin{equation} \label{schwartz}
 \frac{1}{N(\lambda)} \sum_{\lambda_j \leq \lambda}
   | \langle F(\lambda_j)
    u_{\lambda_j}^{b}, u_{\lambda_j}^b \rangle  |^{2} = \frac{1}{N(\lambda)} \sum_{\lambda_j \leq \lambda}
   | \langle \langle F(\lambda_j) \rangle_M
    u_{\lambda_j}^{b}, u_{\lambda_j}^b \rangle  |^{2} \leq \frac{1}{N(\lambda)} \langle F_{M}(\lambda_j) u_{\lambda_j}^b, u_{\lambda_j}^b \rangle, \end{equation}
where,
 \begin{equation} \label{N*N}  F_{M}(\lambda):= \langle F(\lambda)\rangle_M^* \langle F(\lambda)\rangle_M. \end{equation}
 Under the almost nowhere commuting assumption, these operators $F_M(\lambda)$ satisfy
 the assumptions of Theorem \ref{weyl2}.
 The almost nowhere commuting condition (and its quantitative
 refinements) arises because the fixed point
  set of the canonical relation of (\ref{N*N}) is
 \begin{equation} \label{measure1a}
  \Sigma_{F_M(\lambda)} = \bigcup_{k=-M}^{M}   \bigcup_{p \in \Z} \{ \zeta \in B^{*} \partial \Omega; \,\, \beta^{p}
   \beta_{\partial \Omega}^{H} (\zeta) = \beta_{\partial \Omega}^{H}   \beta^{k} (\zeta) \}  \end{equation}
  and the quantitative almost-nowhere commutativity condition  (see Definition \ref{ANCQ})
  implies that
 \begin{equation} \label{true}
 \limsup_{M \to \infty} \frac{1}{M} | \Sigma_{F_M(\lambda)} | = 0. \end{equation}
 This is sufficient to imply  that the
 average over eigenvalues $\lambda_j$ of (\ref{N*N}) is $O(\frac{1}{M})$ and hence  the almost orthogonality statement in
 Theorem \ref{VARIANCE}.

Quite a lot
    might be lost in the use of the Schwartz inequality in (\ref{N*N}). As remarked in \S \ref{VARPROOFSECT},
    the almost nowhere commutativity condition is the condition under which averages over the spectrum of the norms $ ||
     \langle F_{2k}(\lambda_j;a,\epsilon) \rangle_{M}  \,
    u_{\lambda_j}^{b}||  $ tend to zero. It is another curious
    feature that we obtain a `subsequence of density one' along
    which the almost orthogonality holds. It is not clear whether (as
    for quantum ergodicity) this is an
    essential feature or a defect of the proof.

\subsection{\label{FP} Final introductory remarks}

We close the introduction by posing the quantum ergodic
restriction problem in  a more general context,  and by describing
some continuing work on the general problem.

The quantum ergodic restriction problem  concerns the linear
functionals (\ref{ME}) defined on the space $\Psi^0_{sc}(H)$ of
zeroth order semi-classical pseudo-differential operators on $H$.
As we have seen, there may exist large subsequences of
eigenfunctions which restrict to zero; on the other hand, the
known  restriction estimates (see \cite{BGT,KTZ}) do not give
uniform bounds on the $L^2(H)$ norms of the restrictions. Hence
the family $\{\rho_j^H\}$ (\ref{ME}) of functionals corresponding
to Dirichlet (or Neumann) data of $\Delta$- eigenfunctions along
$H$ is not apriori a bounded family of functionals, hence not
necessarily compact in the weak* topology. Nevertheless, we may
pose the
\medskip

\begin{mainprob}\label{W*} Let $H \subset M$ be a hypersurface.  Determine  the set   ${\mathcal Q}_H$ of `restricted quantum limits', i.e.
weak* limit points of the sequence of  the (un-normalized)
semi-positive states $\{\rho_k^H\}$ corresponding to the Dirichlet
(or Neumann) data of eigenfunctions on $H$. Similarly, determine
the restricted quantum limits for the functionals defined (as in
Theorem \ref{CDTHM})  by the full Cauchy data (\ref{CAUCHYDATA}).
\end{mainprob}

The analogue of an interior quantum ergodic sequence of
eigenfunctions is given by:

\begin{maindefn} The pair $(M, H)$ satisfies QER (quantum ergodic
restriction) for Dirichlet data  if  ${\mathcal Q}_H$ contains a
limit state $\omega_{\infty}^{H}$ so that a full density
subsequence of $\rho_j^H \to \omega_{\infty}^H$ in the weak*
sense. Similarly for Neumann and full Cauchy data.
\end{maindefn}

In \cite{TZ2}, we prove a quantum ergodic restriction theorem for
a general Riemannian manifold $(M, g)$ without boundary. In this
setting, there do not exist transfer maps and broken billiard
maps, and we base the proof on the almost nowhere equality of the
left and right return maps to $H$. We believe the condition is
satisfied for closed geodesics of negatively curved surfaces which
are not fixed by isometric involutions.

 In \cite{CHTZ} we use somewhat different methods from this
article on manifolds with or without boundary to prove QER for the
Cauchy data when the billiard flow is ergodic.  We use two
methods: one  method (close to that of this article)  to study the
(matrix) Calderon projector $C(\lambda): \phi_{\lambda} \to
(\phi_{\lambda}|_H,
\partial_{\nu} \phi_{\lambda}|_H)$ as a semi-classical Fourier integral operator
quantizing a certain first return map to $S_H^* M$ (the
 unit cotangent bundle of $M$ along $H$). Our second method is a
 modification of the approach
 of \cite{Bu,GL} in directly relating interior and restricted quantum ergodicity.

In further work in progress \cite{CHTZ2} we are studying a more
general
 microlocal setting in which $H$ (more precisely, $S_H^* M$) is replaced by a general
 hypersurface  $T \subset S^*M$ which is (roughly speaking)  a
cross-section to the geodesic flow. Quantum ergodicity of Cauchy
data should hold if the first return map for $S^*_H M$ is ergodic.
The relevant notion of Cauchy data is provided by a microlocal
restriction operator to $T$ that arises in Grushin reductions. It
is then an interesting question to relate this abstract reduction
and the conditions for QER to those obtained in this article and
in \cite{CHTZ} for the cross section defined by $S^*_H \Omega$,
the unit cotangent vectors with footpoints on $H$.
\medskip

\noindent{\bf Acknowledgements} We thank H. Christianson and H.
Hezari for comments and collaboration on related work. We also
thank  A. Strohmaier and S. Nonnenmacher for helpful comments on
the averaging argument for Fourier integral operators. Thanks also
to D. Hejhal for comments on the relation of our results to those
of \cite{HR},  and to  Peter Sarnak for comments on QUE and 
our  Corollary \ref{nodalthm}.

\section{\label{BILLIARDS} Classical billiards  on a domain with corners}

This section is a compilation of rather standard
 facts and
definitions regarding billiards on domains with corners. The
discussion does not involve $H$ or any objects specific to our QER
problem. However, we go over the basic definitions in some detail
for the sake of clarity and to introduce notation that will be
used in the rest of the article.

We assume throughout that  $\Omega \subset \R^m$ is a smooth domain  with corners. We define a smooth domain
with corners and with $M$ boundary faces (hypersurfaces) to be a
set of the form $\{x \in \R^m: \rho_j \leq 0, j = 1, \dots, M\}$,
where the defining functions $\rho_j$ are smooth in a
neighborhood of $\Omega$, and  are independent in the following
sense: at each $p$ such that $\rho_j(p) = 0$ for some finite
subset $j \in I_p$ of indices, the differentials $d \rho_j$ are
linearly independent for all $j \in I_p$. A boundary hypersurface
$H_j$ is the intersection of $\Omega$ with one of the
hypersurfaces $\{\rho_j = 0\}$. The boundary faces of codimension
$k$ are the components of $\rho_{j_1} = \cdots = \rho_{j_k} = 0$
for some subset $\{j_1, \dots, j_k\}$ of the indices; each is a
manifold with corners. This definition is extrinsic in the sense
of \cite{M} and is not the most general one possible.

It is essential to allow $\Omega$ to have corners since domains
with ergodic billiard flow in $\R^n$ have corners. The
singularities play a very significant role in the dynamics.  They
occur at corners, at trajectories which become tangent to the
boundary, and in our setting of domains with obstacles, at
trajectories which run tangent to the obstacle.

We denote the smooth   part of $\partial
\Omega$ by $(\partial \Omega)^o  $.  Here, and throughout this
article, we denote by $W^o$ the interior of a set $W$ and, when no
confusion is possible, we also use it to denote the regular set of
$\partial \Omega$. Thus, $\partial \Omega = (\partial \Omega)^o
\cup \Sigma$, where $\Sigma = \bigcup_{i \neq j} (H_i \cap H_j)$
is the singular set. When $\dim \Omega = 2,$ the singular set is a
finite set of points and the $H_i$ are smooth curves. In
higher dimensions, the $H_i$ are smooth hypersurfaces; $H_i
\cap H_j$  is a stratified smooth space of co-dimension one, and
in particular $\Sigma$ is of measure zero. We denote by
$S^*_{\Sigma}\Omega$ the set of unit vectors to $\Omega$ based at
points of $\Sigma.$
 We also
define $C^\infty(\partial \Omega)$ to be the restriction of
$C^\infty(\R^n)$ to $\partial \Omega$.

We define the open unit ball bundle $B^* (\partial \Omega)^o$ to
be the projection to $T^* \partial \Omega$ of the inward pointing
unit vectors to $\Omega$ along $(\partial \Omega)^o$. We leave it
undefined at the singular points.

To simplify the exposition, we describe  $B^* \partial \Omega$ as
if it consisted of tangent vectors rather than cotangent vectors,
using the Euclidean metric to identify the two. We denote the
standard projection by $\pi: B^* \partial \Omega \to  \partial \Omega$.

\subsection{\label{BMAP} Billiard map}

At smooth points of  $\partial \Omega$,   the billiard map $\beta$
is defined on $B^*\partial \Omega$ as follows: given $(y, \eta)
\in B^*
\partial \Omega$, i.e.  with $|\eta| < 1$ we let $(y, \zeta) \in S^*
\Omega$ be the unique inward-pointing unit covector at $y$ which
projects to $(y, \eta)$ under the map $T^*_{\partial \Omega}
\overline{\Omega} \to T^* \partial \Omega$ (see (\ref{zetayeta})).
 Then we follow the geodesic (straight line) \begin{equation} \label{EDEF} (y, \eta) \in B^* (\partial \Omega)^o \to E(t, y, \eta) :
  =  q(y) + t \zeta(y, \eta) \in \Omega \subset \mathbb{R}^{n}\end{equation} to the first place it
intersects the boundary again; let $y' \in
\partial \Omega$ denote this first intersection. Thus, we find the
least $t = t(y, \eta)$ so that $\rho_j(E(t, y, \eta)) = 0$ for
some $j$ and let $y': = E(y, \eta)$ solve the equation $$q(y') =
E(t(y, \eta), y, \eta).$$ The notation $E(t, y, \eta)$ is intended
to suggest `exponential map', i.e. the position in
$\overline{\Omega}$ reached by the billiard trajectory at time
$t$.

 Denoting the inward unit normal vector at $y'$ by $\nu_{y'}$,
we let $\zeta' = \zeta + 2 (\zeta \cdot \nu_{y'}) \nu_{y'}$ be the
direction of the geodesic after elastic reflection at $y'$, and
let $\eta'$ be the projection of $\eta'$ to $T^*_{y'}\partial
\Omega$. Then we define explicitly,  \begin{equation} \label{BETA}
\beta(y, \eta) = (y', \eta'). \end{equation}

We call each (linear) segment of a billiard trajectory a `link',
so that $\beta$ relates the initial and terminal directions
 of a link.  We   denote the first link of the
billiard trajectory determined by $(y, \eta) \in B^* \partial
\Omega$ by \begin{equation} \label{LINK} \overline{(y, \eta)}.
\end{equation}

We also denote by \begin{equation} \label{GT} G^{\tau}:
S^{*}\Omega
  \rightarrow S^*\Omega \end{equation}  the (partially defined)  broken geodesic flow
  in the full domain
  $S^*\Omega$. The trajectory $G^{\tau}(x, \xi)$ consists of the straight line motion
  between impacts with the boundary, and by the usual reflection
  law at the boundary. Thus, $\beta$ is the reduction to the
  boundary of $G^{\tau}$. We refer to \cite{PS,SV} for background.

We now address the complications caused by tangent directions and
singular points.  The billiard map is not well-defined on initial
directions which are tangent to $\partial \Omega$, i.e. $(y, \eta)
\in S^*
\partial \Omega$ with $|\eta| = 1$.   If the domain  is convex at $y$,
then the line $q(y) + t \eta$ immediately exits the domain.   On the
other hand, if it is concave at $y$, then the trajectory remains
in the interior of $\Omega$ and defines a link. If the domain is
 a `saddle surface' at $y$, then some tangential links stay in
$\Omega$ and others exit. We introduce some  terminology to
distinguish these cases.

 \begin{defn} \label{UNIT} Let $(y, \eta) \in S^* \partial
 \Omega$ be a unit tangent vector.

 \begin{enumerate}

 \item We say that $(y, \eta)$
  is an exit direction, if the oriented ray $q(y) + t \eta$
 exits $\Omega$ for $t > 0$. We denote the set of exit directions in $B^*_y \partial \Omega$  by $\ecal(y)$ and define $\beta(y, \eta) = (y, \eta)$ for such
 exit directions.

 \item We say that $(y, \eta)$ is an interior tangent direction, if $q(y) + t \eta$
 remains in $\Omega$ for sufficiently small $t > 0$. We define
 $\beta(y, \eta) = (y, \eta')$ by following $q(y) +  t \eta$ for $t > 0$ until
 its first point of intersection $y' \in \partial \Omega$ and
 let $\eta'$ be the tangential projection of $\eta$ at $y'$.

 \end{enumerate}

 \end{defn}

It is common to puncture out all tangent directions from $B^*
\partial \Omega$, i.e. all $\eta \in S^* \partial \Omega$ with $|\eta| = 1$ and only
define billiard maps on the open unit ball. But it makes sense to
retain unit tangent directions as long as the straight-line
trajectory remains well defined in $\Omega$. Unit tangent vectors
to $\partial \Omega$ at convex points can never be terminal
velocity vectors of incoming trajectories and are therefore
irrelevant to the definition of the billiard flow. However, when
$\dim \Omega \geq 3$,  as in the case where the boundary is a
monkey saddle, it is possible for an incoming trajectory to arrive
at an exit direction and terminate there.

We will also need to distinguish initial directions whose
trajectories terminate in tangent vectors to $\partial \Omega$.

 \begin{defn} \label{GRAZINGb}  Let $\dim \Omega \geq 3$. We  denote by  $\gcal(y) \subset B_y^* (\partial \Omega)^o$ the set of grazing directions
$\eta $ such that $\beta(y, \eta) \in S^* \partial \Omega$, i.e.
such that the terminal vector of the link is tangent to $\partial
\Omega$.

\end{defn}

No such exit directions exist when $\dim \Omega = 2$.  We now
consider the regularity properties of $\beta$ and the impact
of the corner set.
\medskip

  \begin{defn} \label{singular} We denote by  $B^*_{\Sigma} \partial \Omega$
  or equivalently by $S^*_{\Sigma} \Omega$ the set of unit vectors to $\Omega$ at
  points of the singular set. We define
$$\rcal^1 = B^* \left( (\partial \Omega)^o \backslash  E^{-1} (\Sigma) \right) \backslash \gcal,
$$ i.e.  the set of smooth directions of links whose next  collision with $\partial \Omega$ is neither at
 $\Sigma$ nor tangent to $\partial \Omega$.
\end{defn}

\begin{lem} \label{RABETA} The domain of $\beta$ is $\rcal^1 \cup \gcal$. It is a  smooth symplectic map  on the interior of $\rcal^1$, which is smooth domain with corners in $B^*(\partial \Omega)$.
\end{lem}

\begin{proof} It is clear  from the definition   that $\beta$ is well-defined on $\rcal^1 \cup \gcal$ and on no
larger domain. It is standard that $\beta$ is symplectic where it
is defined; see \cite{CFS}. A simple way to prove it is that
$\beta$ has a generating function, namely the boundary distance
function $d_{\Omega}(y, y') = |q(y)  - q(y')|$ on $\partial \Omega
\times \partial \Omega$. We have punctured out $\gcal$ since
$\beta$ has jumps along this set, i.e. is discontinouous.

 Regarding regularity, it  is clear from the formulae (\ref{EDEF})-(\ref{BETA}) that  $E(t, y,
\eta)$ and $\beta(y, \eta)$  are smooth wherever $\nu_y$ is smooth and $t(y, \eta)$ is smooth.
Of course, $\nu_y$ is smooth away from the corners.
Clearly, $t(y, \eta)$ is  smooth as long as  $\nu_y$ is smooth,  $E(t(y, \eta),
y, \eta) \notin \Sigma $ and  as long as  $(y, \eta) \notin \gcal;$  at the latter points,  $t(y, \eta)$
is dis-continuous.  Hence,
$\beta(y, \eta)$ is smooth on $\rcal^1$ and on no larger set.

We observe that $E^{-1}(\Sigma) \subset B^* \partial \Omega$  is
locally defined as the simultaneous zero set of two smooth
functions. Indeed, it the union of the sets  $\{(y, \eta): E(y,
\eta) \in H_i \cap H_j\}$. We may extend the hypersurfaces $H_j$
beyond $\Omega$
 as the zero sets of the $\rho_j$ and consider the travel time $t_j$ from $(y, \eta)$
 to these hypersurfaces. I.e.  we may extend the time functions $t_j(y,
\eta)$ as the smallest value of $t$ so that $\rho_j(E(t, y, \eta))
= 0$. Then $E^{-1}(H_i \cap H_j)$ is defined by $t_i(y, \eta) t_j(y, \eta).$  Similarly for multiple intersections.

\end{proof}

\subsection{Iterates of the billiard map}

We now consider the inverse $\beta^{-1}$ of $\beta$ and their
iterates. We note that  the billiard flow is `time-reversal
invariant', i.e.  $\beta^{-1}(y, \eta) = - \beta(y, - \eta)$.
Hence $\beta^{-1}$ is well defined on the set $\rcal^{-1}$ of
directions in $B^*(\partial \Omega)^o$ such that $E(y, - \eta)
\notin \Sigma$. With a slight abuse of notation, we let
$\beta^{-1}(B^*_{\Sigma}) = E^{-1}(\Sigma)$, i.e. it is the set of
directions which map to the corner set; strictly speaking, there
we do not define the image directions there.

We now
  define open sets $\mathcal{R}^k$ such that $\mathcal{R}^1 \supset \mathcal{R}^2 \supset \dots$, $\mathcal{R}^{-1}
  \supset \mathcal{R}^{-2} \supset \dots$  where $\beta^{(k)}$ is defined, for any integer
  $k$.

\begin{defn}

 For $k \geq 1,$ we define the set $  \mathcal{R}^k$  of  $k$-regular
  directions such that the first $k$ links of its trajectory stays
  in $B^* (\partial \Omega)^o \backslash  \gcal$. For $k \leq -1$, we define it
  similarly for $\beta^{-1}$.

  The   coball singular set is the set   $\Sigma^*: = \bigcup_{k}
\beta^{-k} (B^*_{\Sigma} \partial\Omega)$ and the $N$-th order coball
singular set is $\Sigma^{*}_{N}:= \bigcup_{|k| \leq N} \beta^{-k}
(B^*_{\Sigma} \partial \Omega).$

\end{defn}

\begin{defn} \label{GRAZINGgeneralized}
 We define the {\em generalized grazing set} to be
 $$\gcal^* = \bigcup_{k \in \mathbb{Z}} \beta^{k} (S^*\partial \Omega).$$
 Similarily, the $N$-th  grazing set is denoted by
 $$ \gcal^*_{N} = \bigcup_{|k| \leq N} \beta^{k}(S^* \partial \Omega).$$
 \end{defn}

  By repeating the proof of Lemma \ref{RABETA},
  we have:

  \begin{lem} \label{RABETAK} The domain of $\beta^{k}$ is $\rcal^k$ It is a  smooth symplectic map  on the interior of
  $\rcal^{k}$.
\end{lem}

The regularity of $\beta^{k}$ near $(y, \eta)$ is essentially the
issue of the regularity of the  time $t^{(k)}(y, \eta)$ at which
the billiard trajectory defined by $(y, \eta)$ hits $\partial
\Omega$ for the kth time. It fails to be defined if the trajectory
runs into $\Sigma$ or becomes an exit direction before this time,
and additionally fails to  be smooth if the trajectory becomes
tangent to $\partial \Omega$ at time $t^{(k)}(y, \eta)$.

Since the corner set and its images and pre-images under $\beta$
is a countable union of stratified  smooth hypersurfaces,
 $\beta$ is  a measure-preserving map  on a full-measure subset of $B^*  (\partial \Omega)^o$.

\section{\label{billiard maps} Hypersurfaces $H$: classical and quantum transfer maps}

In this section, we define the  quantum transfer operators $
N_{\partial \Omega}^{H}(\lambda)$ (resp. $N_{\partial \Omega}^{H,
\nu}$)   (see (\ref{boundarylayer})), and show that they are
semi-classical Fourier integral operators associated to the
billiard relation given by transfer maps $\tau_{\partial
\Omega}^{H}$ (resp. $\tau_{H}^{\partial \Omega}$).  As discussed
in the introduction, the quantum transfer operators are
fundamental in our approach, which is based on transferring the
expected value $\langle Op_{\lambda}(a) u_{\lambda}^{H},
u_{\lambda}^{H} \rangle_{L^2(H)} $ to the boundary by the identity
(\ref{CTOB}).

The operator $N_{\partial \Omega}^{H}(\lambda):
C^{\infty}(\partial \Omega) \rightarrow C^{\infty}(H)$ is defined
as follows: By Green's theorem, if $\phi_{\lambda_j}$ is a Laplace
eigenfunction, then for any $x \in  \Omega^o,$
\begin{equation}\label{NBDYH} \phi_{\lambda_j} (x) = \int_{\partial \Omega}
\big[\partial_{\nu_{y}} G_0(x, q, \lambda_j)  u_{\lambda_j}^{b} (q)- G_0(x,
q,\lambda_j) \partial_{\nu_{q}} \phi_{\lambda_j}(q)
 \big] d\sigma(q),
\end{equation}
where
\begin{equation}
G_0(x,x',\lambda) =  \frac{i}{4} \lambda^{n-2} (2 \pi \lambda
|x-x'|)^{-(n-2)/2} {\bf Ha}_{n/2-1}(\lambda | x-x' |)
\label{Hankel}\end{equation} is the free outgoing Green function
on $\R^n$.  The Neumann (resp. Dirichlet) boundary conditions
eliminate the second (resp. first) term, and by restricting $z$ to the hypersurface in (\ref{NBDYH}),  one gets the boundary integral equation
$$u_{\lambda_j}^{H}= N_{\partial \Omega}^{H} (\lambda_j)  u_{\lambda_j}^{b}.$$
The transfer operator from boundary data to Dirichlet data along
$H$ is defined as follows:

\begin{defn} \label{NDEF} $N_{\partial \Omega}^{H}(\lambda):
C^{\infty}(\partial \Omega) \rightarrow C^{\infty}(H)$ is the
operator whose Schwartz kernel with respect to the hypersurface
volume element $d\sigma(q)$ on $(\partial \Omega)^{o}$
is given by,
\begin{equation}  N_{\partial \Omega}^{H} (q_{H}, q, \lambda_j) = \left\{ \begin{array}{ll}  \partial_{\nu_{q}} G_0(q_{H}, q, \lambda_j), & \mbox{Neumann}\\  \\
  G_0(q_{H}, q, \lambda_j), & \mbox{Dirichlet} \end{array} \right.
  \end{equation}
  \end{defn}

Similarly, the transfer operator from boundary data to Neumann
data along $H$ is given in:

\begin{defn} \label{NDEFnu} With the same notation and assumptions, $N_{\partial \Omega}^{H \nu}(\lambda):
C^{\infty}(\partial \Omega) \rightarrow C^{\infty}(H)$ is the
operator with  Schwartz kernel
\begin{equation}  N_{\partial \Omega}^{H \nu} (q_{H}, q, \lambda_j) = \left\{ \begin{array}{ll} \partial_{\nu_{q_H}}  \partial_{\nu_{q}} G_0(q_{H}, q, \lambda_j), &
 \mbox{Neumann}\\  \\
 \partial_{\nu_{q_H}}  G_0(q_{H}, q, \lambda_j), & \mbox{Dirichlet} \end{array} \right.
  \end{equation}
  \end{defn}

We now prove that $N_{\partial \Omega}^{H}
(\lambda_j), N_{\partial \Omega}^{H \nu} (\lambda_j)$ are
semi-classical Fourier integral operators (with singularities)
  quantizing the partially defined multi-valued transfer
map $\tau_{\partial \Omega}^{H}$ between  $B^* \partial \Omega$
and  $B^* H.$ This is the first place where the convexity of $H$
is relevant.

We recall that a semiclassical (ie. $\lambda$)-Fourier integral operator, $F(\lambda):C^{\infty}(\partial
\Omega) \rightarrow C^{\infty}(\partial \Omega),$ is an oscillatory
integral operator whose  Schwartz kernel can be written locally in
the form
\begin{equation} \label{oscillatory}
 F(y,y';\lambda) = (2\pi \lambda)^{n-1} \int_{\R^{n-1}} e^{i \lambda \Phi(q(y),q(y'),\theta)} A(q(y),q(y'),\theta,\lambda) \, d\theta \end{equation}
 where the amplitude is a semi-classical (polyhomogeneous) symbol in $\lambda$,  $A(q(y),q(y'),\theta;\lambda) \sim \sum_{k=0}^{\infty} A_{k}(q(y),q(y'),\theta) \lambda^{m-k},$ with  $A_k \in C^{\infty}_{0}(\partial
  \Omega \times  \partial \Omega \times \R^{n-1}), \, k=0,1,2,....$
  The associated canonical relation is defined by
  \begin{equation} \label{CANONICAL} \Gamma_F = \{(y, \nabla_y
  \Phi, y', - \nabla_{y'} \Phi : \nabla_{\theta} \Phi = 0\}
  \subset T^* \partial \Omega \times T^* \partial \Omega.
  \end{equation}

\begin{prop} \label{GAMMAC}  $N_{\partial \Omega}^{H} (\lambda_j),$  resp.  $\lambda_j^{-1}N_{\partial \Omega}^{H \nu} (\lambda_j)$, are semi-classical Fourier integral
operators of order zero  with canonical relation
(\ref{TRANSFERRELATION}). The principal symbol of $N_{\partial
\Omega}^H$ is the half density on the graph of $\tau_{\partial
\Omega}^H$ given at a point $ ((y,\eta), \tau_{\partial
\Omega}^H(y, \eta) ) \in \Gamma_{\tau_{\partial \Omega}^H}$ with
$(y, \eta) \in B^*
\partial \Omega$ by
$$ \sigma(N_{\partial \Omega}^{H}(\lambda))(y,\eta) = \Big( \frac{\gamma(y, \eta)}{\gamma(\tau_{\partial \Omega}^H (y, \eta))} \Big)^{1/2} \big| dy d\eta
\big|^{1/2}$$   in the symplectic coordinates $(y, \eta)$ of $B^*
\partial \Omega$.  Also,
$$ \sigma(\lambda^{-1} N_{\partial \Omega}^{H \nu }(\lambda))(y,\eta) = i \Big( \gamma(y, \eta)\gamma(\tau_{\partial \Omega}^H (y, \eta)) \Big)^{1/2} \big| dy d\eta
\big|^{1/2}.$$
 \end{prop}

 \begin{rem}  We observe that the symbol of $N_{\partial \Omega}^H(\lambda)$
 is infinite on the set where $\gamma(\tau_{\partial \Omega}^H (y,
 \eta))= 0$. This is also true in the case $H = \partial \Omega$
 in \cite{HZ}. We will cut off away  from this set in the proof of the
 main theorems  (see Proposition \ref{Egorov}). The  key  point  is that  for a quantum ergodic sequence,  eigenfunction mass
 cannot concentrate in this set  (see Lemma \ref{tanmass} in subsection \ref{extensions}).
 \end{rem}

 \noindent{\bf Proof}

 The proof of the Proposition consists of a sequence
of Lemmas. We first observe that the quantum transfer maps are
Fourier integral operators with phase equal to the distance
function $|q_H - q|$ with $q_H \in H, q \in \partial \Omega$. We
then discuss the corresponding canonical relation. In section
\ref{SYMBOLN} (see Lemma \ref{symbolcomputation}) we calculate the
principal symbols.

 Under our assumption that $H \subset \Omega^o$, we have  $|q_H - q| \geq dist (H,\partial \Omega) >0,$ and we
may use the asymptotics of the Hankel function  (\ref{Hankel}) to
obtain,
\begin{equation} \label{wkb}
G_{0}(q_H(s),q(y);\lambda_j) \sim_{\lambda_j \rightarrow \infty}  {(2\pi
\lambda_j)}^{\frac{n-3}{2}} e^{i \lambda_j |q_{H}(s)-q(y)|} \sum_{k=0}^{\infty}
b_{k}(q_H(s),q(y)) \lambda_j^{-k}
\end{equation}
with $b_k \in C^\infty(H \times \partial \Omega)$ and $b_0 >0$.
The expansions for $\partial_{\nu_y}G_0(q_H(s),q(y);\lambda_j)$ and of
 $\partial_{\nu_{y_H}} \partial_{\nu_y}G_0(q_H(s),q(y);\lambda_j)$ are
of the same form and  are computed by taking  the  normal
derivatives of the expansion in (\ref{wkb}). Thus, the operators
in Definitions (\ref{NDEF})-(\ref{NDEFnu}) are   semi-classical
Fourier integral operators with phase equal to the distance
 function $|q_H(s) - q(y)|$ between a point of $H$ and a point of
 $\partial \Omega.$ The canonical relation is given by
 \begin{equation} \label{GAMMA}  WF_{\lambda}' (  N_{\partial \Omega}^{H} ) \subset \{ (s, \pi^T_{q_H(s)}  r(q_H(s),q(y)),  y, \pi^T_{q(y)} r(q_H(s),q(y))) \} \subset B^*H \times B^* \partial \Omega, \end{equation}
 where,
 $$ r(q_H,q):= \frac{q_H - q}{|q_H-q|}.$$
 The canonical relation in (\ref{GAMMA}) is easily seen to be the same as (\ref{TRANSFERRELATION}).
\medskip

\begin{rem} This statement is simpler than the one for the
boundary integral operator $N_{\partial \Omega}^{\partial
\Omega}(\lambda)$ (\ref{boundarylayerb}) in \cite{HZ}.
 In that case, the
phase and amplitude are singular on the diagonal and have to be
regularized. Here, the operator kernel is smooth away from the
corners of $\partial \Omega$. The singularities at the corners are
only a technical inconvenience in the quantum ergodic theorems,
and we refer to \cite{HZ} for a more detailed discussion of them.
\end{rem}

\subsection{\label{BMAPSH} Details on the transfer map $\tau^{H}_{\partial \Omega}$}

We now describe in more detail the canonical relation
(\ref{GAMMA}) and the  transfer map $\tau^H_{\partial \Omega}$ of
Definition \ref{TRANSFERMAP}. As mentioned in the introduction,
the definition of this map involves several technicalities:

\begin{itemize}

\item The transfer map $\tau_{\partial \Omega}^H$ is only
partially defined on $B^* \partial \Omega$, namely on the range of
the adjoint transfer map $\tau_H^{\partial \Omega}: B^* H \to B^*
\partial \Omega$.

\item $\tau_{\partial \Omega}^H$ is multi-valued since the
trajectory $\overline{(y, \eta)}$ defined by $(y, \eta) \in B^*
\partial \Omega$ may intersect $H$ generically in multiple points.

\item It has singularities when the trajectory  $\overline{(y,
\eta)}$  from $(y, \eta) \in B^* \partial \Omega$ is tangent to $H$
or when the trajectory is tangent to $\partial \Omega$. Such a
trajectory may go on to intersect $H$ again.

\item The adjoint transfer map $\tau_H^{\partial \Omega}: B^* H
\to B^*
\partial \Omega$ is singular when a ray from $B^*H$ runs into a
corner.

\end{itemize}

We now introduce notation to deal with these issues. At this
point, we assume that $H$ is convex. With this assumption, a
generic line intersects $H$ at most twice (indeed, it only
intersects it more than twice if in intersects it in a flat part).

\begin{defn} \label{HBILLDEF} The domain $\ccal$ of $\tau_{\partial \Omega}^{H}$
is the set of $(y, - \eta) \in B^* (\partial \Omega)^o$ such that
the link $\overline{(y, - \eta)}$ (i.e. $E(t, y, - \eta)$)
intersects $H$ (see (\ref{LINK})). Let $\ccal^o \subset \ccal$
denote the open dense subset such that $\overline{(y, -\eta)} \cap
H$ consists of two distinct points $\{q_H^1(y, -\eta), q_H^2(y,
-\eta)\}$, ordered so that
$$| q_H^1(y, -\eta) - y| < |q_H^2(y, -\eta) - y|, $$
i.e. so that the line intersects $H$ first at $q_H^1(y,- \eta)$
and second at $q_H^2(y, -\eta)$.

Let  $\tcal \subset \ccal$   be the set of directions $\eta$ such
that $\overline{(y, -\eta)}$  is tangent to $H$ at its  points of
intersection.

If  $(y, -\eta) \in \ccal^o = \ccal \backslash \tcal,$  let $t_H^1
(y, \eta) < t_H^2 (y, \eta)$ be the times of intersection with
$H$.

\end{defn}

The reason for the minus sign in $- \eta$ is that the transfer map
follows the trajectory defined by $(y, \eta)$ in the backwards
direction (see (\ref{TRANS}). The time reversal is the map $\eta \to - \eta$.  Since
$H$ is smooth, the time functions $t_j^H (y, \eta)$ are smooth
functions of their arguments in $\ccal^o$.

\subsection{Formulae for the branches of the transfer map} \label{transferformulas}

In this section, we give explicit formulae for the branches of the
transfer map away from $\tcal$. They are  needed in \S
\ref{fixedpoint} in the proof of the genericity of the condition
on $H$.

We define branches $\tau_{\partial \Omega}^{H;1}(y, \eta)$ and
$\tau_{\partial \Omega}^{H;2}(y, \eta)$ of $\tau_{\partial
\Omega}^H$ for $(y, \eta) \notin \tcal$ corresponding to the two
intersection points with $H$. For $(y, \eta) \in \ccal^o$, we
define the once $H$-broken trajectories $E_{H,j}(t, y, \eta)$ with
breaks at $q_H^j(y, \eta)$ and intersection times $t^j_H(y, \eta)$
by
$$E_{H,j}(t, y, \eta): = \left\{ \begin{array}{ll} E(t, y, - \eta), & t \leq t_H^j(y,
\eta); \\ & \\
E(t_H^j (y, \eta), y, \eta)   + (t - t_H^j (y, \eta)) \zeta_H^j(y,
- \eta), & t \geq t_H^j (y, \eta) \end{array} \right. $$ where
\begin{equation} \label{ZETAHJ} \zeta_H^j(y, - \eta) = - \zeta(y, - \eta) + 2 (\zeta \cdot
\nu_{q_H^j(y,  \eta)}) \nu_{q_H^j(y,  \eta)}. \end{equation} As
above, the $- \eta$ is needed since the trajectory proceeds
backwards from the initial conditions. The unit normal is
understood to be on the $+$ side of $N^*H$.

After the break, the  trajectory  proceeds on the straight line
$E(t_H^j (y, \eta), y, \eta) + (t - t_H^j(y, \eta)) \zeta^j_{H}(y,
\eta)$ until it intersects the boundary again at a point $q_{H,
\partial \Omega}^j(y, \eta)$. We denote by $t_{H,
\partial \Omega}^j(y, \eta)$ the time at which this occurs. The
following is a more detailed version of Definition
\ref{TRANSFERMAP}.

\begin{defn} \label{SECONDTRANSFER} Define (see (\ref{ZETAHJ})) $q_{H,\partial \Omega}^j(y,\eta)$ to be the solution of the equation
$$q_H  ( q_{H, \partial \Omega}^j(y, \eta)  )=  E(t_H^j (y, \eta), y, \eta) +
(t_{H,
\partial \Omega}^j(y, \eta) - t_H^j
(y, \eta)) \zeta_H^j(y, \eta). $$ If $q_{H, \partial \Omega}^j(y,
\eta) \notin \Sigma,$ put
$$\begin{array}{lll} \tau_{\partial
\Omega}^{H;j }(y, \eta) & = & (q_{H, \partial \Omega}^j(y, \eta),
\pi_{T(q_{H, \partial \Omega}^j(y, \eta))} r( q_H (q_{H, \partial
\Omega}^j(y, \eta) ), E(t_H^j(y, \eta), y, - \eta)), \,\,\,
 \end{array} $$

 Also define the correspondence
$$\tau_{\partial
\Omega}^{H }(y, \eta)  = \{\tau_{\partial \Omega}^{H; 1 }(y,
\eta), \tau_{\partial \Omega}^{H;2 }(y, \eta)\}. $$

\end{defn}

 In the
notation of \cite{Z}, this double-valued correspondence would be
written in the additive notation,
\begin{equation} \label{SUMNOT} \tau_{\partial \Omega}^{H }(y, \eta)  = \tau_{\partial
\Omega}^{H; 1 }(y, \eta) +  \tau_{\partial \Omega}^{H;2 }(y,
\eta).
\end{equation}

 If $(y, \eta) \in \tcal$, then its trajectory  runs tangent to $H$, and  it proceeds on a
straight line to $\partial \Omega.$ We then define  $
\tau_{\partial \Omega}^{H;1} (y, \eta) = \tau_{\partial
\Omega}^{H;2} (y, \eta)= \beta(y, \eta),$ the usual billiard map.
Thus, the graph  $\Gamma_{\tau_{\partial \Omega}^{H; j } } \subset
B^*
\partial \Omega \times B^* \partial \Omega$   of the
indicated partially defined symplectic map is given by
\begin{equation} \label{GAMMAS} \Gamma_{ \tau_{\partial \Omega}^{H}} =  \Gamma_{\tau_{\partial \Omega}^{H; 1 } } \cup
\Gamma_{\tau_{\partial \Omega}^{H; 2 } }, \end{equation} where the
two graphs on the right intersect along their common values for
$(y, \eta) \in \tcal$.

As in the discussion of $\beta$, we note that $ \tau_{\partial
\Omega}^{H;1}$ and $ \tau_{\partial \Omega}^{H;2}$ are piecewise
smooth maps. It follows easily that the domain of definition of  $ \tau_{\partial
\Omega}^{H;j}$ is  $\ccal \backslash E_{H,j}^{-1}(\Sigma)$  and each map
is smooth on the interior of its domain.

In the case of more general hypersurfaces, billiard trajectories
from $\partial \Omega$ may intersect $H$ in more points, and  may
intersect both tangentially and non-tangentially. As mentioned
above, it  is unnecessary to deal with such complications since
the QER theorem may be proved for symbols $a $ with arbitrarily
small microsupport on $B^* H$.

\subsection{\label{SYMBOLN} Symbols of $N_{\partial \Omega}^H$ and $N_{\partial \Omega}^{H \nu}$}

We first recall the definition of the symbol of a semi-classical
Fourier integral operator $F(\lambda) \in I^0(M \times M, C)$ where $\lambda \geq \lambda_{0}$ is the inverse semiclassical parameter. It is a
smooth section of the bundle $\Omega_{1/2} \otimes L$ of
half-densities on $C$ and the Maslov line bundle.

In a local Fourier integral representation $F(\lambda)(x, y) = (2\pi \lambda)^{-n} \int_{ {\mathbb R}^{n} } e^{- i
\lambda \phi(x, y, \xi)} a(x, y, \xi, \lambda) d\xi$ with $a(x,y,\xi,\lambda) \sim_{\lambda \rightarrow \infty} \sum_{j=0}^{\infty} a_j(x,y,\xi) \lambda^{-j}, \,\, a_j(x,y,\xi) \in S_{1,0}^{-j},$  the symbol is
transported from the critical set  $$C_{\phi} = \{(x, y, \xi):
d_{\xi} \phi (x, y, \xi) = 0\} \subset M \times M \times \R^n$$
 by the Lagrange immersion $i_{\phi}$
$$i_{\phi}: C_{\phi} \to C, \;\; (x, y, \xi) \to (x, \phi'_x, y, -
\phi'_y). $$ On $C_{\phi}$ we form the   Leray density
$$\begin{array}{l} d_{C_{\phi}}  = \frac{dx \wedge dy \wedge d \xi}{d (\frac{\partial \phi}{\partial \xi})}
\end{array}$$
At  a point $(x_0, \xi_0, y_0, \eta_0) \in C$, the symbol of $F$
is defined as
$$\sigma_F (x_0, \xi_0, y_0, \eta_0) = i_{\phi *} a_0
\sqrt{d_{C_{\phi}}}. $$

\begin{lem}  \label{symbolcomputation}The principal symbol of $N_{\partial \Omega}^H$ is the
half density on the graph of $\tau_{\partial \Omega}^H$ given at a
point $ ((y,\eta), \tau_{\partial \Omega}^H(y, \eta) ) \in
\Gamma_{\tau_{\partial \Omega}^H}$ with $(y, \eta) \in B^*
\partial \Omega$ by
$$ \sigma(N_{\partial \Omega}^{H}(\lambda))(y,\eta) = \Big( \frac{\gamma(y, \eta))}{\gamma(\tau_{\partial \Omega}^H (y, \eta)} \Big)^{1/2} \big| dy d\eta
\big|^{1/2}$$   in the symplectic coordinates $(y, \eta)$ of $B^*
\partial \Omega$.

Also,
$$ \sigma(\lambda^{-1} N_{\partial \Omega}^{H \nu }(\lambda))(y,\eta) = i \Big( \gamma(y, \eta))\gamma(\tau_{\partial \Omega}^H (y, \eta) \Big)^{1/2} \big| dy d\eta
\big|^{1/2}.$$

\end{lem}

\begin{proof}

The calculation of the symbol of $N_{\partial \Omega}^H(\lambda)$
is similar to that of $N_{\partial \Omega}^{\partial
\Omega}(\lambda)$ in Proposition 6.1 of  \cite{HZ}, and we follow
that calculation closely, referring there for some of the details
and emphasizing only the steps where something new occurs.  From
the  asymptotics of the Hankel function,
$$
 {\bf Ha}_{n/2 - 1}(t) \sim e^{-i(n-1)\pi/4} e^{it} \sum_{j=0}^\infty
a_j t^{-1/2 - j},
$$
 the principal symbol of $N_{\partial
\Omega}^H (\lambda)$ at the billiard Lagrangian  is the same (up
to an eighth root of unity) as that of
$$
(2\pi)^{-(n-1)/2}  \lambda^{(n-1)/2}  e^{i\lambda |q(y) - q_H(s)|}
| q(y) - q_H(s)|^{-(n-1)/2} d_{\nu_{y}}  |q(y) - q_H(s)|, \;\;
$$
where $q(y) \in \partial \Omega,
 q_H(s) \in H$, and where $ d_{\nu_{y}}  |q(y) - q_H(s)|$ is the directional
derivative taken in $\R^n \times \R^n$. Hence the symbol is given
by
\begin{equation} \label{SYMBOL1}
|q(y) - q_H(s)|^{-(n-1)/2}  d_{\nu_{y}} |q(y) - q_H(s)| \big| dy ds
\big|^{1/2},
\end{equation}
where we use $(y, s) \in {\mathbb R}^{n-1} \times {\mathbb
R}^{n-1}$ as coordinates on the graph of $\tau_{\partial
\Omega}^H$.

  To express the graph half-density in terms of $dy d\eta$ we need
  to express $ds$ in terms of $d\eta$, keeping $y$ fixed.
Along the graph of $\tau_{\partial \Omega}^H$,  $\eta_i = d_{y_i}
|q(y) - q_H(s)|$ where now the derivative is only along $\partial
\Omega$. Hence,
$$
|d\eta| = \det \left( \frac{\partial^2}{\partial Y^i \partial S^j} \left| ( q(y) + Y_i
e_i ) -  ( q_H(s) + S_i e'_i ) \right| \right)_{S=Y=0}  |ds|,
$$
where $e_i$ is an orthonormal basis for $T_y \partial \Omega$, and
$e'_i$ an orthonormal basis for $T_{s} H$. This is the same type
of determinant as in \cite{HZ}, Proposition 6.1, except that
$q_H(s) \in H$ rather than in $\partial \Omega$.  As in (loc.
cit.) we first compute the determinant in  the   two dimensional
case.   We choose coordinates so that $q(y) = (0,0), q_H(s) =
(0,r)$, $e_1 = (\cos \alpha, \sin \alpha)$ and $e_2 = (\cos \beta,
\sin \beta)$. Then the determinant is
$$
\frac{\partial^2}{\partial Y \partial S} \left| (Y \cos \alpha -   S
\cos \beta, Y \sin \alpha - r - S \sin \beta) \right|_{S=Y=0},
$$
which equals
$$
r^{-1} \cos \alpha \cos \beta = |q(y) - q_H(s)|^{-1} \partial_{\nu_y} | q(y) -
 q_H(s)| \partial_{\nu_{s}} | q(y) - q_H(s)|.
$$
Taking this number to the power $-1/2$ and substituting in
(\ref{SYMBOL1}),  we find that the factors of $|q(y) -
q_H(s)|^{-1/2}$ cancel, and that there is a half power
cancellation in $\partial_{\nu_y} | q(y) -
 q_H(s)|$, so that the symbol is given by
\begin{multline} \left( \frac{d_{\nu_{y}} |q(y) - q_H(s)|}{d_{\nu_{s}} |q(y) - q_H(s)|}
\right)^{1/2}  \left| dy d\eta \right|^{1/2}  =  \left(
\frac{\gamma(y, \eta)}{\gamma (\tau_{\partial \Omega}^H(y,\eta))}
\right)^{1/2} \left| dy d\eta \right|^{1/2}.
\label{F-symbol}\end{multline}

To complete the calculation, we observe that  $\partial_{\nu_{s}}
| q(y) - q_H(s)| = \gamma(s, \tau)$, and  $
\partial_{\nu_{y}} | q(y) - q_H(s)| =   \gamma(y, \eta), $ and
this gives the symbol of $N_{\partial \Omega}^H$ in dimension two.

In higher dimensions, as in \cite{HZ}, we use the subspace
$T_{q(y)}
\partial \Omega \cap T_{q_H(s)} H \cap \ell_{q, s}^{\perp}$ where
$\ell_{q, s}$ is the line joining $q(y)$ and $q_H(s)$. We then
introduce the same coordinates as in \cite{HZ} and find that in
the $n$-dimensional case, the determinant equals $r^{-(n + 1)}
\cos \alpha \cos \gamma$ where $e_1 = (\cos \alpha, 0, \sin
\alpha), e_2 = (0, 1, 0)$ are the first two elements of an
orthonormal basis of $T_{q(y)} \partial \Omega$ and $e_1' = (\cos
\gamma \cos \beta, \cos \gamma \sin \beta, \sin \gamma), e_2' = (-
\sin \beta, \cos \beta, 0)$ are the first two elements of an
orthonormal basis of $T_{q_H(s)} H$. The factors of $|q(y) -
q_H(s)|$ again cancel out due to the factor of $r^{-n + 1}$. The
cosine factors also produce half-cancellations as in the
two-dimensional case, leaving the stated symbol.

The calculation of the symbol of  $N_{\partial \Omega}^{H \nu}$ is
similar except that the kernel has one further derivative, and the
leading term arises by differentiating the exponent in $\nu$. So
its  symbol is that of $N_{\partial \Omega}^H$ multiplied by $i
\lambda
\partial_{\nu_{s}} |q_H(s) - q(y)|$ where $s \in H, y \in
\partial \Omega$ and the derivative is computed in $\R^n \times
\R^n$. In terms of the symplectic coordinates $(y, \eta)$ on the
boundary,  the symbol of $N_{\partial \Omega}^{H \nu}$ is the
product of the symbol of $N_{\partial \Omega}^{H}$ by the
additional factor of
$$i \lambda \partial_{\nu_{s}} |q(y) - q_H(s)|. $$

\end{proof}

This completes the proof of Proposition \ref{GAMMAC}. Q.E.D.

\section{\label{DEC}  $F(\lambda) = N^{H}_{\partial \Omega}(\lambda)^{*}
Op_{\lambda}(a) N^{H}_{\partial \Omega}(\lambda)$ and $F^{\nu}
(\lambda) = N^{H \nu}_{\partial \Omega}(\lambda)^{*}
Op_{\lambda}(a) N^{H \nu}_{\partial \Omega}(\lambda)$
\label{splitting}}

The purpose of this section is to prove that the  compositions
\begin{equation} \label{FLAMBDA} F(\lambda): = N_{\partial \Omega}^H(\lambda)^* Op_{\lambda} (a)
N_{\partial \Omega}^H(\lambda), \;\; \mbox{resp.}\;\;
F^{\nu}(\lambda) : = N_{\partial \Omega}^{H \nu} (\lambda)^*
Op_{\lambda} (a) N_{\partial \Omega}^{H, \nu}(\lambda)
\end{equation} have the decomposition as described in
(\ref{decomposition}). The precise statement is given in
Proposition \ref{Egorov}. The details and statements are the same
for $F(\lambda)$ and $F^{\nu}(\lambda)$ except for the calculation
of the symbols. Hence, we sometimes give them only for
$F(\lambda)$.

 By Proposition
\ref{GAMMAC} and the calculus of Fourier integral operators,
$F(\lambda)$ would   be a Fourier integral operator associated to
the composite canonical relation $(\Gamma_{\beta}^H )^* \circ
\Gamma_{\beta}^H $ if the composition were  non-degenerate (or
clean); see \cite{DS}. The composition is degenerate along the
tangential set,  and so in analogy with the boundary case in
\cite{HZ} we must take care to avoid both the tangential and singular sets.
 We handle these problems by using appropriate cutoffs to remove these  thin
subsets. Thus, we  first consider $ \lambda$-pseudodifferential
operators  $Op_{\lambda}(a)
 \in Op_{\lambda}(S^{0,0}_{cl} (T^*H \times [0,\lambda_0^{-1}] ) )$ with supp  $a \,  \cap  \tau_{\partial \Omega}^{H}
  ( \,   \mathcal T  \cup \Sigma   \cup \gcal \, ) = \emptyset$ for $ k =1,2$; here as above,  $\tau_{\partial \Omega}^{H}:B^*\partial
  \Omega \rightarrow B^*H$ is the transfer  billiard map. As in  \cite{HZ}, a density
  argument shows that such operators suffice
 to prove Theorem \ref{maintheorem} in the  general case. With this in mind,  we henceforth assume that the symbol $a \in S^{0,0}_{cl}
  (T^*H \times [0,\lambda_0^{-1}])$ satisfies the following support condition:  For arbitrarily small fixed $\epsilon>0$  and
   any
 $(s_0,\tau_0) \in \tau_{\partial \Omega}^{H}( \, \tcal  \cup \gcal \cup B_{\Sigma}^* \partial \Omega  \, ),$\begin{equation} \label{cutaway}
  a(s,\tau;\lambda) = 0, \end{equation}
  for all $(s,\tau) \in B^*H$ with $|s-s_0|^{2} + |\tau - \tau_0|^{2} < \epsilon^2$ and $\lambda \geq \lambda_0. $

\subsection{\label{betas} The canonical relation of $F(\lambda)$}

 Temporarily ignoring  the
 complications involving singular and tangential rays,
 the phase of the composition is $\Psi(s,y,y'); = |q(y) - q_H(s)| - |q_H(s) - q(y')|$ with $q_H(s)\in H, q(y),q(y') \in
\partial \Omega.$ Hence the critical set of the phase $\Psi$ is
\begin{equation} \label{C} \begin{array}{lll} C_{\Psi} = \{
(s,y,y');
 \,  \nabla_s  (|q(y) - q_H(s)| - |q_H(s) - q(y')|) = 0  \} \\ \\
=  \{ (s,y,y'); \,  \langle r(q_H(s),q(y)),
T_s \rangle =  \langle r(q_H(s), q(y')), T_s \rangle \}. \end{array} \end{equation} The
canonical relation of $F(\lambda)$  is then
$$  \Gamma_{F(\lambda)} = \{(y, \pi^T_y  r(q(y),q_H(s))  , y', - \pi^T_{y'}
r(q(y'),q_H(s));  \, (s,y,y') \in C_{\Psi} \}.$$

One component of $ \Gamma_{F(\lambda)}$  consists of the diagonal
$\Delta_{\ccal \times \ccal} \subset  \ccal \times \ccal $ (see
Definition \ref{HBILLDEF}).  Under the convexity assumption on
$H$, this  component occurs with multiplicity two since there are
two points of $H$ where the ray my intersect $H$ and then reflect
to its initial position. Two other `components' of $
\Gamma_{F(\lambda)}$ (which meet along rays tangential to $H$)
come from once-broken trajectories from $(y, \eta) \in B^*
\partial \Omega$ to $(y', \eta') \in
\partial \Omega$ with the break at $q_H(s) \in H$. There exist (at most) two possible break points
corresponding  to the (at most) two possible intersection points
of a line with $H$.

 We denote the partial symplectic maps
defined by the once-broken trajectories by   $\beta_{H}^{1}$,
resp.   $\beta_{H}^{2 }$, where  the index indicates the first,
resp. second  point of intersection with $H$. As in the
introduction (see (\ref{TRANS}) we refer to these partial
symplectic maps as transmission maps. Summing up, we have

\begin{lem} \label{GAMMABETADEF}  Assuming $H$ is convex, the   canonical relation
 underlying $
N_{\partial \Omega}^H(\lambda)^* Op (a) N_{\partial
\Omega}^H(\lambda)$ (and of $ N_{\partial \Omega}^{H \nu}
(\lambda)^* Op (a) N_{\partial \Omega}^{H \nu}(\lambda)$)  equals
$$\Gamma_{\tau_{\partial \Omega}^{H * } } \circ \Gamma_{\tau_{\partial \Omega}^{H  } } =  \Delta_{\ccal
\times \ccal}  \cup  \Gamma_{\beta_H^1}  \cup
\Gamma_{\beta_H^2}  $$ which   consists of
three branches:

\begin{itemize}

\item The identity branch on $\ccal$ (with
multiplicity two);

\item The graphs of $\beta_{H}^{1}$ and $\beta_{H}^{2 }$; the
index specifies the point of intersection with $H$.

\end{itemize}

\end{lem}

The case of general smooth $H$ is of a similar nature but with
more branches. We therefore omit the details.

\subsection{The decomposition of $F(\lambda)$}

We now give a  precise proof that deals with the
$\lambda$-pseudodifferential  factor, the tangential degeneracy
and the singular points. We continue to assume that $H$ is convex,
but again, this is for notational simplicity.  The $\lambda$-FIO
contribution to $N_{\partial \Omega}^{H}(\lambda)^*
Op_{\lambda}(a) N_{\partial \Omega}^{H}(\lambda)$ is discussed in
subsection \ref{diagonal}.

Let $\delta(\epsilon) >0$ be  small (to be specified later on) and let
$\chi \in C^{\infty}_{0}(\R)$ with $\chi (x) =1$ when $|x| \leq
1/2$ and $\chi(x) =  0$ for $|x| \geq 1$. We define the rescaled
cutoff function $\chi_{\delta(\epsilon)}(x) := \chi( \frac{x}{\delta(\epsilon)})$.
The point of the next Proposition is to show that by choosing symbols $a$  satsifying (\ref{cutaway})  (ie. with support disjoint from the image under the transfer map of the generaltized grazing and singular sets), the corresponding $F_1(\lambda;a,\epsilon)$-operator in the decomposition (\ref{F1DEF}) is, modulo a residual operator, a $\lambda$-pseudodifferential operator of order zero with principal symbol $(\tau_{\partial \Omega}^H)^*
a_0 \cdot | \sigma(N_{\partial \Omega}^{H}) |^{2}.$ Following \cite{HZ} we say that a $\lambda$-dependent operator $R$ is {\em residual} provided it is smoothing and $|\partial_{x}^{\alpha} \partial_y^{\beta} R(x,y;\lambda)| = {\mathcal O}_{\alpha,\beta}(\lambda^{-\infty}).$ The $F_2(\lambda;a,\epsilon)$-operator  in the decomposition $(\ref{F1DEF})$ is studied subsequently in subsection \ref{diagonal}. It is a sum of two zeroth-order $\lambda$-Fourier integral operators.

\begin{prop} \label{Egorov} Suppose that $H$ is convex and that $a \in S^{0,0}(T^{*}H \times (0,\lambda_0^{-1}])$
with $a \sim \sum_{j=0}^{\infty} a_j \lambda^{-j}$ satisfies the support condition (\ref{cutaway}).
Then, for any $\epsilon  >0$
\begin{multline}
 (i) \, N_{\partial
\Omega}^{H}(\lambda)^* Op_{\lambda} (a) N_{\partial
\Omega}^{H}(\lambda) = Op_{\lambda} (  \,  a_0( \tau_{\partial \Omega}^{H}(y,\eta))  \, \cdot \, \gamma(y,\eta) \cdot \gamma^{-1}(\tau_{\partial \Omega}^{H}(y,\eta) )  \,) \\+
F_{21}(\lambda;a,\epsilon)  + F_{22}(\lambda;a,\epsilon) + R(\lambda;\epsilon),  \\ \\
(ii) \, \lambda^{-2} N_{\partial
\Omega}^{H\nu}(\lambda)^* Op_{\lambda} (a) N_{\partial
\Omega}^{H\nu}(\lambda) = Op_{\lambda} (  \,  a_0( \tau_{\partial \Omega}^{H}(y,\eta))  \, \cdot \, \gamma(y,\eta) \cdot \gamma (\tau_{\partial \Omega}^{H}(y,\eta) )  \,) \\ +
F_{21}^{\nu}(\lambda;a,\epsilon)  + F_{22}^{\nu}(\lambda;a,\epsilon) + R^{\nu}(\lambda;\epsilon). \end{multline}
In both cases
 $F_{21}(\lambda;a,\epsilon)$ (resp. $F_{21}^{\nu}(\lambda;a,\epsilon)$) and $F_{22}(\lambda;a,\epsilon)$ (resp. $F_{22}^{\nu}(\lambda;a,\epsilon)$) are $\lambda$-Fourier integral operators of order zero with $\kappa( F_{21}(\lambda;a,\epsilon)) = \kappa( F_{21}^{\nu}(\lambda;a,\epsilon)) =  \beta_{H}^{1}$   and $\kappa ( F_{22}(\lambda;a,\epsilon)) = \kappa ( F_{22}^{\nu}(\lambda;a,\epsilon) ) =  \beta_{H}^{2} $
 and $ \max  \, (  \, \| R(\lambda;\epsilon) \|_{L^2 \rightarrow L^2},  \| R^{\nu}(\lambda;\epsilon) \|_{L^2 \rightarrow L^2} \, ) \, = O_{\epsilon}(\lambda^{-1}).$
\end{prop}

\begin{rem} The case of a general smooth $H$ is similar but with a
number of branches depending on $(y, \eta)$. We omit the details.
\end{rem}

\begin{proof}

 Consider the operator $N_{\partial
\Omega}^{H}(\lambda): C^{\infty} (\partial \Omega) \rightarrow
C^{\infty}(H)$ with Schwartz kernel $N(q_H,q;\lambda) = \partial_{\nu_q} G_{0}(q_H,q;\lambda)$. By an  integration  by parts argument, it follows that
\begin{equation} \label{wf1}
WF'_{\lambda}(N_{\partial \Omega}^{H}(\lambda))  \subset B^* \partial \Omega \times B^* H.
\end{equation}

 Using $\chi_{\delta(\epsilon)}$ we
decompose the Schwartz kernel of  $N_{\partial
\Omega}^{H}(\lambda)^{*}
 Op_\lambda (a)  N_{\partial \Omega}^{H}(\lambda)$ into  near-diagonal and complimentary terms:
\begin{equation} \label{F1DEF}  N^{H}_{\partial \Omega}(\lambda)^{*} Op_{\lambda}(a)
N^{H}_{\partial \Omega}(\lambda )(q,q')  = F_{1}(\lambda;a,\epsilon)(q,q')  + F_{2}(\lambda;a,\epsilon)(q,q');
\,\,\,\, (q,q') \in \partial \Omega \times \partial \Omega,
\end{equation}  where,
\begin{equation} \label{k1}
 F_{1}(\lambda;a,\epsilon)(q,q') := N^{H}_{\partial \Omega}(\lambda)^{*}
 Op_{\lambda}(a) N^{H}_{\partial \Omega}(\lambda)(q,q') \times \chi_{\delta(\epsilon)}(|q-q'|) \,\,
 \end{equation}
and
\begin{equation}\label{k2}
F_{2}(\lambda;a,\epsilon)(q,q') := N^{H}_{\partial
\Omega}(\lambda)^{*} Op_{\lambda}(a) N^{H}_{\partial
\Omega}(\lambda)(q,q') \times (1- \chi_{\delta(\epsilon)} )(|q-q'|).
\end{equation}

We also introduce a new cutoff $\Xi_{\epsilon}\in
C^{\infty}_{0}(B^*_{1+ \epsilon}H)$ with $\Xi_{\epsilon}(y,\eta) = 1$ for $(y,\eta) \in B_{1+ \epsilon/2}^{*}(H)$. Then by the
calculus of  wave-front sets, we have
$$  F_{1}(\lambda;a,\epsilon)(q,q') = N^{H}_{\partial \Omega}(\lambda)^{*}
Op_{\lambda}(\Xi_{\epsilon}) Op_{\lambda}(a) N^{H}_{\partial \Omega}(\lambda)(q,q')
 \times \chi_{\delta(\epsilon)}(|q-q'|) +{\mathcal O}_{\epsilon}(\lambda^{-\infty}) $$
  uniformly for $(q',q) \in \partial \Omega \times \partial \Omega.$  By (\ref{wf1}) it suffices to assume that $a \in C^{\infty}_{0}(B^*_{1+\epsilon}H)$
  and satisfies the support condition (\ref{cutaway}) and we will do so from now on.

From Proposition \ref{GAMMAC} we know that $N_{\partial \Omega}^{H}(\lambda)$ is a zeroth-order $\lambda$-Fourier integral operator with canonical relation $\Gamma_{\tau_{\partial \Omega}^{H}}$. Differentiation of the phase function $\Psi(y,y',s) = |q_H(s) - q(y)| - |q_H(s)-q(y')|$ in $s$ and a repeated  integration by parts implies that for some constant $C>0,$
$$ WF_{\lambda}' ( F_{1}(\lambda;a,\epsilon) )  \subset  \{ (y,\xi,y',\eta) \in B^*\partial \Omega \times B^* \partial \Omega; |y-y'| \leq \delta(\epsilon), \, |\xi -\eta| \leq C \delta(\epsilon) \}.$$
Since  (see section \ref{betas}),
$$( \Gamma_{\beta^1_H} \cup \Gamma_{\beta^{2}_{H}} ) \cap \Delta_{\ccal \times \ccal} \subset \tcal \cup \gcal,$$
  and   $\text{dist}  ( \text{supp}  \, a, \, \tau_{\partial \Omega}^{H}( \Sigma \cup \tcal \cup \gcal ) \, ) \geq \epsilon,$  it follows that by choosing $\delta(\epsilon)>0$ sufficiently small,
$$ WF'_{\lambda}( F_1(\lambda;a,\epsilon)) \cap  ( \Gamma_{\beta_{H}^{1}} \cup \Gamma_{ \beta_{H}^{2}} ) = \emptyset.$$
 Thus,  in view of Lemma \ref{GAMMABETADEF},
\begin{equation} \label{decomp1}
WF'_{\lambda} ( F_{1}(\lambda;a,\epsilon) ) \subset  \Delta_{ \ccal^{o} \times \ccal^{o}}. \end{equation}

By the semiclassical Egorov theorem \cite{DS},
$ F_1(\lambda;\epsilon) \in \Psi_{sc}^{0}(\partial \Omega),$
and moreover,
\begin{equation} \label{decomp2}
F_{1}(\lambda;a,\epsilon) = Op_{\lambda}(  \, (\tau_{\partial \Omega}^{H})^* a_0 \cdot |\sigma (N_{\partial \Omega}^{H}(\lambda))|^{2})|^{2}  \, ) + {\mathcal O}_{\epsilon}(\lambda^{-1})_{L^2 \rightarrow L^2}. \end{equation}

From Lemma \ref{symbolcomputation}, $|\sigma(N_{\partial \Omega}^{H}(\lambda)(y,\eta)|^{2} = \gamma(y,\eta) \times \gamma^{-1}(\tau_{\partial \Omega}^{H}(y,\eta))$ and so the formula in (\ref{decomp2}) implies that
\begin{equation} \label{decomp3}
F_{1}(\lambda;a,\epsilon) = Op_{\lambda}(  \, a_0(\tau_{\partial \Omega}^{H}(y,\eta)  ) \cdot \gamma(y,\eta) \cdot \gamma^{-1}(\tau_{\partial \Omega}^{H}(y,\eta))  \, ) + {\mathcal O}_{\epsilon}(\lambda^{-1})_{L^2 \rightarrow L^2}. \end{equation}
\smallskip

In case (ii) one makes the analogous decomposition
\begin{equation} \label{F2DEF}  N^{H\nu}_{\partial \Omega}(\lambda)^{*} Op_{\lambda}(a)
N^{H\nu}_{\partial \Omega}(\lambda )(q,q') F_{1}^{\nu}(\lambda;a,\epsilon)(q,q')  + F_{2}^{\nu}(\lambda;a,\epsilon)(q,q');
\,\,\,\, (q,q') \in \partial \Omega \times \partial \Omega
\end{equation}  where,
\begin{equation} \label{k1a}
 F_{1}^{\nu}(\lambda;a,\epsilon)(q,q') := N^{H\nu}_{\partial \Omega}(\lambda)^{*} Op_{\lambda}(a) N^{H\nu}_{\partial \Omega}(\lambda)(q,q') \times \chi_{\delta(\epsilon)}(|q-q'|) \,\,
 \end{equation}
and
\begin{equation}\label{k2a}
F_{2}^{\nu}(\lambda;a,\epsilon)(q,q') := N^{H\nu}_{\partial
\Omega}(\lambda)^{*} Op_{\lambda}(a) N^{H\nu}_{\partial
\Omega}(\lambda)(q,q') \times (1- \chi_{\delta(\epsilon)} )(|q-q'|).
\end{equation}

The same reasoning as in case (i) implies that

\begin{equation} \label{decompnormal1}
F_{1}^{\nu}(\lambda;a,\epsilon) = Op_{\lambda}(  \, (\tau_{\partial \Omega}^{H})^* a_0 \cdot |\sigma (N_{\partial \Omega}^{H\nu}(\lambda))|^{2}  \, ) + {\mathcal O}_{\epsilon}(\lambda^{-1})_{L^2 \rightarrow L^2} \end{equation}
and by the symbol computation in Lemma \ref{symbolcomputation} one gets
\begin{equation} \label{decompnormal1a}
F_{1}^{\nu}(\lambda;a,\epsilon) = Op_{\lambda}(  \, a_0 (\tau_{\partial \Omega}^{H}(y,\eta)) \cdot \gamma (\tau_{\partial \Omega}^{H}(y,\eta) ) \cdot \gamma(y,\eta) \, ) + {\mathcal O}_{\epsilon}(\lambda^{-1})_{L^2 \rightarrow L^2}. \end{equation}

\subsubsection{Analysis of the $F_{2}(\lambda;\epsilon)$-operator.} \label{diagonal}

In this section, we prove:

\begin{lem} \label{F2}
 $F_{21}(\lambda;a,\epsilon)$ (resp. $F_{22}(\lambda;a,\epsilon)$) are $\lambda$-FIO's of order zero with $\kappa( F_{21}(\lambda;a,\epsilon))  \beta_{H}^{1}$  (resp. $\kappa ( F_{22}(\lambda;a,\epsilon)) = \beta_{H}^{2}
 )$. Their principal symbol (on the respective branch of the correspondence) is given by \begin{multline}
  a \left( \frac{d_{\nu_{y_1}} |q(y_1) - q_H(s)|}{d_{\nu_{s}} |q(y_1) - q_H(s)|}
\right)^{1/2}  \left( \frac{d_{\nu_{y_2}} |q(y_2) - q_H(s)|}{d_{\nu_{s}}
|q(y_2) - q_H(s)|} \right)^{1/2} \big|_{s=q_{H;k}(y_1,y_2)}\\
 = a(\tau_{\partial \Omega}^H (y_2,
\eta_2)) \Big( \frac{\gamma(y_1, \eta_1))}{\gamma(\tau_{\partial
\Omega}^H (y_1, \eta_1)} \Big)^{1/2}  \Big( \frac{\gamma(y_2,
\eta_2))}{\gamma(\tau_{\partial \Omega}^H (y_2, \eta_2)}
\Big)^{1/2}; \,\,\, (y_2,\eta_2) = \beta_{H}^{k}(y_1,\eta_1), \, k=1,2.
\end{multline}

 Here, $s = q_{H,k}(y,y') \in H; k=1,2$ is the point on $H$ joining
$y,y' \in \partial \Omega$
  under the $H$-broken billiard map $\beta_{H}^{k}:; k=1,2,$ (see section \ref{billiard maps}).
\end{lem}

\begin{proof}

 Because of the cutoff
$(1-\chi_{\delta(\epsilon)})(|q(y)-q(y')|)$ in the amplitude of the
$F_{2}(\lambda;a,\epsilon)$-kernel the canonical relation
corresponding to $F_{2}(\lambda;a,\epsilon)$ does not intersect the
diagonal $\Delta_{ B^* \partial \Omega \times B^*
\partial \Omega}$ and so, by the argument in section
\ref{billiard maps} it must consist of the  reflection
canonical relations $\beta_{H}^{1}$ (resp.
$\beta_{H}^{2}$ ). The  corresponding
quantizations are the $\lambda$-Fourier integral operators  $F_{21}(\lambda;a,\epsilon)$
(resp. $F_{22}(\lambda;a,\epsilon)$). Modulo pointwise uniform
$O_{\epsilon}(\lambda^{-\infty})$-errors, the Schwartz kernel is
locally given by the   formula
\begin{multline} \label{snellterm1}
F_{2k}(\lambda;a,\epsilon)(y,y')= (2\pi \lambda)^{n-1} \int_{H}  e^{ i
\lambda [ |q_{H}(s)-q(y)| - |q_{H}(s) - q(y')|]} \,
c(y,y',s;\lambda) \, (1-\chi_{\delta(\epsilon)})(|y-y'|)\\
 \times \chi_{\epsilon}(q_{H}(s)-
q_{H,k}(y,y'))   d\sigma_{H}(s) \end{multline}
 where,  $c \sim \sum_{j=0}^{\infty} c_j \lambda^{-j}$ with
  \begin{multline}
 c_{0}(y,y',s) = a_0(q_{H,j}(s), d_s q_{H,j}(s)  r(q_{H,j}(s),q(y')))
\times \langle \nu_{y}, r(q_H(s),q(y)) \rangle \\
 \times \langle \nu_{y'}, r(q_H(s),q(y')) \rangle   \cdot  b_{0}(q(y),q_{H,j}(s) ) \cdot  \overline{b_{0}(q_{H,j}(s),q(y'))}  \end{multline}
 where, $b_0(q_H(s),q(y)) = |q_H(s)-q(y)|^{- \frac{n-1}{2}}.$

 The symbol is computed by taking the product of the principal
 symbols of the factors given in Lemma
 \ref{symbolcomputation}.

 \end{proof}

 Combining (\ref{decomp2}) with Lemma \ref{F2} completes the proof of  Proposition \ref{Egorov} in the case of   the  $N_{\partial
\Omega}^{H}(\lambda)^{*}
 Op_\lambda (a)  N_{\partial \Omega}^{H}(\lambda)$- operator.

 The following subsection completes the proof of the Proposition.

\subsubsection{Modification for $ N_{\partial \Omega}^{H \nu} (\lambda)^* Op
(a) N_{\partial \Omega}^{H \nu} (\lambda)$}

\begin{lem} \label{F2nu}
 $F_{21}^{\nu}(\lambda;a,\epsilon)$ (resp. $F_{22}^{\nu}(\lambda;a,\epsilon)$) are $\lambda$-Fourier integral operators of order one  with $\kappa( F_{21}(\lambda;a,\epsilon)) =  \beta_{\partial \Omega}^{H;1}$  (resp. $\kappa ( F_{22}(\lambda;a,\epsilon)) = \beta_{\partial \Omega}^{H;2}
 )$. Its symbol
$$ - \left(  \, d_{\nu_{y_1}} |q(y_1) - q_H(s)| \,  d_{\nu_{s}} |q(y_1) - q_H(s)| \,
\right)^{1/2} \times  \left( \, d_{\nu_{y_2}} |q(y_2) - q_H(s)| \, d_{\nu_{s}} |q(y_2) -
q_H(s)|\,
 \right)^{1/2} |_{s=q_{H;k}(y_1,y_2)}$$
 is that of $F^{\nu}_{2 j}$ multiplied by  $$-(i \lambda \partial_{\nu_{s}} |q(y_1) - q_H(s)|) (- i \lambda \partial_{\nu_{s}} |q(y_2) - q_H(s)|)|_{s=q_{H;k}(y_1,y_2)}. $$
In  symplectic coordinates, it is given by
 $$-  \Big(
{\gamma(y_1, \eta_1)) \gamma(\tau_{\partial \Omega}^H (y_1,
\eta_1)} \Big)^{1/2}  \Big( \gamma(y_2,
\eta_2))\gamma(\tau_{\partial \Omega}^H (y_2, \eta_2) \Big)^{1/2}; \,\,\,(y_2,\eta_2) = \beta_{H}^{k}(y_1,\eta_1), \, k=1,2.
$$
\end{lem}

The proof is the same as for Lemma \ref{F2} and is therefore
omitted. This completes the proof of the Proposition. \end{proof}

\section{Proof of Theorem \ref{CDTHM}} \label{CDsection}

It is important to observe the sign difference between the symbol
of the FIO part of $N_{\partial \Omega}^H(\lambda)^* Op(a)
N_{\partial \Omega}^H(\lambda)$ and its $\nu$-partner. In the
calculation of the first symbol,  the normal derivative with
respect to $s \in H$ comes about when we change from $(s, s')$
coordinates to $(s, \tau)$ coordinates. So this normal is
symmetric with respect to the two sides of $H$ (at each
intersection point).  But in the $N_{\partial \Omega}^{H, \nu}$
operator it is assymetrical. This explains why the $F_2$ operators
have opposite signs in the two $N^*N$ terms.

 We note also that  in the case of the $\lambda$-pseudodifferential $F_{1}^{\nu}(\lambda;a,\epsilon)$-term,
  $y_1 = y_2$ and so,  $y_1$ and $y_2$  trivially lie on the same side of $H$ relative to
   the normal vector $\nu_{s=q_{H;k}(y_1,y_2) }.$   On the other hand, in the $\lambda$-Fourier
   integral case of $F_{2}^{\nu}(\lambda;a,\epsilon)$, $y_1$ and $y_2$ lie on opposite sides of $H$
   relative to the normal vector $\nu_{s=q_{H;k}(y_1,y_2)}$. This creates the extra minus sign in
    Lemma \ref{F2nu} which allows for the cancellation between the (appropriately weighted) $F_2$ \and
     $F_2^{\nu}$-terms.

We use these observations to reduce the proof of Theorem
\ref{CDTHM}  to that of \cite{HZ}:

\begin{lem} \label{cancellation} The principal symbol of the Fourier integral operator
part of $$ N_{\partial \Omega}^{H } (\lambda)^* Op_{\lambda}((1 - |\tau|^2)
a) N_{\partial \Omega}^{H} (\lambda) + \lambda^{-2}
N_{\partial \Omega}^{H \nu} (\lambda)^* Op (a) N_{\partial
\Omega}^{H \nu} (\lambda) $$ equals zero. \end{lem}

\begin{proof}

The principal symbols of the $F_2(\lambda)$ terms of the two
operators are calculated respectively in Lemmas \ref{F2} and
\ref{F2nu}. As discussed in the proofs, the second (Neumann)
symbol equals the first (Dirichlet) multiplied by the factor of $
\partial_{\nu_{s}} |q(y_1) - q_H(s)|)
\partial_{\nu_{s}}|q(y_2) - q_H(s)|) |_{s = q_{H;k}(y_1,y_2)}$.  By definition, $q(y_2) - q_H(s)$ and $q(y_1) - q_H(s)$ have the same
 tangential projection to $H$. But they lie on opposite sides of $H$ and
 have opposite normal projections. So this factor equals $- \cos
 \vartheta(y_1, s)^2$ where $\vartheta$ is the angle to the
 normal. In symplectic coordinates of the projection   $(s,\tau)$ on $B^* H$,   $ \cos
 \vartheta(y_1, s)^2 = 1 -
 |\tau|^2$.
) Hence, if we take a diagonal pseudo-differential operator ${\bf
A}_2$ with $\sigma_{A_{11}}(s, \tau) = (1 - |\tau|^2) a(s, \tau)$
and $ \sigma_{A_{22}}(s,\tau) = a(s,\tau)$ for all $(s, \tau) \in
B^*H$, and if we normalize the second term by dividing the Neumann
data by $\lambda$,  then the Fourier integral parts of the
$N_{\partial \Omega}^{H}(\lambda)^* Op_{\lambda}(a) N_{\partial
\Omega}^{H}(\lambda)$ compositions cancel each other. Then we are
left with only the pseudo-differential parts of each composition,
completing the proof of the Lemma.
\end{proof}

 Thus, by Lemma \ref{cancellation} and Proposition \ref{Egorov},
\begin{multline} \label{CDupshot}
\langle  A_{11}(\lambda) u_{\lambda}^{H}, u_{\lambda}^{H}  \rangle_{L^{2}(H)} +  \lambda^{-2} \langle  A_{22}(\lambda) u_{\lambda}^{H}, u_{\lambda}^{H} \rangle_{L^{2}(H)} \\   = 2  \, \langle Op_{\lambda}( \, (\tau_{\partial \Omega}^{H})^* a  \, \cdot \, (\tau_{\partial \Omega}^{H})^* \gamma  \, \cdot \, \gamma )  u_{\lambda}^{b}, u_{\lambda}^{b} \rangle_{L^{2}(\partial \Omega)}  +  {\mathcal O}(\lambda^{-1})  \| u_{\lambda}^{b} \|_{ L^{2}(\partial \Omega) }^{2} \end{multline}
 and the statement of
Theorem \ref{CDTHM} then follows immediately from  (\ref{CDupshot}) and an application of the boundary
quantum ergodicity result of \cite{HZ}. \qed

\section{Boundary Weyl law: Proof of Theorem \ref{weyl2}} \label{Weyl}

We now prove a key result underlying  Theorem \ref{maintheorem} on
quantum ergodic restriction for Dirichlet data. The proof hinges
on an averaging argument in \S \ref{CRAZY}, and on a boundary
local Weyl law. The purpose of this and the next section is to
prove  a weak version of a boundary Weyl law for $\lambda$-Fourier
integral operators. The presence of the boundary causes a number
of technical complications, but the result seems to us of
independent interest. This section is independent of the choice of
$H$.

To prepare for the boundary result, let us first recall  the case
of manifolds without boundary and eigenfunctions of a Laplacian
\cite{Z}. The  Weyl law for classical Fourier integral operators
$F: C^{\infty}(M) \rightarrow C^{\infty}(M)$ of order zero
 says that
\begin{equation} \label{oldWeyl}
\frac{1}{N(\lambda)} \sum_{j: \lambda_j \leq \lambda} \langle F
\phi_{\lambda_j}, \phi_{\lambda_j} \rangle \to \int_{\Gamma_F \cap
\Delta_{T^*M} } \sigma_{\Delta} (F) d\mu, \end{equation} where the
integral is over the intersection of the canonical relation
$\Gamma_F$ of $F$ with the diagonal of $T^*M \times T^*M$, i.e.
the fixed-point set of the symplectic correspondence underlying
$F$.  The right side is zero unless the fixed-point set  has
dimension $m = \dim M$, the leading order trace  sifts out the
`pseudo-differential part' of $F$. The fixed point set is assumed
to be almost-clean in the sense defined above Theorem
\ref{weyl2}.

The purpose of this section is to prove an analogous result for
semi-classical Fourier integral operators acting on boundary
values of eigenfunctions.  Our goal is to prove  (Theorem
\ref{weyl2}) that
 \begin{equation} \label{FLWL} N_{F(\lambda)} (\lambda): =  \sum_{j: \lambda_j \leq
\lambda} \langle F(\lambda_j) u_{\lambda_j}^{b}, u_{\lambda_j}^{b}
\rangle  = o(1) \end{equation} if  $|\Sigma_{F(\lambda)}| = 0$ or
if the quantitative almost nowhere commuting condition holds.

 The study of these $\lambda$-FIO Weyl sums  introduces two new aspects to the local Weyl law. First is
the semi-classical aspect, which would also arise in local Weyl
laws for matrix elements $\langle F(\lambda_j) \phi_{\lambda_j},
\phi_{\lambda_j} \rangle $ of semi-classical FIO's relative to
Laplace eigenfunctions. As is usual with semi-classical FIO's, we
 `homogenize' the traces by taking the   Fourier transform in the
semiclassical parameter $\lambda$  and integrating  over the dual
frequency variable.

The second novelty is the restriction to the boundary. The
semi-classical parameter is the interior eigenvalue, not the
eigenvalue of an operator on the boundary. The restriction
operator to the boundary resembles a Fourier integral operator,
except for the complications of grazing rays, corners and so on.
These play a minor role in the trace since they occur on sets of
measure zero and are handled as in \cite{HZ} by introducing
appropriate  cutoff operators.  We omit some of these technical
details for the sake of brevity, since they are the same as in
\cite{HZ}. In the subsequent article  \cite{TZ3},  we prove a more
complete  boundary Weyl law for $\lambda$-FIO's without assuming
measure zero conditions.

Before stating our result, we introduce some further notation and
assumptions.  We assume that  there is an coveing $
\bigcup_{j=1}^{K} q^{(j)}(U) = \partial \Omega$ with $U \subset
\mathbb{R}^{n-1}$ open,  such that  in terms of the  local
parametrizations $y \mapsto q^{(j)}(y); j=1,...,K $, the phase
  function has the form
\begin{equation} \label{PHASEASSUMPTION} \Phi(q^{(j)}(y),q^{(j)}(y'),\theta) = \langle y, \theta \rangle -
\psi^{(j)}(y',\theta), \,\,\,\,\,\, \psi^{(j)} \in C^{\infty}(U \times {\mathbb R}^{n-1}). \end{equation}
The   associated canonical
relation (\ref{CANONICAL}) is the graph of a symplectic
correspondence \begin{equation} \label{KAPPA} \kappa_{F}:
B^*\partial \Omega \rightarrow B^*\partial \Omega. \end{equation}
This is sufficient for our purposes since the  semi-classical
FIO's of interest here are the  pieces of $ F(\lambda) N_{\partial \Omega}^H(\lambda)^* Op_{\lambda}(a) N_{\partial
\Omega}^H(\lambda)$. We denote by  ${\bf r}_{\partial \Omega}: S^{*}\Omega_{
\partial
 \Omega} \rightarrow S^{*}\Omega_{ \partial \Omega}$ the  reflection of inward pointing covectors
 at the boundary with respect to the interior unit normal given by ${\bf r}_{\partial \Omega}(\zeta) = \zeta + 2 \langle \zeta,\nu_q \rangle \nu_q; \,\, q\in \partial \Omega.$
 Also, the map $\pi_{T}:S^{*}_{\partial \Omega} \Omega \rightarrow B^{*} \partial \Omega$ denotes  the canonical tangential projection along $\partial \Omega$ and we recall that $G^{\tau}$ denotes the broken geodesic (billiard)
   flow (\ref{GT}). From now on, we drop the $(j)$-superscripts in the phase functions (\ref{PHASEASSUMPTION}) with the understanding that computations are local.

\medskip

\subsection{Proof of Theorem \ref{weyl2}}
\begin{proof} In the following, we assume that $\partial \Omega$ is strictly convex. We treat the general non-convex  case  in subsection  \ref{nc}.

Following the strategy of the  Fourier Tauberian theorem
\cite{SV}, we calculate the  principal term of the Weyl
asymptotics of order $\lambda^n$ by computing the singularity at
$t = 0$ of order $(t + i0)^{- n}$ of
\begin{equation} \label{IFLAT} I^{b}(t):= \sum_j e^{i t \lambda_j} \langle
F(\lambda_j) u_{\lambda_j}^{b}, u_{\lambda_j}^{b} \rangle.
\end{equation} We note that by adding a constant multiple of the identity
operator, we may assume that the Weyl sums are monotonically
increasing, hence the Fourier Tauberian theorem applies (see
\cite{Z} for the parallel statement in the boundaryless case).

To calculate the principal singularity of (\ref{IFLAT}) at $t 0$, we write  $$F(\lambda) = \frac{1}{2\pi}
\int_{\R} e^{i \lambda s} \hat{F}(s) ds, $$  where
$$\hat{F}(s) = ( \fcal_{\mu \mapsto s} F)(s) =  \int_{\R} e^{- i \mu s} F(\mu) d\mu. $$ Then
$\hat{F}(s)$ is a homogeneous Fourier integral operator with
Schwartz kernel,
\begin{equation} \label{FIOkernel}
 \hat{F}(s)(q, q') = (2\pi \mu)^{n-1} \int_{\R_+} \int_{\R^{n-1}} e^{i \mu(  \Phi(q, q',
\theta) - s)} A(q,q',\theta,\mu) \,  d \theta d\mu,\end{equation}
 where $A(q,q',\theta,\mu) \sim \sum_{j=0}^{\infty} A_j(q,q',\theta) \mu^{-j}, $ with $A_j \in C_{0}^{\infty}(\partial \Omega \times \partial \Omega \times \mathbb{R}^{n-1} ).$ We then
have,
$$\sum_j e^{i t \lambda_j} \langle F(\lambda_j) u_{\lambda_j}^{b}, u_{\lambda_j}^{b}
\rangle = \int_{\R}   \sum_{j} e^{i \lambda_j (s + t)} \langle
\hat{F}(s) u_{\lambda_j}^{b}, u_{\lambda_j}^{b} \rangle
ds,
$$
or equivalently,  \begin{equation} I^{b}(t):=Tr \int_{\R}
 \gamma_{\partial \Omega}^* \hat{F}(s) \gamma_{\partial \Omega}  U(t + s)  \,  ds.
\end{equation}  The presence of the boundary restriction operator $\gamma_{\partial \Omega} $ has the effect
of restricting  kernels to the boundary. To see this, we note that
if $\gamma_H$ denotes the restriction operator to $H$, $\gamma_H f = f|_H, $
then $\gamma_H^* f = f \delta_H, $ since $\langle \gamma_H^* f, g \rangle \int_H f g d\sigma. $  Hence, in the Neumann case,  writing $U^{b}(s) = \gamma_{\partial \Omega} U(s) \gamma_{\partial \Omega}; \, \, s \in {\mathbb R}$ for the boundary restriction of the operator $U(s),$
\begin{equation}
\label{NICETRACE}  I^{b}(t) = \int_{\R} \int_{\partial \Omega}
\int_{
\partial \Omega}U^{b}(t +s, q, q') \hat{F}(s, q', q) ds d\sigma(q')
d\sigma(q).
\end{equation}
In the Dirichlet case, restriction first takes normal derivatives. In the following, we wish to compute the coefficient of the large
$(t+i0)^{-n}$-singularity of $I^{b}(t)$ as $t \rightarrow 0.$

\subsection{Microlocal cutoffs}

To simplify  the trace we introduce some
microlocal cutoffs:  Since both  the singular set, $\Sigma,$ and the grazing set, $\gcal$, are closed and of measure zero,  by the $C^{\infty}$ Urysohn lemma, for  $\epsilon >0$  arbitrarily small, we let $\chi_{\epsilon} \in C^{\infty}_{0}(B^*\partial \Omega)$ be supported an $\epsilon$-neighbourhood of $\Sigma \cup \gcal.$.  Such cutoffs are necessary but quite standard \cite{HZ}.

 In addition, we will also  insert a cutoff  in the integration in $s$ in (\ref{NICETRACE}) near $s \sim \Phi$ and localize the trace in time. This cutoff is more novel and is also important since it allows us to ultimately replace $U^{b}(t+s)$ by the corresponding {\em finite-time} reflection (ie. Chazarain) parametrix in the formula (\ref{NICETRACE}) for the boundary trace, $I^b(t)$.   We define the $\lambda$-microlocally cutoff operators $ F_{\epsilon}(\lambda):=  Op_{\lambda}(1-\chi_{\epsilon}) F(\lambda) Op_{\lambda}(1-\chi_{\epsilon})$
  and $U_{\epsilon}^{b}(t+s):=  Op_{\lambda}(1-\chi_{\epsilon}) U^{b}(t+s) Op_{\lambda}(1-\chi_{\epsilon}).$
  By standard composition calculus, one can write the Schwartz kernel  $F_{\epsilon}(\lambda)(q(y),q(y')) = (2\pi \lambda)^{n-1} \int_{{\mathbb R}^{n-1}} e^{i \lambda \Phi(q(y),q(y'),\theta)} A_{\epsilon}(q(y),q(y'),\theta) \, d\theta$ where $A_{\epsilon}(q(y),q(y'),\theta) = A(q(y),q(y'),\theta)  \cdot (1-\chi_{\epsilon})(q(y),d_{y}\Phi) \, \cdot \, (1-\chi_{\epsilon})(q(y'),d_{y'}\Phi) + {\mathcal O}(\lambda^{-1}).$ The precise statement of the $\lambda$-microlocalization of the boundary trace is given in the following \begin{lem} \label{cutofftrace}
   Let $\chi \in C^{\infty}_{0}(\R)$ with $\chi(s) = 1$ near $s=0$ and for any $\epsilon >0$ let  $\chi_{\epsilon} \in C^{\infty}_{0}(B^*\partial \Omega)$ be supported in an $\epsilon$-neighbourhood of $B_{\Sigma}^* \partial \Omega \cup  \gcal.$ Then,
   \begin{multline}
   I^{b}(t) = \int_{\R} \int_{\R} \int_{\R^{n-1}}
\int_{\partial \Omega} \int_{\partial \Omega} U^{b}_{\epsilon}(t +s, q, q') e^{i \mu( \Phi(q', q, \theta) - s)} A_{\epsilon}(q',q ,
\theta,\mu) \\ \nonumber
\times \chi(s - \Phi(q,q',\theta)) \, (2\pi \mu)^{n-1} d\sigma(q') d\sigma(q) d\theta ds d\mu + R(t) + R_{\epsilon}(t), \end{multline}
where $R(t) \in C^{\infty}({\mathbb R}_{t})$  and $ R_{\epsilon}(t) = {\mathcal O}(\epsilon) t^{-n}$ for $|t|$ small. \end{lem}

\begin{proof}
From $L^2$-boundedness of $F(\lambda)$ and the boundary Weyl law
(see \cite{HZ} Lemmas 7.1, 9.2, Appendix 12 as well as Lemma \ref{tanmass} in section \ref{extensions} of this  paper),  it follows that
$$ \frac{1}{N(\lambda)}  \left| \sum_{\lambda_j \leq \lambda}
\langle Op_{\lambda_j}(\chi_{\epsilon}) F(\lambda_j) Op_{\lambda_j}(\chi_{\epsilon}) u_{\lambda_j}^{b}, u_{\lambda_j}^{b} \rangle \right| \leq \frac{C}{N(\lambda)} \sum_{\lambda_j \leq \lambda} \langle Op_{\lambda}(\chi_{\epsilon})^* Op_{\lambda}(\chi_{\epsilon}) u_{\lambda_j}^{b}, u_{\lambda_j}^{b} \rangle = {\mathcal O}(\epsilon).$$
Similarily, by applying Cauchy-Schwartz and an argument like
the one above it follows that the composite operators  $ Op_{\lambda_j}(\chi_{\epsilon}) F(\lambda_j)
Op_{\lambda_j}(1-\chi_{\epsilon}) $ and  $ Op_{\lambda_j}(1-\chi_{\epsilon}) F(\lambda_j) Op_{\lambda_j}(\chi_{\epsilon}) $
each have ${\mathcal O}(\epsilon)$  boundary traces (see also \cite{Z,MO}).

 Thus, substitution of the the integral formula (\ref{FIOkernel}) for
$\hat{F}(s,q,q')$ in the formula for $I^{b}(t)$ gives
$$ \begin{array}{l} \label{start}
I^{b}(t) = \int_{\R} \int_{\partial \Omega} \int_{\partial \Omega} U^{b}(t +s, q, q') \hat{F}(s, q', q) ds
d\sigma(q') d\sigma(q)\\ \\ =  \int_{\R_+} \int_{\R} \int_{\R^{n-1}}
\int_{\partial \Omega} \int_{\partial \Omega}  U_{\epsilon}^{b}(t +s, q, q') e^{i \mu( \Phi(q', q, \theta) - s)} A_{\epsilon}(q',q ,
\theta,\mu)   \, (2\pi \mu)^{n-1} \, d\sigma(q') d\sigma(q) d\theta ds d\mu + R_{\epsilon}(t).
\end{array}
$$
We now cutoff in the $s$-time variable. First, note
that since $A_{\epsilon}$ is symbolic in $\mu,$
$$ \partial_{\mu}^{\alpha} A_{\epsilon}(q,q',\theta,\mu) = {\mathcal O}_{\alpha,\epsilon} ( \mu^{-\alpha} )$$
 uniformly in the $(q,q',\theta)$-variables.

 We make the following decomposition
 \begin{multline} \label{trace1}
 I^{b}(t) = \int_{\R_+} \int_{\R} \int_{\R^{n-1}}
\int_{\partial \Omega} \int_{\partial \Omega} U^{b}_{\epsilon}(t +s, q, q') e^{i \mu( \Phi(q', q, \theta) - s)} A_{\epsilon}(q',q ,\theta,\mu) \\
 \times \chi(s - \Phi(q,q',\theta)) \, (2\pi \mu)^{n-1} d\sigma(q') d\sigma(q) d\theta ds d\mu + R(t) + R_{\epsilon}(t), \end{multline}
where, $\int_{0}^{\infty} e^{i\mu t} R_{\epsilon}(t) dt = {\mathcal O}(\epsilon) \mu^{n-1}$ and
\begin{multline} \label{rem}
R(t) = \int_{\R_+} \int_{\R} \int_{\R^{n-1}}
\int_{\partial \Omega} \int_{\partial \Omega} U^{b}_{\epsilon}(t +s, q, q') e^{i \mu( \Phi(q', q, \theta) - s)} A_{\epsilon}(q',q ,
\theta,\mu) \\
 \times (1-\chi)(s - \Phi(q,q',\theta)) \,(2\pi \mu)^{n-1}  d\sigma(q') d\sigma(q) d\theta ds d\mu. \end{multline}
 By integration by parts in $\mu$ and the fact that $\partial_{\mu}^{\alpha} A = {\mathcal O}_{\alpha}(\mu^{-\alpha})$, it follows that for any $N>0,$
 $$R(t) = \int_{\R_+} \int_{\R} \int_{\R^{n-1}}
\int_{\partial \Omega} \int_{\partial \Omega} U^{b}_{\epsilon}(t +s, q, q') e^{i \mu( \Phi(q', q, \theta) - s)}  \,  A_{\epsilon}(q',q ,
\theta,\mu)  $$
$$\times \, (1-\chi)(s - \Phi(q,q',\theta)) \times {\mathcal O}_{N}( |s|^{-N} \mu^{n-1-N}) \,\,  d\sigma(q') d\sigma(q) d\theta ds d\mu.$$
Since the integrand is absolutely convergent, one can take the $t \rightarrow 0$ limit inside the integral and this shows that $R(t)$
is continuous near $t =0.$ Continuity for the derivatives of $R$ follows a similar fashion:  Since $\partial_{t} U^{b}_{\epsilon}(t+s) = \partial_s U^{b}_{\epsilon}(t+s)$,
 one integrates  by parts in $s$. Each $s$-dertivative of the exponential $e^{i\mu(s-\Phi)}$ creates a power of $\mu$. However, this is
  killed by the ${\mathcal O}_{N}(|s|^{N} \mu^{m-N})$-factor since $N>0$ is arbitrary. The result is an absolutely convergent integral.
   By iterating this argument, one can  differentiate to any order in $t$  inside the integral and the result is absolutely convergent.  It follows  that  $R(t) \in C^{\infty}$ for $|t|$ small and this finishes the proof of the lemma. \end{proof}
Thus, it is enough to consider from now on the principal part of the trace, $I^{b}(t)$, which by Lemma \ref{cutofftrace}, is given by
\begin{multline} \label{trace2}
I^{b}_{sing}(t) :=  \int_{\R_+} \int_{\R^{n-1}} \int_{\R}
\int_{\partial \Omega} \int_{\partial \Omega} U^{b}_{\epsilon}(t +s, q, q') e^{i \mu( \Phi(q', q, \theta) - s)} A_{\epsilon}(q',q ,
\theta,\mu) \\
 \times \chi(s - \Phi(q,q',\theta)) \, (2\pi \mu)^{n-1} d\sigma(q') d\sigma(q) d\theta ds d\mu. \end{multline}

The remainder of the proof consists of carrying out a singularity analysis of
 the integral for $I^{b}_{sing}(t)$ in (\ref{trace2}) near $t =0.$   It is useful to note that since it suffices here to choose $|t| \leq \epsilon_0$ with
$\epsilon_0 >0$ fixed arbitrarily small,  $|t + s| \leq \epsilon_0 + \sup |\Phi(q,q',\theta)| < \infty$ for
 $(q,q',\theta) \in $ supp $ A_{\epsilon}$  since  \, supp $ A_{\epsilon}$ is  compact.   Thus, an immediate consequence of Lemma \ref{cutofftrace} is the identity
 \begin{equation} \label{timecutoffformula}
 \begin{array}{ll}
 I^{b}(t) = I^{b}_{sing}(t) + R_{\epsilon}(t) + R(t),  \\ \\
 I^{b}_{sing}(t) = \, Tr \, \int_{I} U_{\epsilon}^{b}(t+s) \,  (\fcal_{\mu \rightarrow s}F_{\epsilon} )(s) \, ds.  \end{array} \end{equation}  Here, we recall that
 \begin{equation} \label{FIOkernelcutoff}
 F_{\epsilon}(q,q';\mu):= (2\pi \mu)^{n-1} \int_{{\mathbb R}^{n-1}} e^{i \mu \Phi(q,q',\theta)} A_{\epsilon}(q,q',\theta;\mu) \,\, \chi(s- \Phi(q,q',\theta) )\,\, d\theta
 \end{equation}
and for $\epsilon >0$ arbitrarily small, $ R_{\epsilon}(t)  = {\mathcal O}(\epsilon) t^{-n}, R(t) \in C^{\infty}(\mathbb{R}_{t})$ for $|t|$ small
 and $I:=  [ - \|\Phi \|_{L^\infty} - \epsilon_0, \|\Phi \|_{L^{\infty}} + \epsilon_0 ].$

\subsection{Chazarain parametrix}

   Since  $U^{b}_{\epsilon}(t+s)$ is $\lambda$-microlocalized
  away from the set $(B^* \partial \Omega \times [\gcal\cup B_{\Sigma}^*\partial \Omega] ) \cup ( [\gcal \cup B_{\Sigma}^*\partial \Omega ]
  \times B^* \partial \Omega),$  this allows us to substitute  the boundary trace of the  Chazarain parametrix for $U^{b}_{\epsilon}(t+s)$ in  the formula for $I_{sing}^{b}(t)$ in (\ref{trace2}). Here, we use the crucial fact that by Lemma \ref{cutofftrace} (see also (\ref{timecutoffformula})), one need
   only consider finite times with $|s+t| \leq \sup |\Phi| +\epsilon_0.$   This parametrix construction is
    well-known \cite{GM,PS}.  However, as far as we know, the properties vis-a-vis boundary restrction are not available in the literature.  Thus,  we  review here the aspects that are most relevant to the analysis of the boundary traces.

  It suffices to assume that $\partial \Omega$ is $C^{\infty}$ since $ WF'_{\lambda}( U_{\epsilon}^{b}(t+s)
   \circ \hat{F}_{\epsilon}(s) ) \cap  (  [B_{\Sigma}^*\partial \Omega \cup \gcal ] \times B^* \partial \Omega  \cup B^* \partial \Omega \times [ B_{\Sigma}^*\partial \Omega \cup \gcal] ) = \emptyset.$ Let $\partial \Omega = \{ x \in {\mathbb R}^{n}: f(x)= 0$ with $df(x) \neq 0$ for $x \in \partial \Omega.$ Given $(q,\omega) \in S_{+}^{*}(\partial \Omega),$ we let $0 < t_{1}(q,\omega) < t_{2}(q,\omega) < t_{3}(q,\omega) < ...$ with $t_{j} \in C^{\infty}(S^{*}_{+}(\partial \Omega) )$ be the ordered intersection points of the bicharacteristic starting at $(q,\omega)$. Thus, the $t_j; j=1,2....$ satisfy the defining equation
$$ f(t_{j}(q,\omega)) = 0 ; \, j=1,2,3,...$$

  The remainder of the proof of Theorem \ref{weyl2} consists of substituting the boundary trace of the explicit Chazarain parametrix (which we describe below) in (\ref{timecutoffformula}) and then carrying out a singularity analysis of $I^{b}(t)$ near $t =0.$

  To describe the parametrix boundary trace, we decompose $S_{+}^*\partial \Omega = \cup_{k}^{N} \Gamma_k$ where the $\Gamma_k$
are sufficiently small open sets and let $t_{j}^{(k)} = \min_{(q,\omega) \in \Gamma_k} t_j(q,\omega)$
 and $T_j^{(k)}= \max_{(q,\omega) \in \Gamma_k} t_j(q,\omega).$  We choose $\Gamma_k$ small enough so that
\begin{equation} \label{monotone}
t_{1}^{(k)} < T_{1}^{(k)} < t_{2}^{(k)} < T_{2}^{(k)} < \cdots; \,\, k=1,2,3,...,N. \end{equation}
Let $\chi_{k} \in C^{\infty}(S^*\partial \Omega); k=1,...,N,$ be a partition of unity subordinate to the covering by $\Gamma_k; k=1,..,N.$

One constructs a finite-time microlocal parametrix for the wave operator, $U(s),$  on the sets $\Gamma_k; k =1,..,N.$
To describe the construction, consider first the case where $s \in (-2\epsilon_0, t_1^{(k)} - \epsilon_0)$
where $t_1^{(k)}:= \min_{(q,\omega) \in \Gamma_k} t_{1}(q,\omega)$ and $\epsilon_0 >0$ is a fixed
small constant. We  let $$U_{0}^{(k),+}(s,x,y;\lambda): = U_{0} \circ Op_{\lambda}(\chi_k) (s,x,y;\lambda) = (2\pi \lambda)^{n} \int_{{\mathbb R}^{n}} e^{i \lambda \phi (s,x,y,\xi)}  B_{0}^{(k)}(s,y,\xi,\lambda)  d \xi,$$
 where $U_0$ is  the free-space solution of the wave equation with $ \phi(s,x,y,\xi):= \langle x - y, \xi \rangle -
  s |\xi|.$  Let $x^*(x) \in {\mathbb R}^{n}-\Omega$ be the geodesic reflection of $x \in \Omega$ in the boundary
  $\partial \Omega$  and put $$U_{0}^{(k),-}(s,x,y;\lambda): = U_{0} \circ Op_{\lambda}(\chi_{k}) (s,x^*(x),y;\lambda)  =  (2\pi \lambda)^{n} \int_{{\mathbb R}^{n}} e^{i \lambda \phi(s,x^*(x),y,\xi)}  B_{0}^{(k)}(s,y,\xi,\lambda)  d \xi $$
    and  for $(s,x,y) \in (-\epsilon_0, t_1^{(k)}+\epsilon_0) \times \Omega \times \Omega,$ define
\begin{equation} \label{reflection0}
2 U_{1}^{(k)}(s,x,y;\lambda):= U_{0}^{(k),+}(s,x,y;\lambda) + U_{0}^{(k),-}(s,x.y;\lambda). \end{equation}
Clearly, for all $(s,x,y) \in (-2\epsilon_0, t_1^{(k)}-\epsilon_0) \times \Omega \times\Omega,$  $U_{1}^{(k)}$ solves the equation
$ \frac{1}{i} \partial_s U_{1}^{(k)} + \sqrt{\Delta_x} U_1^{(k)} = 0$  and since $x^{*} \in {\mathbb R}^{n} - \Omega,$ by
integration by parts in $\xi$ it follows that for $(x,y) \in \Omega \times \Omega,$
\begin{equation} \label{approxid}
 U_{1}^{(k)}(0,x,y;\lambda) = Op_{\lambda}(\chi_k)(x,y) + R(x,y;\lambda),\end{equation} where $R(x,y;\lambda) \in C^{\infty}$
 with $|\partial_{x}^{\alpha} \partial_{y}^{\beta} R(x,y;\lambda)| = {\mathcal O}_{\alpha,\beta}.(\lambda^{-\infty}).$  Following \cite{HZ}, we call such kernels {\em residual}.
 Simiarily,  since $\partial_{\nu_q}[ \phi(s,x^*,y,\xi) + \phi(s,x,y,\xi) ]|_{x=x^*=q \in \partial \Omega} = 0,$ $U_{1}^{(k)}$
  satisfies the approximate Neumann boundary condition
 \begin{equation} \label{Neumann}
  \partial_{\nu_{q}}U_{1}^{(k)}(s,q,y;\lambda)  = R_{1}(s,q,y;\lambda),   \, \,\,  (s,q,y) \in (-\epsilon_0, t_1^{(k)}+\epsilon_0) \times \partial \Omega \times  \Omega
\end{equation}

   with $R_{1}(s,q,y;\lambda)= {\mathcal O}(\lambda^{-\infty})$
residual.     Also, it is clear that
$$[U_{1}^{(k)}]^{b}(s,q,q';\lambda) = [U_{0}^{(k),\pm}]^{b}(s,q,q';\lambda); \,\, (q,q') \in \partial \Omega \times \partial \Omega.$$
We microlocally cutoff  $U_1^{(k)}$ away from grazing and singular directions and define the microlocal parametrix
\begin{equation} \label{reflection0.1}
U_{1,\epsilon}^{(k)}(s;\lambda) := (Id - Op_{\lambda}(\chi_{\epsilon})  ) \circ U_{1}^{(k)}(s;\lambda)\circ (Id - Op_{\lambda}(\chi_{\epsilon}) ). \end{equation}
The analogous, microlocally-cutoff wave operator is
\begin{equation} \label{wavecutoff}
U_{\epsilon}^{(k)}(s;\lambda) :=  (Id - Op_{\lambda}(\chi_{\epsilon}) ) \circ U(s) \circ Op_{\lambda}(\chi_k) \circ (Id - Op_{\lambda}(\chi_{\epsilon}) ). \end{equation}
 From (\ref{approxid}) and (\ref{Neumann}), both $U_{\epsilon}^{(k)}(s;\lambda)$ and $U_{1,\epsilon}^{(k)}(s;\lambda)$ are microlocal parametrices for the wave operator on supp $\chi_{k}  (1-\chi_{\epsilon}).$ Thus, $U_{1,\epsilon}^{(k)}(s;\lambda) - U_{\epsilon}^{(k)}(s;\lambda)$ is residual for $s \in (-2\epsilon_0, t_{1}^{(k)}-\epsilon_0)$  and in particular,  for $(t+s,q,q') \in (-2\epsilon_0,t_{1}^{(k)}-\epsilon_0) \times \partial \Omega \times \partial \Omega,$
 \begin{equation} \label{parametrix1}
[U_{\epsilon}^{(k)}]^{b}(t+s,q,q';\lambda) = [U_{1,\epsilon}^{(k)}]^{b}(t+s, q,q';\lambda) + R(t+s,q,q';\lambda) \end{equation}
where, $R (\cdot,\cdot,\cdot;\lambda) \in C^{\infty}( \, (-2\epsilon_0, t_1^{(k)}-\epsilon_0) \times \partial \Omega \times \partial \Omega)\, )$ is residual for $\lambda \geq \lambda_0.$

To construct  parametrices up to finite,
multiple reflections at the boundary, we choose
$M = M(k,\epsilon) $ large enough so that the finite
 time-interval, $I$ of integration in (\ref{timecutoffformula})
  satisfies  $I \subset \bigcup_{j=1}^{M}  ( t_{j-1}^{(k)} -
  2\epsilon_0, t_{j}^{(k)} - \epsilon_0  ) \,  \bigcup_{j=1}^{M}
    ( - t_{j}^{(k)} -  2 \epsilon_0 , -t_{j-1}^{(k)} - \epsilon_0) \, $
    with $t_{0}^{(k)} = 0$ and we let $\chi_{\pm j}^{(k)} \in C^{\infty}_{0}(\mathbb{R}); \, j=1,...,M$ be a partition of unity subordinate to this covering.
  In view of (\ref{parametrix1}) the parametrix
  for $s+t \in (-2\epsilon_0,t_{1}^{(k)}- \epsilon_0)$ is $U_{1,\epsilon}^{(k)}(s+t)$. For times $t+s \in (t_1^{(k)}-2\epsilon_0, t_2^{(k)} - \epsilon_0),$ one puts
\begin{equation} \label{parametrix2} \begin{array}{ll}
U_{2,\epsilon}^{(k)}(s+t,x,q;\lambda) =  \int_{\Omega}
U_{1,\epsilon}^{(k)}(s+t-\tau,x,z;\lambda) \cdot   U_{1,\epsilon}^{(k)}(\tau,z,y;\lambda) \, \, dz, \,\,\,\, \tau:= |z-q(y)|.
\end{array}\end{equation}
Clearly, $\frac{1}{i} \partial_{s} U_{2,\epsilon}^{(k)} + \sqrt{\Delta_x} U_{2,\epsilon}^{(k)} =0$
and by (\ref{parametrix1}), $U_{2,\epsilon}^{(k)}(s+t)$ satisfies the Neumann boundary condition.

Moreover, iteration of  this construction  gives the formula
\begin{equation} \label{genparametrix} \begin{array}{lll}
U_{j,\epsilon}^{(k)}(s+t,x,y;\lambda) = \int_{\Omega}
 U_{1,\epsilon}^{(k)}(s+t- \tau,x,z;\lambda) \cdot  U_{j-1,\epsilon}^{(k)}(\tau,z,y;\lambda)  \, dz,  \\ \\
   s+t \in (t_{j-1}^{(k)}-2\epsilon_0, t_{j}^{(k)}-\epsilon_0). \end{array} \end{equation}
Just as in the $j=1$ case, it follows  that for any $j=1,...,M,$
and $s+t \in (t_{j-1}^{(k)}-2\epsilon_0, t_{j}^{(k)}-\epsilon_0),$
\begin{equation} \label{energy2} [U_{j,\epsilon}^{(k)}]^{b}(s+t,q,q';\lambda)  =    U_{\epsilon}^{(k)}(s+t,q,q';\lambda)
 + R_{j}(q,q';\lambda); \,\, (q,q') \in \partial \Omega, \end{equation} with $R_{j}$ residual.
We summarize the parametrix approximation of the boundary trace in the following
 \begin{lem} \label{chazarainapprox}
Given $\chi_{j}^{(k)} \in C^{\infty}_{0}(\mathbb{R}); j=\pm
1,...,\pm M$ as above, it follows from (\ref{timecutoffformula})
and (\ref{energy2}) that for $|t|$ sufficiently small,
$I^{b}_{sing}(t) = \sum_{j,k}[ I_{j}^{(k)}]^{b}(t)$ where,
\begin{equation}
\begin{array}{ll}
  [I_{j}^{(k)}]^{b}(t):=Tr  \int_{I} \chi_{j}^{(k)} (s+t)   \left(  [U_{\epsilon}^{(k)}]^{b}(s+t) \circ \hat{F}_{\epsilon}(s)  \right) \, ds \\ \\
 = Tr  \int_{I}   \chi_{j}^{(k)} (s+t)  \left( [U_{j,\epsilon}^{(k)}]^{b}(s+t) \circ \hat{F}_{\epsilon}(s)  \right) \, ds, \nonumber  \end{array} \end{equation}
where $U_{j,\epsilon}^{(k)}$ is the microlocal parametrix  with Schwartz kernel defined in (\ref{genparametrix})
\end{lem}
To carry out our analysis of the principal singularity of $I_{sing}^{b}(t)$ we use the following  generalized stationary phase lemma (\cite{SV} Proposition 4.1.16):
\begin{lem} \label{SP}
 Given $\Psi(x) \in C^{\infty}({\mathbb R}^{n})$  with critical set
$\text{Crit}(\Psi):= \{ x \in {\mathbb R}^{n}: \nabla_{x} \Psi(x)=0 \}$
and $\alpha(x) \in C^{\infty}_{0}({\mathbb R}^{n}),$
 \begin{equation} \label{sp} \int_{{\mathbb R}^{n}} e^{i\lambda \Psi(x)} \alpha(x) dx
 = \int_{\text{Crit}(\Psi)} e^{i\lambda \Psi(x)} \alpha(x) dx + o(1). \end{equation} \end{lem}

 The computation of the principal singularity of $I^{b}(t)$
amounts to a detailed stationary phase analysis of the explicit
boundary trace of the $\lambda$-microlocalized Chazarain
parametrix given in  Lemma \ref{chazarainapprox} combined with
multiple applications of (\ref{sp}). The analysis is somewhat
lengthy since there are many critical points of the phases in the
parametrix coming from various multilink bicharacteristics. We
thus carry out the analysis systematically for the various links
indexed by the the number of boundary reflections.

\subsection{Single-link contribution to $I^{b}_{sing}(t)$} We start
 here with the computation of $[I_{1}]^{b}(t)$ for small $t$ with
  $t \in (-2 \epsilon_0,  \min_k  t_{1}^{(k)}].$ This is the part of the trace that
   comes from single-link contributions. These  are rays that are either trivial (ie. ones that
   intersect the boundary at  only the initial point) or ones that intersect twice, with the initial
    and terminal points lying on the boundary.  The main result of this section is
\begin{lem} \label{1linktrace}
Under the assumptions in Theorem \ref{weyl2}, the single-link contribution to the boundary wave trace is
$$ I_{1}^{b} (t) = o(t^{-n})$$ for $|t|$ sufficiently small. \end{lem}

 In view of Lemma \ref{chazarainapprox}, substitution of the restricted parametrix $[U_{1,\epsilon}^{(k)}]^{b}$ in $[I_{1}^{(k)}]^{b}(t)$  combined with the polar variables decomposition $\xi = r \omega$ where $(r,\omega) \in [0,\infty) \times {\mathbb S}^{n-1}$  and  the semiclassical rescaling $\xi \mapsto \mu \xi,$ gives the formula
\begin{multline} \label{parametrix1.3}
 [I_{1}^{(k)}]^{b}(t) = \int_{\R_+}  \int_{I}   \int_{\R^{n-1}} \int_{\partial
\Omega} \int_{\partial \Omega}  \int_{\mathbb{S}^{n-1}_{+}} e^{ i \mu   \, [ r    \phi(s + t, q, q', \omega)  + \Phi(q', q, \theta) - s \, ]}  A_{\epsilon} (
q', q, \theta,\mu)   \, B^{(k)}_{\epsilon}(q,q', r \omega,\mu)\\
 \times  \chi(s-\Phi(q',q,\theta)) \chi(r-1) \,r^{n-1} \, dr d\omega
 \, \chi_{1}(s+t) d \theta  d\sigma(q') d\sigma(q) ds  \,  (2 \pi \mu)^{2n-1} \, d\mu + R(t). \end{multline}
Here, $R(t) \in C^{\infty}(\mathbb{R})$ for $|t|$ small and  the cutoff
$\chi(r-1)$ is inserted in (\ref{parametrix1.3}) modulo smooth error by
integration by parts in $s$ using the eikonal equation $\partial_{s} \phi = r
 + {\mathcal O}(t)$ with $|t| \ll 1.$ Also, the integration in  $\omega$ is
  over the hemisphere $S^{n-1}_{+}$ since  $WF_{\lambda}'(U_{\epsilon}^{(k)}(s+t))
  \subset S_{\partial \Omega,+}^* \Omega \times S_{\partial \Omega,+}^*\Omega$ for all $s+t \in I.$ From
  (\ref{parametrix1.3}) it is clear that $[I_{1}^{(k)}]^{b}(t)$ is a  semi-classical (in $\mu$)
oscillatory integral with phase function
$$\Psi( q, q', r,\omega, \theta, s; t)  = r \phi(t + s, q, q', \omega) +   \Phi(q', q,
\theta) - s. $$

We apply stationary phase to (\ref{parametrix1.3}).  Since the integrand in (\ref{parametrix1.3}) is now  compactly
supported in $(s,r)$, one can apply stationary phase in those variables
first. Since
$$  \phi(t+s,q,q',\omega) = \langle q-q',\omega \rangle +(s+t),  $$ the critical point equations are
$$\left\{ \begin{array}{l}  d_{r} \Psi = \phi(t+s,q,q',\omega)
=0,\\ \\
 d_{s}\Psi = r  \partial_{s} \phi(t+s,q,q',\omega) - 1 = 0. \end{array} \right.$$
The non-degeneracy of the phase is clear from the identity
$$ \partial_{r} \partial_{s} \Psi( q, q', r,\omega, \theta, s; t) = \partial_{s}
 \phi (t+s,q,q',\omega) = 1$$ which follows from the eikonal equation. The critical points are
$$  \left\{ \begin{array}{l}
 r=1,  \\ \\
 s(q,q',\omega) = -\langle q-q',\omega \rangle - t. \end{array} \right.$$

The result is that
\begin{multline} \label{parametrix1.4}
\begin{array}{lll}
[I_{1}^{(k)}]^{b}(t) = \int_{\R_{+} } \int_{\R^{n-1}}
  \int_{\partial \Omega}  \int_{\partial \Omega } \int_{ {\mathbb S}_{+}^{n-1} }
  e^{i \mu t}  e^{ i \mu   \, [  \Phi(q', q, \theta) + \langle q-q', \omega \rangle \, ]}   B_0^{(k)}( -\langle q-q',\omega \rangle,
q, q', \omega) \\ \\
\times A_{\epsilon}(q,q',\theta,\mu) \, \,  \chi(- \langle q-q',\omega
 \rangle-\Phi(q',q,\theta)) \, \chi_1(- \langle q-q',\omega \rangle )\,  d\omega  d \theta  d\sigma(q') d\sigma(q)   \, \mu^{2n-2} \, d\mu + R(t)  \\ \\
=\int_{\R_+}  e^{i t \mu}  \, c_{1}^{(k)}(\mu)  \, \mu^{2(n-1)} \, d\mu + R(t).
 \end{array} \end{multline}  We put $c_{1}:= \sum_{k} c_{1}^{(k)}$ and then from the last line of (\ref{parametrix1.4}) it follows that
 \begin{equation}\label{coefficient}
c_{1}(\mu) =  \int_{\R^{n-1}} \int_{\partial \Omega} \int_{\partial \Omega}
 \int_{\mathbb{S}_{+}^{n-1}} e^{ i \mu   \, [  \Phi(q', q, \theta) + \langle q- q', \omega \rangle \, ]}   A_{\epsilon,0}(q, q', \theta) \, \chi_1(-\langle q-q',\omega \rangle ) d\omega  d \theta  d\sigma(q') d\sigma(q) \end{equation} where, $A_{\epsilon}(q,q',\theta;\mu) \sim \sum_{m=0}^{\infty}  A_{\epsilon,m}(q,q,\theta) \mu^{-m}$ and $R(t)$ is only singular of order ${\mathcal O}(t^{-n+1}).$ Here, we have also used that $B_0=1.$

We further decompose $c_1(\mu)$ into  diagonal and off-diagonal parts by introducing a further microlocal cutoff. We write

 \begin{multline}\label{coefficientb}
c_{1}(\mu) =  \int_{\R^{n-1}} \int_{\partial \Omega} \int_{\partial \Omega}
\int_{\mathbb{S}_{+}^{n-1}} e^{ i \mu   \, [  \Phi(q', q, \theta)
+ \langle q- q', \omega \rangle \, ]}
 A_{\epsilon,0}(q, q', \theta) \,\, \chi (\epsilon^{-1}|q-q'|) \\ \times \chi_1(-\langle q-q',\omega \rangle ) \, d\omega  d \theta  d\sigma(q') d\sigma(q) \\ \\
+ \int_{\R^{n-1}} \int_{\partial \Omega} \int_{\partial \Omega}
\int_{\mathbb{S}_{+}^{n-1}} e^{ i \mu   \, [  \Phi(q', q, \theta)
 + \langle q- q', \omega \rangle \, ]}   A_0(q, q', \theta) \,\, (1- \chi) (\epsilon^{-1}|q-q'|)  \\ \times \chi_1(-\langle q-q',\omega \rangle ) \, d\omega  d \theta  d\sigma(q') d\sigma(q, d\omega  d \theta  d\sigma(q') d\sigma(q) \\ \\
 =: c_{1}^{\Delta}(\mu) + c_{1}^{\Delta^c}(\mu). \end{multline}

\subsubsection{Diagonal term}
The term $c_{1}^{\Delta}(\mu)$ sifts out the $\lambda$-pseudodifferential piece of $F(\lambda)$. Indeed, since $|q-q'| \leq \epsilon,$ by Taylor expansion,
$$ \langle q(y)-q(y'), \omega \rangle = \langle T_{y'},\omega \rangle \, (y-y') + {\mathcal O}(|y-y'|^{2}),$$
and so, by making first the change of variables $ r: S^{n-1}_{+}
\rightarrow B^{n-1}, $ given by $ r(\omega)= \omega - \langle
 \nu_{y'}, \omega \rangle \nu_{y'} =: \eta$ followed by another of the form $ \eta \mapsto \eta ( 1 + {\mathcal O}(|y-y'|)),$ it follows that
\begin{equation} \label{psdo part}
c_{1}^{\Delta}(\mu) =  \int_{\R^{n-1}} \int_{\partial \Omega}
\int_{\partial \Omega} \int_{\mathbb{B}^{n-1}} e^{ i \mu   \,
[ y' \theta - \psi(y,\theta) + \langle y-y', \eta \rangle \, ]}   A_{\epsilon,0}(q(y), q(y'), \theta) \, ( 1 + {\mathcal O}(|y-y'|) ) \end{equation}
$$  \times \,\, \chi (\epsilon^{-1}|y-y'|) \, \gamma(\eta)^{-1}  d\eta  d \theta  d\sigma(y') d\sigma(y). $$
Here, we recall that $\gamma(\eta) = \sqrt{1 - |\eta|^{2}}.$ An  application of stationary phase in $(y',\theta)$ implies that the RHS in (\ref{psdo part}) simplifies to
\begin{equation}\label{psdo part2}
 (2\pi \mu)^{-(n-1)}  \int_{\partial \Omega} \int_{\mathbb{B}^{n-1}}
  e^{ i \mu   \, [  \langle y ,\eta \rangle \, - \psi (y,\eta) ]}
    A_{\epsilon,0}(q(y), q(\nabla_{\eta} \psi), \eta) \, ( 1 + {\mathcal O}(|y-\nabla_{\eta}\psi(y,\eta)|) )  \, \gamma(\eta)^{-1} \, dy d\eta. \end{equation}

  The assumption $| \Sigma_{F(\lambda)} | =0$ in  Theorem \ref{weyl2} implies
  in particular that $| \Gamma_{\kappa_F} \cap \Delta| = 0.$ Consequently, the
   generalized stationary phase  estimate (\ref{sp}) applied to the integral in (\ref{psdo part2})
   implies that for the diagonal piece, $c_{1}^{\Delta}(\mu),$ of the single-link contribution to $I^{b}(t)$,
\begin{equation} \label{diagonalpart}
c_{1}^{\Delta}(\mu) = o(\mu^{-(n-1)}).\end{equation}
Thus, it follows that the contribution of the diagonal
 to the trace $I^{b}_{sing}(t)$ is $ \int_{0}^{\infty}
  e^{-it\mu} c_{1}^{\Delta}(\mu)  \mu^{2(n-1)} d\mu = o(t^{-n})$ for $|t|$-small.

\subsubsection{Computation of $c_{1}^{\Delta^c}(\mu)$:}
This piece sifts-out the contribution to $c_1(\mu)$ of single links with initial and terminal points on the boundary. To compute the asymptotics of $c_1^{\Delta^c}(\mu),$ we apply stationary phase in $\omega \in
\in \Ss^{n-1}_{+} .$ The solution of the critical point equation
$$ d_{\omega} \Psi = d_{\omega} \langle q-q',\omega \rangle = 0$$
is the non-degenerate critical point $\omega(q,q') = - \frac{q-q'}{|q-q'|}$ (we note here that the integrand in $c_1(\mu)$ is supported on the set where $|q-q'| \geq \epsilon >0$). There is only one critical point, since the integration   is only over $\omega \in \mathbb{S}_{+}^{n-1}$ corresponding to  directions pointing into the domain $\Omega$.  The result is
 that
 $$c_1^{\Delta^c}(\mu) =  (2\pi \mu)^{-(n-1)/2} \int_{\R^{n-1}} \int_{\partial \Omega} \int_{\partial \Omega} e^{i \mu [ \Phi(q(y'),q(y),\theta) - |q(y)-q(y')|] } \,  A_{\epsilon,0}(q(y),q(y)',\theta) \, ( 1 + {\mathcal O}(\mu^{-1}) )$$
 $$ \times (1-\chi)(\epsilon^{-1}|y-y'|) \, \chi_1( |y-y'|)  d\sigma(y') d\sigma(y) d\theta, $$
where, $\Phi(q(y'),q(y),\theta) = \langle y', \theta \rangle  - \psi(y,\theta).$

Stationary phase in the $y'$-variables gives the critical point equation
$$ d_{y'} \Psi =  \langle T_{y'}, \frac{q(y')-q(y)}{|q(y)-q(y')|} \rangle - d_{y'} \Phi =  \langle T_{y'}, \frac{q(y')-q(y)}{|q(y)-q(y'|)} \rangle - \theta = 0.$$

The non-degeneracy follows from the identity
$$ | \det \nabla_{y'}^{2} [ y' \theta - \psi(y,\theta) - |q(y')-q(y)|] |  = | \det \nabla_{y'}^{2} (|q(y')-q(y)|) | $$
$$ = \left| \det \nabla_{y'}  \langle T_{y'}, \frac{q(y')-q(y)}{|q(y')-q(y)|} \rangle \right|  \geq C_{\partial \Omega}  \,  \langle \nu_y', r(q(y'),q(y)) \rangle   \,   \left( 1 \, + \,  \frac{ \langle \nu_{y'},r(q(y'),q(y)) \rangle }{ |q(y')-q(y)|} \, \right)$$
 where $C_{\partial \Omega} >0.$ The Hessian matrix is non-degenerate since the integral formula
 for $c(\mu)$ is microlocally cutoff away from grazing  directions, ghost directions and unclean points and so, in particular, $ \langle \nu_{y'}, r(q(y'),q(y)) \rangle \geq C(\epsilon)>0$.
  We denote the critical point by $y_c'(y,\theta)$  and let $q_c' = q(y_{c}'(y,\theta))$. Then,
\begin{equation} \label{lastupshot}
c_1^{\Delta^c}(\mu) = (2\pi \mu)^{-n+1}   \int_{\R^{n-1}} \int_{\partial \Omega} \, e^{i \mu [- |q(y) - q'_c| + \Phi(q_c',q(y),\theta) ] }  \,   \frac{ A_{\epsilon,0}(q(y),q'_c,|q(y)-q'_c|)}{ |q(y)-q'_c|^{\frac{1}{2}} } \end{equation}
$$ \times   \langle \nu_{y'}, r(q(y),q'_c) \rangle^{ \frac{-1}{2} }  \, \cdot \,   \left( 1 \, + \,  \frac{ \langle \nu_{y'},r(q(y),q'_c) \rangle }{ |q(y)-q'_c|} \, \right)^{ \frac{-1}{2} }  \chi_1(|y-y'||)
\,   d\sigma(y) \,  d\theta   \,  \,  + R(\mu), $$

where, $R(\mu) = {\mathcal O}(\mu^{-n})$ as $ \mu \rightarrow \infty.$
The
critical set of the phase function $\Psi_1(q(y),\theta)= - |q(y) - q'_c| + \Phi(q_c',q(y),\theta)$ in (\ref{lastupshot}) is
$$\text{Crit}(\Psi_1):=\{ (y,\theta) \in \partial \Omega \times \R^{n-1}; \theta = \langle T_{y'},
\frac{q(y')-q(y)}{|q(y)-q(y')|} \rangle, \,\, d_{\theta} \Phi = 0, \,\,
\langle T_{y}, \frac{q(y)-q(y')}{|q(y)-q(y')|} \rangle = d_{y'} \Phi \}.$$
The assumption $|\Sigma_{F(\lambda)}|=0$ in Theorem \ref{weyl2} implies that
\begin{equation} \label{nomeasure1}
|\text{Crit}(\Psi_1)| = 0. \end{equation} The measure zero condition in (\ref{nomeasure1}) follows by relating
$\text{Crit}(\Psi_1)$ to the Lagrange immersion of the critical set $C_{\Phi}
= \{ d_{\theta} \Phi =0 \}$
  (resp. $C_{\phi_{\tau}} = \{ d_{\omega} \phi_{\tau} = 0 \}$), where $\phi_{\tau}(q,q',\omega):= \phi(\tau,q,q',\omega).$  The Lagrangian
  immersions are given by
$$ \iota_{\Phi}: C_{\Phi} \rightarrow \Gamma_{F}$$
$$ (y,y',\theta) \mapsto (y', d_{y'} \Phi, y, - d_{y} \Phi) = ( y', \theta; y, - d_{y} \psi), \,\,\, d_{\theta} \Phi = 0$$
and
$$ \iota_{\phi_{\tau}}: C_{\phi_{\tau}} \rightarrow \Gamma_{\beta}$$
$$ (y,y', \omega) \mapsto \left( y', \langle T_{y'}, \frac{q(y')-q(y)}{|q(y')-q(y)|} \rangle, y, -\langle T_{y}, \frac{q(y')-q(y)}{|q(y')-q(y)|} \rangle \right),\,\,\, y' = \pi G^{\tau}(y,\omega).$$
It then follows that under the above identification, $\text{Crit}(\Psi_1)$ corresponds to the set
\begin{align} \label{support}
 \Sigma^{(1)}_{F(\lambda)} &= \{ (y,\eta) \in B^* \partial \Omega; \, \beta (y,\eta)
 = \kappa_{F}(y,\eta) \}.
 \end{align}
By assumption, $| \Sigma^{(1)}_{F(\lambda)}| =0$ and so $|\text{Crit}(\Psi_1)|= 0$ also.  Thus from (\ref{lastupshot}) and (\ref{sp}) it follows that $c_{1}^{\Delta^c}(\mu) = o(\mu^{-n-1})$ as $\mu \rightarrow \infty$  and so, $\int_{0}^{\infty} e^{i\mu t} \mu^{2(n-1)} c_{1}^{\Delta^c}(\mu) d\mu = o(t^{-n})$ for $|t|$ small.

\subsection{Multiple-link contribution  to $I^{b}(t)$}

 The main result of this section is the computation of the contribution of higher-order multilinks to $I^{b}(t)$. We prove here
\begin{lem} \label{klinks}
Under the assumption that $|\Sigma_{F(\lambda)}| = 0,$
 $$I_{j}^{b}(t) = o(t^{-n}) \,\,\, \text{ for all} \,\,\, j=1,...,M.$$ \end{lem}

\begin{proof}
 We first compute the contribution to $I^{b}(t)$ given by double links with initial and terminal points pinned on the boundary. The computation for higher-order multilinks is carried out in the same way as for double links.  In view of Lemma \ref{chazarainapprox} the contribution of double links amounts to computing the small $t$  asymptotics of $$I_{2}^{b}(t) =\sum_{k=1}^{M} \int_{\partial \Omega} \int_{I} \chi_{2}^{(k)}(s+t) \left( [U_{2,\epsilon}^{(k)}]^{b}(s+t) \circ \hat{F}_{\epsilon}(s) \right)(q,q) \, ds d\sigma(q).$$  From (\ref{parametrix1}),  wriing $\tau = |z-q(y')|$ the reflected parametrix
  \begin{equation}
\begin{array} {lll} \label{2link}
[U^{(k)}_{2,\epsilon}]^{b}(s+t,q(y),q(y')) = \int_{\Omega} \ U_{1,\epsilon}^{(k)}(s+t-\tau,q(y),z) \circ U_{1,\epsilon}^{(k)}(\tau,z,q(y')) \, dz \,  \\ \\
= (2\pi \mu)^{2n} \int_{\Omega}  \int_{S^{n-1}_+} \int_{S^{n-1}_{+}} \int_{0}^{\infty} \int_{0}^{\infty} e^{i \lambda \Psi(y,y',r_1 \omega_1,r_2\omega_2,z,\tau)  } \, B_{\epsilon}^{(k)}(q(y),q(y'),z,r_1 \omega_1,r_2 \omega_2) \\  \times \chi_k(q(y'),r_2 \omega_2) \chi_k(q(y),r_1 \omega_1) r_{1}^{n-1} r_{2}^{n-1} \, dr_1 dr_2 d\omega_1 d\omega_2  \,  \, dz\\
\end{array} \end{equation}
where,
$$\Psi(y,y',r_1\omega_1,r_2\omega_2,z,\tau):=  r_1 \langle q(y)-z, \omega_1 \rangle + r_2 \langle z -q(y'),\omega_2\rangle -  (s+t -\tau)r_1 - \tau r_2 .$$

In analogy with the single-link case, one applies stationary phase in the spherical variables $(\omega_1,\omega_2) \in S^{n-1}_{+} \times S^{n-1}_{+}.$  The result is that
\begin{multline} \label{2link2}
I_{2}^{(k)}(t) =\int_{0}^{\infty} (2\pi \mu)^{2n}e^{-i\mu t}    \,c_{2}^{(k)}(\mu) \, d\mu,\\
c_{2}^{(k)}(\mu):=  \int_{\Omega}\int_{\mathbb{R}^{n-1}} \int_{\partial \Omega} \int_{\partial \Omega}  \int_{0}^{\infty} \int_{0}^{\infty} \int_{I} e^{i\mu [ r_1 |q(y)-z| + r_2 |q(y')-z| - r_1(s-\tau) - r_2 \tau + s + \Phi(q(y'),q(y),\theta) ]}  A^{(k)}_{\epsilon,0} (q(y),q(y'),\theta) \\   \times B_{\epsilon,0}^{(k)}(q(y),q(y'),z, r_1  r(q(y),z), r_2  r(q(y'),z))
\chi(r_1-1) \, \chi_{2}^{(k)}(s+t) r_1^{n-1} r_2^{n-1}  ds dr_1 dr_2 d\sigma(y) d\sigma(y') d\theta dz.\end{multline}

Now, stationary phase in $(s,r_1)$ gives the critical points $r_1  = 1, s = |q(y)-z| + \tau =  |q(y) - z| + |q(y')-z^*|$ and so, the result is that

\begin{multline} \label{2link3}
c_{2}^{(k)}(\mu) = (2\pi \mu)^{-1}  \int_{\Omega} \int_{\mathbb{R}^{n-1}} \int_{\partial \Omega} \int_{\partial \Omega}  e^{i\mu [ |q(y)-z| + |q(y')-z| + \Phi(q(y'),q(y),\theta) ]}  \tilde{A}^{(k)}_{\epsilon,0} (q(y),q(y'),z,\theta) \\
\times  \chi_{2}^{(k)}(|q(y)-z| + |q(y')-z| )  \, d\sigma(y) d\sigma(y') d\theta dz, \end{multline}
with $\tilde{A}^{(k)}_{\epsilon,0} \in C^{\infty}_{0}(\partial \Omega \times \partial \Omega \times \Omega \times {\mathbb R}^{n-1} ).$
We expand the $z$-dependent part of the phase function,
\begin{equation} \label{reducedphase}
\tilde{\Psi}(z,y,y')  = |q(y') -z| + |z - q(y)| ; \, z \in \Omega, \end{equation} in normal coordinates $(z',z_1)$
  near the boundary, so that $z_1 = 0$ on $\partial \Omega$ and $z_1 > 0$ in the interior of $\Omega.$  Writing $ z = q(z') + z_1 \nu_{z'},$ one makes a Taylor expansion around $z_1 = 0:$
\begin{equation} \label{reducedphase2} \begin{array}{ll}
   \tilde{\Psi}(z,y,y') = |q(y) - z| + |q(y')-z| \\ \\ = |q(y)-q(z')| + |q(y')-q(z')| + \left( \langle \nu_{z'}, \, r(q(y),q(z')) + r(q(y'),q(z')) \rangle \right) \cdot z_1 + {\mathcal O}(z_{1}^{2}). \end{array}\end{equation}
   Since the parametrix is constructed so that the operator wave front propagates along reflected, non-tangential bicharacteristics, it follows that there exists a constant $C(\epsilon)>0$ such that
 \begin{equation} \label{reducedphase3}
   \langle \nu_{z'}, \, r(q(y),q(z')) + r(q(y'),q(z')) \rangle \geq C(\epsilon); \,\, (q(y),q(y'),z) \in \text{supp} \, \tilde{A}_{\epsilon,0}(\cdot,\theta). \end{equation}
  By Taylor expansion of the amplitude, $\tilde{A}_{\epsilon,0}^{(k)},$  in (\ref{2link3}) around $z_1=0 $ and using the inequality (\ref{reducedphase3}), one integrates out the $z_1$ normal variable to the boundary in (\ref{2link3}). This creates an extra power of $\mu^{-1}$ and  the result is that
 \begin{equation}  \begin{array}{ll}\label{2link4}
c_{2}^{(k)}(\mu) = (2\pi \mu)^{-2} \int_{\partial \Omega} \int_{\partial \Omega} \int_{\mathbb{R}^{n-1}} \int_{\partial \Omega}  e^{i\mu [  \, |q(y)-q(z')| + |q(y')-q(z')| + \Phi(q(y'),q(y),\theta) \, ]} \,  a^{(k)}_{\epsilon,0} (q(y),q(y'),q(z'),\theta) \\ \\
\times  \chi_{2}^{(k)}(|q(y)-q(z')| + |q(y')-q(z')| )  \, d\sigma(y) d\sigma(y') d\theta dz', \end{array}\end{equation}
with $a_{\epsilon,0}^{(k)} \in C^{\infty}_{0} (\partial \Omega \times \partial \Omega \times \partial \Omega \times {\mathbb R}^{n-1}).$
 One then  applies stationary phase in the $(y',z')$-variables. Strict convexity ensures non-degeneracy of the Hessian.  The critical point equations are:

\begin{equation}
\left\{ \begin{array}{ll} \label{snell}
\langle T_{y'}, r(q(y'),q(z')) \rangle + d_{y'} \Phi(q(y'),q(y),\theta) = 0, \\  \\
\langle T_{z'}, r(q(z'),q(y')) \rangle = - \langle T_{z'}, r(q(z'),q(y))  \rangle.   \end{array}  \right. \end{equation}

The second set of equations in (\ref{snell}) determines  a locally unique critical point $y'(y,\theta) \in \partial \Omega$  $z'(y,\theta) \in \partial \Omega$  that is a reflection point for the double-link joining $y \in \partial \Omega$ and $y'(y,\theta) \in \partial \Omega.$ The simplified formula is
\begin{multline} \label{2link5}
c_{2}^{(k)}(\mu) = (2\pi \mu)^{-n-1} \int_{\partial \Omega} \int_{\mathbb{R}^{n-1}}  e^{i\mu \Psi_2(y,\theta) }  \\ \times \tilde{a}^{(k)}_{\epsilon,0} (q(y),q(y'(y,\theta)),z'(y,\theta),\theta) \, \, \chi_{2}^{(k)}(|q(y)-z'(y,\theta)| + |q(y'(y,\theta))-z'(y,\theta)| )  \, d\sigma(y) \, d\theta, \end{multline}
where,
$$ \Psi_2(y,\theta) =  |q(y)- q(z'(y,\theta))| + |q(y'(y,\theta))-q(z'(y,\theta))| + \Phi( q(y'(y,\theta)),q(y),\theta).$$
By another application of (\ref{sp}), one supplements the equations in (\ref{snell}) with the critical point equations in $(y,\theta):$
\begin{equation}
\left\{ \begin{array}{ll} \label{snell2}
\langle T_{y}, r(q(y),z') \rangle + d_{y} \Phi(q(y'),q(y),\theta) = 0, \\  \\
d_{\theta} \Phi(q(y),q(y'),\theta) = 0  \end{array} \right. \end{equation}
Let
\begin{multline} \label{critical}
\text{Crit}(\Psi_{2}) := \{ (y,\theta) \in \partial \Omega \times \R^{n-1}; \theta = \langle T_{y},
\frac{z'-q(y)}{|q(y)-z'|} \rangle, \,\, d_{\theta} \Phi = 0, \,\,
\langle T_{y'}, \frac{q(y')-z'}{|q(y')-z'|} \rangle = d_{y'} \Phi, \\  \langle T_{z'}, r(z',q(y')) \rangle = - \langle T_{z'}, r(z',q(y))  \rangle  \}. \end{multline}
Under the Lagrange immersion, $\iota_{\phi_\tau}: C_{\phi_{\tau}} \rightarrow \Gamma_{\beta},$ we have that
$$ \iota_{\phi_{\tau}} ( \text{Crit}(\Psi_{2})) = \Sigma_{F(\lambda)}^{(2)},$$  where
$$ \Sigma_{F(\lambda)}^{(2)} = \{ (y,\eta) \in B^* \partial \Omega; \beta^{2}(y,\eta) = \kappa_{F}(y,\eta) \},$$
and so, the assumption that $|\Sigma_{F(\lambda)}^{(2)}| = 0$ implies and $|\text{Crit}(\Psi_{2})| = 0.$ Summation over the microlocal partition  (ie. summing in $k=1,...,M$) and an application of (\ref{sp}) then implies  that $c_{2}(\mu) = o(\mu^{-n-1})$ and so,
$ I_{2}^{b}(t) = \int_{0}^{\infty} e^{i\mu t} (2\pi \mu)^{2n} c_{2}(\mu) d\mu = o(t^{-n}) \,\, \text{as} \,\,   t \rightarrow 0.$

The computation for arbitrary multlinks follows in the same way as for double links and this finishes the proof of Lemma \ref{klinks}. \end{proof}

 Since $I^{b}(t) = \sum_{j=1}^{M} I_{j}^{b}(t) + R_{\epsilon}(t) + R(t)$ where $R(t)$ is locally smooth  and $
R_{\epsilon}(t) = {\mathcal O}(\epsilon) t^{-n}$ near $t = 0.$  Since $\epsilon >0$ in Lemma \ref{cutofftrace} can be taken arbitrarily small,  by letting $\epsilon \rightarrow 0^+$ and
applying  the Fourier Tauberian theorem, this completes the proof
of Theorem  \ref{weyl2} under the assumption that $\partial
\Omega$ is strictly-convex.

\subsection{WKB Case} \label{WKB}

In our application to QER, the semi-classical Fourier integral
operators $F(\lambda)$ which arise  are naturally represented in a
somewhat different form with Schwartz kernels of
 WKB-type. We now modify the proof of Theorem \ref{weyl2} to apply more
 directly to operators of this form.
 The proof is very similar to the one given in Theorem \ref{weyl2} but simplifies slightly.

 Consider the special case where $F(\lambda)$ has Schwartz kernel
$$ F(\lambda) (q(y),q(y')) =  \lambda^{\frac{n-1}{2}}  e^{i \lambda \psi(q(y),q(y'))} B(q(y),q(y'),\lambda),$$
where $B(q(y),q(y'),\lambda) \sim_{\lambda \rightarrow \infty}
\sum_{k=0}^{\infty} B_k(q(y),q(y')) \lambda^k$ with $B_{k} \in
C^{\infty}(\partial \Omega \times \partial \Omega)$ and  supp $B_k
\cap \Delta_{\partial \Omega \times \partial \Omega} = \emptyset,
\,\, k=0,1,2,...$. The Fourier transform
 $$\hat{F}(s)(q(y),q(y')) =   \int  \mu^{\frac{n-1}{2}}e^{i \mu( \psi(q(y),q(y')) - s)}  \, B(q(y),q(y'),\mu) \, d\mu. $$

To determine the leading $(t+i0)^{-n}$-singularity of  $I^{b}(t) = Tr \,  \int_{\R} U^{b}(t + s) \hat{F}(s) ds, $ one repeats the
argument in Theorem  \ref{weyl2}, except that there is no
$\theta$-integration at the end and one does not carry out
stationary phase in $y'$. Rather, one directly computes that the
single link contribution to the $(t+i0)^{-n}$-singularity is given
by $\int_{0}^{\infty} e^{i\mu t} c_1(\mu) d\mu $  where,
\begin{equation} \label{ firstlink}  \begin{array}{lll}
c_1(\mu) = (2\pi \mu)^{n-1} \int_{\partial \Omega} \int_{\partial \Omega}  \, e^{i \mu [ \psi(q(y),q(y'))  - |q(y)-q(y')|] } \frac{ B_0(q(y),q(y')) }{ |q(y)-q(y')|^{\frac{n-1}{2}} }    d\sigma(y) d\sigma(y')  + {\mathcal O}(\mu^{n-2}) \\ \\
=  \int_{\partial \Omega} \int_{\partial \Omega}
F(\mu)(q(y),q(y')) \, [ (2\pi \mu) \, G_0(\mu)(q(y'),q(y))] \, d\sigma(y)
d\sigma(y') + {\mathcal O}(\mu^{n-2}),\end{array} \end{equation}
where, $G_0(\mu)$ is the free Greens function in (\ref{Hankel}).

The $k$-link contributions for $|k|\geq 2$ are computed in a
similar way by inserting the reflection parametrix terms. The
result is that
\begin{equation} \label{k link}
c_{k}(\mu)= \int_{\partial \Omega} \int_{\partial
\Omega} F(\mu)(q(y),q(y')) \, [ (2\pi \mu) \, G_0(\mu)]^{k}(q(y'),q(y)) \,
d\sigma(y) d\sigma(y') + {\mathcal O}(\mu^{n-2}).\end{equation}

 As in the proof of Theorem \ref{weyl2}, one shows that the critical point equations in $(y,y')$ are equivalent  to the fixed point equations
$ \beta^{k} (y, \eta) = \kappa_{F}(y,\eta); \,\, (y,\eta) \in
B^*\partial \Omega$ where $(y,\eta; \kappa_{F}(y,\eta)) = (y,
d_{y} \psi(q(y),q(y')); y', - d_{y'} \psi(q(y),q(y') )).$   It
follows from (\ref{sp})  that in the sense of distributions,
\begin{equation} \label{big1}
 I^{b}(t) = o(t^{-n}) \end{equation}
under the measure zero assumption, $| \Sigma_{F(\lambda)}| = 0.$

\subsection{Non-convex boundary} \label{nc}

In the non-convex case, the canonical relation,  $\Gamma_{N^{\partial \Omega}_{\partial \Omega}} $, of the boundary jumps operator is larger than  graph of the billiard map, $\Gamma_{\beta}$. This is due to "spurious" geodesics that exit $\Omega$ at a boundary point $y \in \partial \Omega$, only to re-enter at another point $y' \in \partial \Omega.$  Following \cite{HZ}, we  say that $(y,\eta,y',\eta') \in B^* \partial \Omega \times B^* \partial \Omega$ with $\eta = d_{y}|q(y)-q(y')|$ and $\eta'= - d_{y'}|q(y)-q(y')|$ is a spurious point of $WF'_{\lambda}(N^{\partial \Omega}_{\partial \Omega}(\lambda))$ if the line $\overline{q(y)q(y')}$ leaves $\overline{\Omega}.$
It is shown in \cite{HZ} that these geodesics can be ignored for the purposes for spectral asymptotics of boundary traces of eigenfunctions. We briefly recall the argument here. Further details can be found in \cite{HZ} section 11.

Let $b$ be a smooth, compactly-supported non-negative function on ${\mathbb R}^{n}$ which vanishes on $\overline{\Omega}.$ Consider the metric
$$ g_{s} = (1 + s b)g_{\text{Euclidean}}; \,\,\, s \in [0,\delta],$$
with $\delta >0$ small and the associated resolvent  kernel $G_{s} (z,z';\lambda) = ( \Delta_s  - (\lambda + i 0)^{2})^{-1}(z,z')$ of the $g_s$-Laplacian on ${\mathbb R}^{n}$. The associated  boundary jumps operator is
$$ [N_s]^{\partial \Omega}_{\partial \Omega} (\lambda)(q,q') =2 \frac{\partial}{\partial \nu_q} G_s(q,q',\lambda); \,\,\, q,q' \in \partial \Omega.,$$
and one forms  the averaged operator
\begin{equation} \label{nc2}
\tilde{N}^{\partial \Omega}_{\partial \Omega}(\lambda):= \int_{0}^{1} \chi(s)  [N_s]^{\partial \Omega}_{\partial \Omega} (\lambda) \, ds,\end{equation}
where $\chi^{\infty}_{0}( (0,\delta))$ with $\chi \geq 0$ and $\int_{0}^{\delta} \chi = 1.$  The averaged operator, $\tilde{N}^{\partial \Omega}_{\partial \Omega}(\lambda)$ satisfies two key properties (see \cite{HZ} section 11):

\begin{equation} \begin{array}{ll} \label{prop}
(i) \tilde{N}^(\lambda) u_{\lambda}^{b} = u_{\lambda}^{b}, \\ \\
(ii) WF_{\lambda}'( \tilde{N}(\lambda)) \subset \Gamma_{\beta}. \end{array} \end{equation}

The first condition follows from Greens formula since the metric $g_{s}$ stays Euclidean in $\overline{\Omega}$ and the second, by an integration by parts argument in $s \in [0,\delta]$ (see \cite{HZ} section 11). Now, let $\chi^{sp}_{\epsilon} \in C^{\infty}_{0}(B^*\partial \Omega)$ with $\chi^{sp}_{\epsilon}= 0$ in an $\epsilon$-neighbourhood of set of spurious directions and $\chi^{sp}_{\epsilon} = 1$ outside a $2\epsilon$-neighbourhood of the same set. Then, from (\ref{prop})

\begin{equation} \label{sp2} \begin{array}{lll} \frac{1}{N(\lambda)} \sum_{\lambda_j \leq \lambda} \langle F(\lambda_j) u_{\lambda_j}^{b}, u_{\lambda_j}^{b} \rangle =  \frac{1}{N(\lambda)} \sum_{\lambda_j \leq \lambda} \langle \tilde{N}(\lambda_j) F(\lambda_j) \tilde{N}(\lambda_j)u_{\lambda_j}^{b}, u_{\lambda_j}^{b} \rangle  \\ \\
= \frac{1}{N(\lambda)} \sum_{\lambda_j \leq \lambda} \langle \tilde{N}(\lambda_j) Op_{\lambda_j}(\chi^{sp}_{\epsilon}) F(\lambda_j) Op_{\lambda_j}(\chi^{sp}_{\epsilon}) \tilde{N}(\lambda_j)u_{\lambda_j}^{b}, u_{\lambda_j}^{b} \rangle + {\mathcal O}(\epsilon) \\ \\ = \frac{1}{N(\lambda)} \sum_{\lambda_j \leq \lambda} \langle  Op_{\lambda_j}(\chi^{sp}_\epsilon) F(\lambda_j) Op_{\lambda_j}(\chi^{sp}_{\epsilon}) u_{\lambda_j}^{b}, u_{\lambda_j}^{b} \rangle + {\mathcal O}(\epsilon).
 \end{array}  \end{equation}
In the second last line of (\ref{sp2}), one uses boundary
ergodicity and Cauchy-Schwartz to get the ${\mathcal O}(\epsilon)$
error term.  We replace $F(\lambda_j)$ by $
Op_{\lambda_j}(\chi^{sp}_\epsilon) F(\lambda_j)
Op_{\lambda_j}(\chi^{sp}_{\epsilon}) $ in the Weyl sum and repeat
the argument in the strictly convex case. Finally,  we recall that
by  removing the cutoff $\chi_{\epsilon}$ supported in
$\epsilon$-tubes around the  grazing and singular set $\gcal
\cup B_{\Sigma}^*\partial \Omega $, we obtain an ${\mathcal O}(\epsilon)$-error. As in
\cite{Z,MO} we can let $\epsilon \to 0$ to complete the proof of
Theorem  \ref{weyl2}.  \end{proof}

\begin{rem} We may also consider the singularity at any other
value $t = t_0$ of
$$\sum_j e^{i t \lambda_j} \langle F(\lambda_j) u_{\lambda_j}^{b}, u_{\lambda_j}^{b}
\rangle. $$
 It gives the growth rate of
$$\sum_{j: \lambda_j \leq \lambda}e^{i t_0 \lambda_j}  \langle F(\lambda_j) u_{\lambda_j}^{b}, u_{\lambda_j}^{b}
\rangle $$ As an example, when $F(\lambda) = N_{\partial
\Omega}^{\partial \Omega}(\lambda)$ and under the non-looping
assumption, it follows by the argument in Theorem  \ref{weyl2}
that for $t_{0} \neq 0,$  $|\Sigma_{F(\lambda)}| =0$ and so,

$$ \frac{1}{N(\lambda)} \sum_{\lambda_j \leq \lambda} e^{it_0 \lambda_j}
 \| u_{\lambda_j}^{b} \|_{L^2(\partial \Omega)}^2 = o(1) \,\, \mbox{as} \,\, \lambda \rightarrow \infty.$$

\end{rem}

\section{\label{LWLH} Local Weyl law on $H$: calculation of
$\omega_{\infty}$}

We use the local Weyl law for Fourier integral operators to
calculate the limit state $\omega_{\infty}(a)$. It is
provisionally defined  on $C^{\infty}(B^* H)$ by
\begin{equation} \label{OMEGAINF} \omega_{\infty}(a) = \lim_{\lambda \to \infty} \frac{1}{N(\lambda)}
\sum_{j: \lambda_j \leq \lambda} \langle Op_{\lambda_j}(a)
\phi_{\lambda_j}|_{H},\phi_{\lambda_j}|_{H} \rangle_{L^{2}(H)}.
\end{equation}
We now show that the limit exists and we calculate it. We use the
notation,
    $$\lim_{\lambda \to \infty} \frac{1}{N(\lambda)} \sum_{\lambda_j \leq \lambda}
    \langle Op_{\lambda}(a)
    u_{\lambda_j}^{b}, u_{\lambda_j}^{b} \rangle     \omega_{\infty}^{\partial \Omega} (a)$$
    for the limit measure in the boundary case. This section of
    course depends on the choice of $H$, but the cutoff estimates
    above show that the singular and tangential sets, hence the
    convexity of $H$, is irrelevant to the result.

\begin{lem} The limit state $\omega_{\infty}$ is well defined
and is given by
 $$ \omega_{\infty}(a) = \omega_{\infty}^{\partial \Omega} (\tau_{H}^{\partial \Omega})^* a
 \rho_H^{\infty}. $$
 \end{lem}

 \begin{proof}

We prove the Lemma by reducing the calculation to the local Weyl
law along the boundary in Lemma 1.2 of \cite{HZ}.

    By (\ref{CTOB}), we have
 \begin{equation} \label{CTOBaa} \omega_{\infty}^H(a) = \lim_{\lambda \to \infty} \frac{1}{N(\lambda)} \sum_{\lambda_j \leq \lambda} \langle N^{H}_{\partial \Omega}(\lambda_j)^* Op_{\lambda_j}(a)
   N_{\partial \Omega}^{H}(\lambda_j) u_{\lambda_j}^{b}, u_{\lambda_j}^{b} \rangle_{L^2(\partial \Omega)}. \end{equation}
By  (\ref{decomposition}), the limit in (\ref{CTOBaa}) is the sum
of a pseudo-differential Weyl law term and a Fourier integral
term. By Egorov's theorem  (see Proposition \ref{Egorov}),  the
limit of the pseudo-differential term is
    $$\frac{1}{N(\lambda)} \sum_{\lambda_j \leq \lambda}
    \langle  Op_{\lambda_j}((\tau_H^{\partial \Omega})^* a \rho_H^{\infty})
    u_{\lambda_j}^{b}, u_{\lambda_j}^b \rangle \to \omega_{\infty}^{\partial \Omega} (\tau_H^{\partial \Omega})^* a  \rho_H^{\infty}.$$
 To complete the proof, we need to calculate the limit of the
 Fourier integral term. We claim that
    $$\lim_{\lambda \to \infty} \frac{1}{N(\lambda)} \sum_{\lambda_j \leq \lambda}
    \langle  F_2(\lambda_j;a,\epsilon)
    u_{\lambda_j}^{b}, u_{\lambda_j}^b \rangle \to 0.$$
    By Theorem \ref{weyl2}, it suffices to show that the
set of points where the broken billiard map $\beta_{H}$ equals a power of $\beta$ has measure zero. However,
this is special case  of our non-commutativity assumption: If
$\beta_H(\zeta) = \beta^k(\zeta)$  then
$\beta^{\ell}  \beta_{H}(\zeta) = \beta^{k +
\ell}(\zeta)$. Also, if $\beta_{H}(\beta^{\ell}
\zeta) = \beta^{k} (\beta^{\ell} \zeta), $ then both equal
$\beta^{k + \ell}(\zeta)$. Hence, $\beta^{\ell}  \beta_{H}(\zeta) =  \beta_{\partial \Omega}^{H}(\zeta)
\beta^{\ell}$ on a set of positive measure if $\beta_H(\zeta) = \beta^k(\zeta)$  on a set of positive measure
and if, for some $\ell \not= 0$, $\beta_{H}(\beta^{\ell} \zeta) = \beta^k(\beta^{\ell} \zeta)$ on
a set of positive measure. We claim that the first condition
implies the second: let the first set be called $B$. Then
$\beta^{\ell} B \cap B$ must have positive measure for some $\ell$
by Poincar\'{e}'s recurrence theorem or by ergodicity of $\beta$.

\end{proof}

\section{\label{CRAZY} Quantum ergodic restriction: Proof of Theorem \ref{maintheorem}}

We now have the ingredients to prove the main result. As mentioned
in section \S \ref{DEC},  we prove Theorem \ref{maintheorem} first
for $ \lambda$-pseudodifferentials $Op_{\lambda}(a) \in
Op_{\lambda}(S^{0,0}_{cl} (T^*H \times [0,\lambda_0^{-1}] ) )$
with supp  $ a \,  \cap \tau_{\partial \Omega}^{H} ( \, \mathcal T
\cup B_{\Sigma}^*\partial \Omega \cup \gcal \, ) = \emptyset; k =1,2 $.
 In \S \ref{extensions} we show that for a full-density of egenfunctions, $L^{2}$-mass is not
  conentrated on the set $\tau_{\partial \Omega}^{H}( B_{\Sigma}^*\partial \Omega \cup \tcal \cup \gcal ) $
   (see \cite{HZ} for an alternative proof of this fact using the boundary trace of the
    Neumann heat kernel). Using this fact, in \S \ref{extensions}  we extend the proof of Theorem \ref{maintheorem}
     to arbitrary symbols $a \in S^{0,0}_{cl}(T^*H \times [0,\lambda_0^{-1}]).$  We therefore assume here that the
      symbol $a \in S^{0,0}_{cl}(T^*H \times [0,\lambda_0^{-1}])$
satisfies the following support condition (\ref{cutaway}) with
arbitrarily small but fixed $\epsilon >0.$

We now reduce the proof of Theorem \ref{maintheorem} to Theorem
\ref{VARIANCE}.

\subsection{Reduction to Theorem \ref{VARIANCE}}

To prove quantum ergodicity for the eigenfunction restrictions
 $ u_{\lambda_j}^{H} = \phi_{\lambda_j}|_{H}, j=1,2,...$
we put $F(\lambda) = N_{\partial \Omega}^H(\lambda)^{*}
Op_{\lambda}(a) N_{\partial \Omega}^{H}(\lambda)$ and use
(\ref{decomposition}) to compute the variance   for the
eigenfunction restrictions, $u_{\lambda_j}^{b}.$
$$ \frac{1}{N(\lambda)} \sum_{\lambda_j \leq \lambda} | \langle Op_{\lambda_j}(a)
\phi_{\lambda_j}|_H, \phi_{\lambda_j}|_{H} \rangle - \omega_{\infty}(\sigma(F))|^{2} $$
$$ = \frac{1}{N(\lambda)} \sum_{\lambda_j \leq \lambda} | \langle
[F(\lambda_j)- \omega_{\infty}(\sigma (F) )] u_{\lambda_j}^{b}, u_{\lambda_j}^{b} \rangle |^{2} + o(1)$$

$$ = \frac{1}{N(\lambda)} \sum_{\lambda_j \leq \lambda} | \langle \,
[ \, F_{1}(\lambda_j;a,\epsilon)- \omega_{\infty}(\sigma (F) ) \, + F_{2}(\lambda_j;a,\epsilon) \, ] u_{\lambda_j}^{b}, u_{\lambda_j}^{b} \rangle \  |^{2} + o(1).$$

\begin{equation} \label{variance}
 \leq \frac{2}{N(\lambda)} \sum_{\lambda_j \leq \lambda}   | \, \langle [F_1(\lambda_j;a,\epsilon) -  \omega_{\infty}( \sigma_{\Delta}(F) ) ]  u_{\lambda_j}^{b}, u_{\lambda_j}^{b} \rangle  \, |^{2} \\
+  \frac{2}{N(\lambda)} \sum_{\lambda_j \leq \lambda}   | \,
\langle F_2(\lambda_j;a,\epsilon)  u_{\lambda_j}^{b},
u_{\lambda_j}^{b} \rangle  \, |^{2} + o(1). \end{equation}
where the limiting state $\omega_{\infty}(\sigma(F))$ was
determined in \S \ref{LWLH}.

 The second line in (\ref{variance}) follows from the special case
 of the boundary Weyl law of \cite{HZ} that says that $\frac{1}{N(\lambda)}
  \sum_{\lambda_j \leq \lambda} \| u_{\lambda_j} \|_{L^{2}}  \sim 1$. From the same asymptotic
  formula, it follows by
   the Cauchy-Schwartz inequality that the first term in the last line of (\ref{variance}) is bounded by
\begin{equation} \label{main}
 \frac{1}{N(\lambda)} \sum_{\lambda_j \leq \lambda} \langle [F_{1}(\lambda_j;a,\epsilon) - \omega_{\infty} (\sigma(F))]^* \circ  [F_{1}(\lambda_j;a,\epsilon) - \omega_{\infty}(\sigma(F))] u_{\lambda_j}, u_{\lambda_j} \rangle  + o(1).\end{equation}

We know by Proposition \ref{Egorov} that for $ \epsilon >0$ sufficiently small,
$$F_{1}(\lambda;a,\epsilon) \in Op_{\lambda}(S^{0,0}_{cl}(T^* \partial
\Omega)).$$

 By Theorem  \ref{weyl2}  and  boundary quantum ergodicity for $\lambda$-pseudodifferentials \cite{HZ}, it follows
that for $\epsilon >0$ small,
\begin{equation} \label{pseudo result}
 \frac{1}{N(\lambda)} \sum_{\lambda_j \leq \lambda} \langle [F_{1}(\lambda_j;a,\epsilon) - \omega_{\infty} (\sigma(F(\lambda)))]^* \circ  [F_{1}(\lambda_j;a,\epsilon) - \omega_{\infty}(\sigma(F))] u_{\lambda_j}^{b}, u_{\lambda_j}^{b} \rangle  =  o_{\lambda}(1).\end{equation}
 As for the second term in (\ref{variance}), by making the further decomposition $F_{2}(\lambda_j;a,\epsilon) = F_{21}(\lambda_j;a,\epsilon) + F_{22}(\lambda_j;a,\epsilon)$ with $\Gamma_{F_{2k}(\lambda;\epsilon)} = \text{graph} \,  \beta_{H}^{k}; \, k=1,2,$ it follows  from the inequality $|a+b|^{2} \leq 2(a^{2} + b^{2})$ that
 \begin{equation} \label{second term}
\frac{2}{N(\lambda)} \sum_{\lambda_j \leq \lambda}   | \, \langle F_2(\lambda_j;a,\epsilon)  u_{\lambda_j}^{b}, u_{\lambda_j}^{b} \rangle  \, |^{2}  \end{equation}
$$\leq \frac{4}{N(\lambda)} \sum_{\lambda_j \leq \lambda}  |  \langle F_{21}(\lambda_j;a,\epsilon) u_{\lambda_j}^{b}, u_{\lambda_j}^{b} \rangle |^{2}   \, + \frac{4}{N(\lambda)} \sum_{\lambda_j \leq \lambda}   | \langle F_{22}(\lambda_j;a,\epsilon)  u_{\lambda_j}^{b}, u_{\lambda_j}^{b} \rangle |^{2}.$$

The  remaining step is the proof of Theorem \ref{VARIANCE}.

\subsection{\label{VARPROOFSECT} Proof of Theorem \ref{VARIANCE}}

We now prove a special case of Theorem \ref{VARIANCE} where
$\kappa_F = \beta_H^1$ \, (resp. $\beta_H^2$), and where
$F(\lambda) = F_21(\lambda;a,\epsilon)$ \, (resp.
$F_{22}(\lambda;a,\epsilon).$  We re-state it for the sake of
clarity in the special case.

\begin{theo} \label{MEFIO} Let $F_{21}(\lambda;a,\epsilon) $ be a semi-classical Fourier integral
operator associated to the canonical relation $\beta_H^1$.  Assume
that $(\beta_H, \beta)$ almost never commute. Then
$$\lim_{\lambda \to \infty}
\frac{1}{N(\lambda)} \sum_{\lambda_j \leq \lambda}   | \, \langle
F_{21}(\lambda_j;a,\epsilon)  u_{\lambda_j}^{b}, u_{\lambda_j}^{b}
\rangle  \, |^{2}= 0. $$ Similarly for $F_{22}(\lambda;a,\epsilon)$. \end{theo}

The proof of the  general case is essentially the same. As
discussed in the introduction, the  proof is based on
time-averaging as in (\ref{AVE}).

\begin{proof}

 In analogy with the boundary case treated in \cite{HZ}  we define the $M$-th order time-averaged operator:
 $$ \langle F_{2k}(\lambda;a,\epsilon) \rangle_{M} := \frac{1}{M}
 \sum_{m =1}^{M}[ N_{\partial \Omega}^{\partial \Omega}(\lambda)^{*}]^{m}
 F_{2k}(\lambda;a,\epsilon) [N_{\partial \Omega}^{\partial \Omega}(\lambda)]^{m}; \,\, k=1,2,$$
 From the boundary jump equation $N_{\partial \Omega}^{\partial \Omega}(\lambda) u_{\lambda}^{b} =  u_{\lambda}^{b},$  we have for $k=1,2,$
 $$ \frac{4}{N(\lambda)} \sum_{\lambda_j \leq \lambda}    | \langle F_{2k}(\lambda_j;a,\epsilon) u_{\lambda_j}^{b},
  u_{\lambda_j}^{b} \rangle \, |^{2}  =  \frac{4}{N(\lambda)} \sum_{\lambda_j \leq \lambda}    | \langle  \,
   \langle F_{2k}(\lambda_j;a,\epsilon) \rangle_M \,  u_{\lambda_j}^{b}, u_{\lambda_j}^{b} \rangle \, |^{2}.$$
By the Cauchy-Schwartz inequality (combined with the local  Weyl
law $\frac{1}{N(\lambda}) \sum_{\lambda_j \leq
   \lambda} \|u_{\lambda_j}^{b} \|^{2}$ $\sim_{\lambda \rightarrow \infty}
   1$), we have
 \begin{equation} \label{f3}
\frac{1}{N(\lambda)} \sum_{\lambda_j \leq \lambda}  \, |   \langle \,
\langle F_{2k}(\lambda_j;a,\epsilon) \rangle_M \,  u_{\lambda_j}^{b}, u_{\lambda_j}^{b}  \rangle |^{2}
  \leq  \frac{1}{N(\lambda)} \sum_{\lambda_j \leq \lambda}  \langle \,\,
   \langle F_{2k}(\lambda_j;a,\epsilon) \rangle_{M}^* \,  \langle F_{2k}(\lambda_j;a,\epsilon) \rangle_{M}  \,
    u_{\lambda_j}^{b}, u_{\lambda_j}^{b} \, \rangle + o(1). \end{equation}

    \begin{rem} This use of the Cauchy-Schwartz inequality may
    seem
    rather un-natural. We are trying to prove that $\langle F_{2k}(\lambda_j;a,\epsilon) \rangle_{M}  \,
    u_{\lambda_j}^{b}, u_{\lambda_j}^{b} \, \rangle \to 0$
    as $\lambda_j \to \infty$, i.e. that $ F_{2k}(\lambda_j;a,\epsilon) \rangle_{M}  \,
    u_{\lambda_j}^{b}$ is asymptotically orthogonal to
    $u_{\lambda_j}^{b}$. Use of the Cauchy-Schwartz inequality
    converts the orthogonality condition to the condition that the average over the spectrum of the norms  $||
    \langle F_{2k}(\lambda_j;a,\epsilon) \rangle_{M}  \,
    u_{\lambda_j}^{b}|| \to 0$ is of order $O(\frac{1}{M})$.
    \end{rem}

 It  suffices to bound the averaged matrix-elements on the right side of (\ref{f3}).
Writing out these terms explicitly, one gets
\begin{equation} \label{upshot1} \begin{array}{l}
 \frac{1}{N(\lambda)} \sum_{\lambda_j \leq \lambda}  \langle F_{2k}(\lambda_j;a,\epsilon) \rangle_{M}^* \,
  \langle F_{2k}(\lambda_j;a,\epsilon) \rangle_{M}  u_{\lambda_j}^{b}, u_{\lambda_j}^{b}\rangle
  \\ \\
 = \frac{1}{M^2 N(\lambda)} \sum_{m, n=0}^{M} \, \sum_{\lambda_j \leq \lambda} \\ \\
 \cdot \;\;\;\langle
 \,\, [N_{\partial \Omega}^{\partial \Omega}(\lambda_j)^*]^{m}\,
  F_{2k}(\lambda_j;a.\epsilon)^{*}\, [N_{\partial \Omega}^{\partial \Omega}(\lambda_j)]^{m}
   \,[N_{\partial \Omega}^{\partial \Omega}(\lambda_j)^{*}]^{n}\, F_{2k}(\lambda_j;a,\epsilon)\,
    [N_{\partial \Omega}^{\partial \Omega}(\lambda_j)]^{n} u_{\lambda_j}^{b}, u_{\lambda_j}^{b} \rangle
    \\ \\
 = \frac{1}{M^2} \, \sum_{  m \neq n, m, n = 0  } \mbox{same}
 + \frac{1}{M^2} \sum_{m = n = 0}^{M} \mbox{same} \end{array} \end{equation}
In (\ref{upshot1}), we have broken up the double sum into the
diagonal and off-diagonal terms.

 Due to the time-averaging over $m,n$ in the sums in (\ref{upshot1}) it is useful to first slightly refine the microlocal cutoff condition in (\ref{cutaway}).  We first prove  (\ref{upshotdiagonal}) and Proposition \ref{qaclemma} for this somewhat more restrictive class of symbols and then apply a density argument using the pointwise Weyl law argument in Lemma \ref{tanmass} to deal with the general case.

  We assume that  $|m-n| \leq 2M; m,n \in {\mathbb Z}$ and fix $\epsilon >0.$ We say that a link is $\epsilon$-transversal at the boundary if it makes an angle of at least $\epsilon$ to the tangent plane at each point of intersection with $\partial \Omega.$ We define the extremal link-lengths:

\begin{multline} \label{hlink}
 L_H(M,\epsilon):= \max_{|m-n| \leq 2M} \sup_{(y,\eta) \in \Omega_{\epsilon}^{M}}   ( \, | \pi \beta_H^{-1} \beta^{m-n} \beta_H(y,\eta) - \pi \beta^{m-n}\beta_H(y,\eta)| \\   + \sum_{k=1}^{m-n} | \pi \beta^{k} \beta_H(y,\eta)-   \pi \beta^{k-1} \beta_H (y,\eta)| + | \pi \beta_H(y,\eta) - y|  \, ),\end{multline}
where, $\Omega_{\epsilon}^{M} := \{ (y,\eta); \text{dist} (  \, (y,\eta),  \Sigma^*_{2M+1} \cup \gcal^{*}_{2M+1} \cup \tcal \, ) \geq \epsilon \}.$ Thus, $L_{H}(M,\epsilon)$ is the maximum length of $H$-broken and $\epsilon$-transversal geodesics with at most $2M$-reflections at the boundary, $\partial \Omega.$ Similarily, we define
\begin{equation} \label{boundarylink}
 L(p,\epsilon):= \inf_{(y,\eta) \in \Omega_{\epsilon}^{p}} \sum_{k=1}^{p} |\pi \beta^{k}(y,\eta) - \pi \beta^{k-1}(y,\eta)|. \end{equation}
Thus, $L(p,\epsilon)$ is the mininum length of $\epsilon$-transversal geodesics in $\Omega$ with $p$-reflections at the boundary, $\partial \Omega.$
We define the constant
\begin{equation} \label{refinedcutoff}
  C(M,\epsilon) = 2M+1 + \max \{ |p|; L(p,\epsilon) \leq L_H(M,\epsilon) + 1 \}. \end{equation}

 To estimate the time-averaged matrix elements in (\ref{upshot1}),  in the next two sections, we require that  the symbol  $a \in S^{0,0}_{cl}(T^*H \times (0,\lambda_0])$ satisfy the somewhat stronger support assumption
\begin{equation} \label{cutaway2}
 \text{dist} \, ( \text{supp} \,  a, \, \tau_{\partial \Omega}^{H} ( \gcal^*_{C(M,\epsilon)} \cup \Sigma^{*}_{C(M,\epsilon)} \cup \tcal ) \, ) \geq \epsilon >0. \end{equation}
We deal with the case of general symbols at the end in section \ref{completion}.

\subsubsection{Analysis of the diagonal sum}

The first important point is that the diagonal sum is bounded by
$O(\frac{1}{M})$. Indeed, the operators
 \begin{equation} \label{ANN} A_{n,n}(\lambda_j;a,\epsilon): =  \langle [N_{\partial \Omega}^{\partial \Omega}(\lambda_j)^{*}]^{n}\,
 F_{2k}(\lambda_j;a,\epsilon)^{*}\, [N_{\partial \Omega}^{\partial \Omega}(\lambda_j)]^{n}
 \,[N_{\partial \Omega}^{\partial \Omega}(\lambda_j)^{*}]^{n}\, F_{2k}(\lambda_j;a,\epsilon) [N_{\partial \Omega}^{\partial
 \Omega}(\lambda_j)]^{n}\end{equation}
in the diagonal terms are $\lambda$-pseudodifferentials;  i.e.
 for each $n=1,...M,$  $$A_{n,n}(\lambda_j;a,\epsilon) \in Op_{\lambda}(S^{0,0}_{cl}(T^* \partial \Omega)).$$
By  the $\lambda$-Egorov theorem for boundary
$\lambda$-pseudodifferential operators (see \cite{HZ} Lemma 1.4),
it follows  that with $\zeta:= (y,\eta),$
\begin{equation} \label{bdyegorov}
A_{n,n}(\lambda_j;a,\epsilon)  = Op_{\lambda_j} \left(
\frac{\gamma(\zeta) }{ \gamma( \beta^{n}\zeta )}   \times \sigma [
F_{2k}^{*}(\lambda;a,\epsilon) F_{2k}(\lambda;a,\epsilon )](
\beta^{n}\zeta) \right) +
 {\mathcal O}_{n,\epsilon}(\lambda_j^{-1})_{L^2 \rightarrow L^2} \,\,\, \text{as} \,\, \lambda_j \rightarrow
 \infty.\end{equation}
Substitution of (\ref{bdyegorov}) in (\ref{upshot1}) together with
the symbol compuation for $F_{2k}(\lambda;a,\epsilon); k=1,2$ in
Lemma \ref{F2}  and an application of the  boundary local Weyl-law
for $\lambda$-pseudodifferential operators \cite{HZ} gives
\begin{equation} \label{upshotdiagonal}\begin{array}{lll}
\lim_{\lambda \rightarrow \infty}   \frac{1}{N(\lambda)} \,
\sum_{\lambda_j \leq \lambda} \,\langle \,\,
A_{n,n}(\lambda_j;a,\epsilon)\,\,u_{\lambda_j}^{b},
u_{\lambda_j}^{b}\, \rangle & = &   \int_{B^{*} \partial
\Omega} \frac{\gamma(\zeta)}{ \gamma(\beta^{n}\zeta)}  \times |
\sigma( F_{2k}(\lambda;a,\epsilon) )(\beta^{n}(\zeta)|^{2} \,
\gamma^{-1}(\zeta) \, dy d\eta \\ && \\& = & \int_{B^{*} \partial
\Omega} \gamma^{-1}(\beta^{n}\zeta)  \times | \sigma(
F_{2k}(\lambda;a,\epsilon) )(\beta^{n}(\zeta)|^{2}  \, dy d\eta\\
&& \\& = & \int_{B^* \partial \Omega} \gamma^{-1}(\zeta)  \times |
\sigma( F_{2k}(\lambda;a,\epsilon) )((\zeta)|^{2}  \, dy d\eta  \\ &&\\
&=& \int_{B^*H} a(s,\tau;\epsilon) \rho_{\partial \Omega}^{H}(s,\tau) \, ds d\tau.
\end{array} \end{equation}
 \vspace{3mm}
since $\beta$ is symplectic. Summing over the $M$ terms and
dividing by $\frac{1}{M^2}$ gives the following formula for the
diagonal sum
\begin{equation} \label{diagonalupshot}
\frac{1}{M^2} \sum_{n=1}^{M}   \lim_{\lambda \rightarrow \infty}  \left( \, \frac{1}{N(\lambda)} \sum_{\lambda_j \leq \lambda} \langle A_{n,n}(\lambda_j;a,\epsilon) u_{\lambda_j}, u_{\lambda_j}  \, \right) = \frac{1}{M} \int_{B^*H} a(s,\tau;\epsilon) \rho_{\partial \Omega}^{H}(s,\tau) \, ds d\tau  = {\mathcal O} ( M^{-1} \| a \|_{L^{\infty}} ). \end{equation}

\subsubsection{Analysis of the off-diagonal sum}

The next step is to  apply Theorem  \ref{weyl2} to the averaged
operator $F_M(\lambda)$ in (\ref{N*N}). The non $\lambda$
pseudodifferential pieces of this operator can be written as a sum
of the  $\lambda$-FIO's $F_{m,n}(\lambda;a,\epsilon)$
 where $m \neq n$ and  $WF_{\lambda}'(F_{m,n}(\lambda;a,\epsilon)) \subset  \text{graph} \,[ (\beta^{k}_{H})^*
  (\beta^{m}) (\beta^n) (\beta^{k}_{H}) ]$ (see (\ref{FMN})). One can locally write these canonical
   graphs in terms of generating functions and so the operators $F_{m,n}(\lambda)$ can be written as a sum of $\lambda$-FIO's with phase functions
   of  the form (\ref{PHASEASSUMPTION}). Thus, Theorem  \ref{weyl2} applies to these operators.

We first observe that we can remove the outer factors of
$N_{\partial \Omega}^{\partial \Omega}$ since they have eigenvalue
$1$ on $u_{\lambda_j}^b$. For $a \in S^{0,0}_{cl}(T^*H \times (0,\lambda_0])$ safisfying the support assumption in (\ref{cutaway2}), we define
\begin{equation} \label{FMN} F_{m,n}(\lambda_j;a.\epsilon) : = F_{2k}(\lambda_j;a,\epsilon)^{*}\,
 [N_{\partial \Omega}^{\partial \Omega}(\lambda_j)]^{m}
 \,[N_{\partial \Omega}^{\partial \Omega}(\lambda_j)^{*}]^{n}\,
 F_{2k}(\lambda_j;a,\epsilon), \,\,\, m \neq n. \end{equation}

The remaining step in Theorem \ref{maintheorem} is the following

\begin{prop}\label{qaclemma} Let $a \in S^{0,0}_{cl}(T^*H \times (0,\lambda_0])$ satisfy the support assumption in (\ref{cutaway2}). Then, under the quantitative almost commutativity assumption  in Definition \ref{ANC}, it follows that
\begin{equation} \label{CONCLUDED} \limsup_{M \to \infty} \lim_{\lambda \to \infty}  \frac{1}{N(\lambda)}  \frac{1}{M^2} \sum_{n, m=0; m \not= n}^{M} \,   \sum_{\lambda_j \leq
\lambda}
  \langle
 \,\, F_{m,n}(\lambda_j;a,\epsilon)\,
 u_{\lambda_j}^{b}, u_{\lambda_j}^{b} \, \rangle = 0.
 \end{equation} \end{prop}

\begin{proof}

In the sum in (\ref{CONCLUDED}), the quantum observables are non-pseudo-differential
$\lambda$-Fourier integral operators, i.e. their canonical relations are graphs of
non-identity maps or correspondences. We apply  the Fourier
integral operator local Weyl law of Theorem \ref{weyl2} and the quantitative
almost nowhere commuting condition in (\ref{ANC}) to prove (\ref{CONCLUDED}).

For $m\neq n$, we
denote the canonical relation
 of $F_{m,n}(\lambda_j;a,\epsilon)$
by $ \Gamma_{    F_{2k}(\lambda_j;\epsilon)^{*}\,
 [N_{\partial \Omega}^{\partial \Omega}(\lambda_j)]^{m} \,[N_{\partial \Omega}^{\partial \Omega}(\lambda_j)^{*}]^{n}\, F_{2k}(\lambda_j;\epsilon)\,}.$
By the composition theorem for semi-classical Fourier integral
operators, this canonical relation is the union of the two
branches ($k = 1, 2$),
 $$ =  \text{graph} \, [  (\beta_{H}^{k})^* (\beta^{m} ) (\beta^{n})^*  (\beta_{H}^{k}) ]
  \subset B^* \partial \Omega \times B^* \partial \Omega,$$
 where $\beta: B^* \partial \Omega \rightarrow B^* \partial \Omega$ is the usual billiard map of the
  domain $\Omega$ and $\beta_{H}^{k}: \mathcal C \backslash E_k^{-1}(\Sigma)
  \rightarrow B^{*} \partial \Omega $ are the branches of billiard maps defined in (\ref{TRANS}) and  in
  Lemma \ref{GAMMABETADEF}.
Equivalently, the canonical relation is the graph of
$$ \left(\beta_{H}^{k} \right)^{-1} \beta^{m - n} \beta_{H}^{k}; \,\, k=1,2.$$
Note that there are two factors of $N_{\partial \Omega}^{H *}
N_{\partial \Omega}^H$, which accounts for the two factors of
$(\beta_{H}^{k})$.

By Theorem \ref{weyl2}, the fixed point set whose measure we need
to calculate is the fixed point set
\begin{equation} \label{fixed1}
 \left(\beta_{H}^{k} \right)^{-1} \beta^{m - n} \beta_{H}^{k} (y, \eta) = (y', \eta'), G^{- \tau}(y', \eta') = (y,
\eta), \,\,\, \tau = \psi_{m-n}(q(y), q(y')), \end{equation}
 where $\psi_{m-n}$ is the phase function of
$F_{m,n}(\lambda;a,\epsilon)$. Clearly, $\psi_{m-n}(q(y), q(y'))$ is the length of
the broken trajectory which starts at $q$ in the direction $\eta +
\sqrt{1-|\eta|^2} \nu_{y}$, breaks on $H$, then proceeds on a
broken geodesic until it has made $m-n$ reflections at the
boundary. Finally, for the last two links the trajectory   breaks
again at $H$. Equivalently, \begin{equation} \label{M-N}
\begin{array}{lll} \Sigma_{F_{m,n}(\lambda;a,\epsilon)}  =   \bigcup_{p \in \Z} \ccal \ocal_{p,m-n}, \\ \\
\ccal \ocal_{p,m-n} = \{(y, \eta) \in B^* \partial \Omega:
\beta^{m - n} \beta_{H}^{k}(y, \eta) = \beta_{H}^{k} \beta^p (y,
\eta)\}.  \end{array}
\end{equation}  As this shows, $\Sigma_{ F_{m,n}(\lambda;a,\epsilon) }$ depends
only on $m - n$.

      We compute
    \begin{equation} \label{CONCLUDEDa} \begin{array}{l} \limsup_{M \to \infty} \lim_{\lambda \to \infty}  \frac{1}{N(\lambda)}
     \frac{1}{M^2}  \left| \sum_{n, m=0; m \not= n}^{M} \,   \sum_{\lambda_j \leq
\lambda}
  \langle
 \,\, F_{m,n}(\lambda_j;a,\epsilon)\,
 u_{\lambda_j}^{b}, u_{\lambda_j}^{b} \, \rangle \, \right|
 \\ \\  \leq   \limsup_{M \to \infty}
     \frac{1}{M} \sum_{|m+n| = 1}^{2M}   \frac{1}{M}  \sum_{|m-n|=1}^{2M}  \lim_{\lambda \rightarrow \infty} \left|  \frac{1}{N(\lambda)}  \langle
 \,\, \sum_{ \lambda_j \leq \lambda} F_{mn}(\lambda_j;a,\epsilon)\,
 u_{\lambda_j}^{b}, u_{\lambda_j}^{b} \, \rangle   \right| \,
 \\ \\  \leq C \limsup_{M \rightarrow \infty} \frac{1}{M}  \sum_{|m-n|=1}^{2M}   \lim_{\lambda \rightarrow \infty}  \left| \frac{1}{N(\lambda)} \sum_{\lambda_j \leq \lambda} \langle F_{mn}(\lambda_j;a,\epsilon) u_{\lambda_j}^{b}, u_{\lambda_j}^{b} \rangle \right|
\end{array}
 \end{equation}

 The final step in the proof of Proposition \ref{qaclemma}   reduces to estimationt of the Weyl sum in the last line of (\ref{CONCLUDEDa}) in terms of the quantitative ANC condition in Definition \ref{ANC}.
 First, we need to define the relevant $\mu_{p,k}^{\epsilon}$-measures (see Definition \ref{ANC}). To do this, we need to further $\lambda$-microlocalize the $F_{mn}(\lambda;a,\epsilon)$-operators away from large-time iterates of the grazing and singular sets.
When $a \in S^{0,0}_{cl}(T^*H \times (0,\lambda_0])$ satisfies (\ref{cutaway2}), by semiclassical wavefront calculus,

\begin{equation} \label{wf1a}
WF_{\lambda}' (F_{2}(\lambda;a,\epsilon)) \subset   \iota_{\Delta} \, \left( \, \cap_{|k| \leq C(M,\epsilon)}    \{ (y,\eta) \in B^* \partial \Omega;  \, \text{dist} (\beta^k(y,\eta), S^*\partial \Omega) \geq \epsilon, \,\, \mbox{dist} (  \beta^k(y,\eta), \Sigma) \geq \epsilon \}  \, \right), \end{equation}
where,  $\iota_{\Delta}: B^*\partial \Omega \rightarrow B^*\partial \Omega \times B^*\partial \Omega$ with $\iota_{\Delta}(y,\eta) = (y,\eta,y,\eta).$
Consider the closed set $ \gcal_{C(M,\epsilon)}^* \cup \Sigma_{C(M,\epsilon)}^* \cup \tcal$ with  $| \gcal_{C(M,\epsilon)}^* \cup \Sigma_{C(M,\epsilon)}^* \cup \tcal| = 0$ and the disjoint open set
 \begin{equation} \label{goodset} \begin{array}{ll} U_{\epsilon}^{M}: = \bigcap_{ |k| \leq C(M,\epsilon)} \{ (y,\eta) \in B^* \partial \Omega; \text{dist} ( \beta^{k}(y,\eta), S^*\partial \Omega) > \epsilon, \,\, \mbox{dist} (\beta^{k}(y,\eta), \Sigma) > \epsilon     \}. \end{array} \end{equation} By the $C^{\infty}$ Urysohn lemma there exists $\chi_{\epsilon}^{M} \in C^{\infty}_{0}(T^*\partial \Omega)$ with the property that
$\chi_{\epsilon}^M(y,\eta) =   1$ when $ (y,\eta) \in \gcal_{C(M,\epsilon)}^* \cup \Sigma_{C(M,\epsilon)}^* \cup \tcal $ and $\chi_{\epsilon}^M(y,\eta) = 0$ for $  (y,\eta) \in U^M_{\epsilon}.$ To simplify the writing in the following, we abuse notation somewhat and write $\chi_{\epsilon}^M$ instead of $Op_{\lambda}(\chi_{\epsilon}^{M}).$

Clearly, $ \beta^{k}(U^M_{\epsilon}) \subset U^M_{\epsilon,k}$ where for $|k| < C(M,\epsilon),$
$$ U^M_{\epsilon,k}:=  \bigcap_{ | \ell | \leq C(M,\epsilon) - |k|}  \{ (y,\eta) \in B^* \partial \Omega; \text{dist} ( \beta^{\ell}(y,\eta), S^*\partial \Omega) > \epsilon, \,\, \mbox{dist} (\pi \beta^{\ell}(y,\eta), \Sigma) > \epsilon \}.  $$
 Let $\chi^M_{\epsilon,k} \in C^{\infty}_{0}(\partial \Omega)$ be a smooth cutoff equal to $1$ on $\gcal_{C(M,\epsilon)}^* \cup \Sigma_{C(M,\epsilon)}^* \cup \tcal$ and equal to $0$ on $U^M_{\epsilon,k}$.  Since $WF'_{\lambda}([N_{\partial \Omega}^{\partial \Omega}(\lambda)]^{k}) \subset \text{graph} \, \beta^{k} $ and $C(M,\epsilon) \geq 2M + 1 \geq |m-n| +1,$  it then follows by semiclassical wave front calculus that
\begin{multline} \label{wf3}
F_{m,n}(\lambda;a,\epsilon)= \tilde{F}_{m,n}(\lambda;a,\epsilon) + {\mathcal O}_{\epsilon,m-n}(\lambda^{-\infty})_{L^2 \rightarrow L^2}, \,\, \text{where,} \\ \\
\tilde{F}_{m,n}(\lambda;a,\epsilon): =   (1-\chi^M_{\epsilon}) F_{2}(\lambda;a,\epsilon)^* (1-\chi^M_{\epsilon}) \cdot  N_{\partial \Omega}^{\partial \Omega}(\lambda) (1-\chi^M_{\epsilon,|m-n|}) \cdot N_{\partial \Omega}^{\partial \Omega}(\lambda)  (1-\chi^M_{\epsilon,|m-n|-1}) \\  \cdots (1-\chi^M_{\epsilon,2}) N_{\partial \Omega}^{\partial \Omega} (\lambda) (1-\chi^M_{\epsilon,1}) F_2(\lambda;a,\epsilon) (1-\chi^M_{\epsilon}) + {\mathcal O}_{\epsilon,m-n,M}(\lambda^{-\infty})_{L^2 \rightarrow L^2} ; \,\,\, |m-n| \leq 2M. \end{multline}
Simlilarily, we define the microlocalized operator with Schwartz kernel
\begin{equation} \label{microGreen}
G_{0}^{\epsilon}(q(y), q(y')):= (1-\chi_{\epsilon}^{M}) G_{0}(\lambda) (1-\chi_{\epsilon}^{M})(q(y),q(y')), \end{equation}
where $G_{0}(\lambda)$ is the free Greens operator in  (\ref{Hankel}). For each $ p \in {\mathbb Z},$  the composite operator $ F_{mn}(\lambda;a,\epsilon) [ \lambda G_0^{\epsilon}(\lambda)|^{p}$ is a $\lambda$-Fourier integral operator of order zero in the sense of (\ref{oscillatory}). Consequently, for $(q(y),q(y')) \in \partial \Omega \times \partial \Omega,$
$$ \rho_{p,m-n}^{\epsilon}(q(y),q(y')): = \limsup_{\lambda \rightarrow \infty} \frac{1}{\lambda^{n-1}} |\tilde{F}_{mn}(\lambda;a,\epsilon)(q(y),q(y'))| \times  |[ \lambda \, G_0^{\epsilon}(\lambda)]^{p}(q(y'),q(y))| < \infty.$$

\subsubsection{ The QANC-condition}  \label{quant condition}
We make the
\begin{defn} \label{mumeasures} Given the map $ \omega: \partial \Omega \times \partial \Omega - \Delta_{\partial \Omega \times \partial \Omega} \rightarrow B^* \partial \Omega$ with $w(y,y') = (y, d_{y}|q(y)-q(y')|),$ we define the  measures (see Definition \ref{ANCQ})
$$ \mu_{p,m-n}^{\epsilon}: = w_{*} (  \, \rho_{p,m-n}^{\epsilon}(q(y),q(y')) d\sigma(y) d\sigma(y') \, ).  $$
\end{defn}
We can now define the QANC condition in terms of the measures $\mu_{p,m-n}^{\epsilon}.$

\begin{maindefn} \label{ANCQ}

We say that  $(\beta_H,\beta)$  {\it quantitatively
       almost never commute} if
       $$\limsup_{M \to \infty} \frac{1}{M} \sum_{|k| \leq M} \sum_{p \in \Z}  \limsup_{\epsilon \rightarrow 0^+} \mu_{p,k}(\ccal
       \ocal_{p,k}  \cap U_{\epsilon}^{M})= 0. $$

 \end{maindefn}

 \begin{rem} \label{messy}
 For general piecewise smooth domains, the measures $\mu_{p,k}^{\epsilon}$ are given by formulas that are quite complicated because of  multiple-reflection solutions to the commutator equations. However, when $\partial \Omega$ is convex, the measures $\mu_{p,k}^{\epsilon}$ simplify  and  satisfy the estimates in Lemma \ref{uniformbound}(b) below. \end{rem}

 \begin{lem} \label{uniformbound}
  Fix $\epsilon >0$ and let $a \in S^{0,0}_{cl}(T^*H \times [\lambda_0, \infty))$ satisfying the support condition in (\ref{cutaway2}).

  (a) \, For $m,n \in {\mathbb Z}$ with $1 \leq |m-n| \leq 2M,$
 \begin{equation} \label{qacupshot}
 \lim_{\lambda \rightarrow \infty}  \left| \frac{1}{N(\lambda)} \sum_{\lambda_j \leq \lambda} \langle F_{m,n}(\lambda_j;a_M,\epsilon) u_{\lambda_j}^{b}, u_{\lambda_j}^{b} \rangle  \right|  \leq  \sum_{p; L(p,\epsilon) \leq L_{H}(|m-n|,\epsilon)+1} \mu_{p,m-n}^{\epsilon}( \ccal \ocal_{p,m-n}) \end{equation}
 where, the measures $\mu_{p,m-n}^{\epsilon}$ are defined in Definition \ref{mumeasures}.

(b) \, When $\partial \Omega$ is convex and $\acal \subset (B^* \partial \Omega \cap U_{\epsilon}^{M})$ is measurable, for all $k$ with $|k| \leq 2M,$
  $$\mu_{p,k}^{\epsilon} (\acal) \leq \int_{\acal} L_{p}(y,\eta)^{-\frac{n-1}{2}} L_{k}( \beta_H(y,\eta) )^{-\frac{n-1}{2}} \prod_{\ell =1}^{p} \gamma^{-2}(\beta^{\ell}(y,\eta)) \, dy d\eta,$$  with $L_{p}(y,\eta):= |\pi \beta^{p}(y,\eta) - \pi \beta^{p-1}(y,\eta)| + |\pi \beta^{p-1}(y,\eta) - \pi \beta^{p-2}(y,\eta)|.$
\end{lem}

\begin{proof}
Part (a) follows from the formulas in section \ref{WKB}. Indeed by (\ref{k link}),
\begin{equation} \label{general} \begin{array}{llll}
\limsup_{\lambda \rightarrow \infty}  \left| \frac{1}{N(\lambda)} \sum_{\lambda_j \leq \lambda} F_{m,n}(\lambda_j;a,\epsilon) u_{\lambda_j}^{b}, u_{\lambda_j}^{b} \rangle \right|  = \limsup_{\lambda \rightarrow \infty}  \left| \frac{1}{N(\lambda)} \sum_{\lambda_j \leq \lambda} \tilde{F}_{m,n}(\lambda_j;a,\epsilon) u_{\lambda_j}^{b}, u_{\lambda_j}^{b} \rangle \right| \\ \\
 \leq \sum_{p; L(p,\epsilon) \leq L_{H}(|m-n|,\epsilon)+1}   \int_{\ccal \ocal_{p,m-n}}   \, w_{*}( \,  | \, [ \lambda \, G_{0}^{\epsilon}(\lambda)]^{p} (q(y),q(y')) |  \times  |{F}_{m,n}(\lambda;a,\epsilon)(q(y'),q(y))| \,\, dq(y) dq(y') \, )  \\ \\ = \sum_{p; L(p,\epsilon) \leq L_{H}(|m-n|,\epsilon)+1} \mu_{p,m-n}^{\epsilon}( \ccal \ocal_{p,m-n}).
  \end{array} \end{equation}

As for part (b),  when $\partial \Omega$ is convex, we prove the estimate in (\ref{qacupshot}) by substituting the explicit WKB formulas for the operators $N_{\partial \Omega}^{H}(\lambda)$ (\ref{wkb})  and $N_{\partial \Omega}^{\partial \Omega}(\lambda)$ \cite{HZ} suitably $\lambda$-microlocalized  away from the  generalized grazing and singular sets.   Since $\Delta  \cap WF_{\lambda}^{'}( (1 -\chi_{\epsilon,k}^{M}) N_{\partial \Omega}^{\partial \Omega}(\lambda) (1-\chi_{\epsilon,k-1}^{M}) = \emptyset$ for all $k =  \pm 1,...,\pm 2M,$  from (\ref{snellterm1}), the WKB formula for the $ (1 -\chi_{\epsilon,k}^{M}) N_{\partial \Omega}^{\partial \Omega}(\lambda) (1-\chi_{\epsilon,k-1}^{M})$ \cite{HZ}  and a standard stationary phase argument, it follows that
 \begin{multline} \label{qacupshot2}
 \tilde{F}_{m,n}(\lambda;\epsilon)(q(y),q(y'))  \sim_{\lambda \rightarrow \infty} (2\pi \lambda)^{\frac{n-1}{2}}  e^{i\lambda \psi_{m-n}^H(q(y),q(y') )} [B_{m-n,\epsilon}^{(0)}(q(y),q(y')) \\   + \lambda^{-1} B_{m-n,\epsilon}^{(1)} (q(y),q(y')) + \dots ]. \end{multline}
 Here, for $m>n,$ $\psi_{m-n}^H(q(y),q(y'))$ is the length of the locally unique $H$-broken, $\epsilon$-transversal $m-n$ link joining $q(y)$ and $q(y').$ For $m<n$ it is minus the length of the link.
In the same way, for all $p \in {\mathbb Z}$ satisfying $L(p,\epsilon) \leq L_H(M,\epsilon) + 1 $ (see (\ref{refinedcutoff})) it follows by wavefront calculus and repeated applications of stationary phase that
 \begin{multline} \label{qacjumps}
[(1-\chi^M_{\epsilon}) \lambda G_0(\lambda) (1-\chi^M_{\epsilon})]^{p}(q(y),q(y'))  \\
 \ \sim_{\lambda \rightarrow \infty} (2\pi \lambda)^{ \frac{n-1}{2}} e^{i \lambda \psi_p(q(y),q(y') )} [A_{p,\epsilon}^{(0)}(q(y),q(y') )   + \lambda^{-1} A_{p,\epsilon}^{(1)}(q(y),q(y')) + \cdots]. \end{multline}
 In (\ref{qacjumps}), for $p>0$ the phase $\psi_{p}(q,q')$ is the length of the locally unique $p$-link joining $q(y)$ and $q(y')$. For $p<0$ it is negative of the length. In the following, it will be useful to define the sum of these phase functions
 \begin{equation} \label{sumphases}
 \Psi_{m-n,p}(q(y),q(y'):= \psi^H_{m-n}(q(y),q(y')) + \psi_{p}(q(y'),q(y)). \end{equation}
 It is clear from  (\ref{boundarylink}) and (\ref{hlink}) that for $|m-n| \leq 2M,$
 \begin{equation} \label{indexbound}
 \inf_{q,q'}| \Psi_{m-n,p}(q(y),q(y'))| \geq L(p,\epsilon) - L_{H}(M,\epsilon). \end{equation}
 We write $y_k = \pi \beta^{k}(y,\eta); k=1,...,m-n-1$ for the intermediate reflection points of the geodesic joining $y_0 = y$ and $y_{m-n} = y'$ and put $L(q(y_p),q(y_{p-2})) = |q(y_p)-q(y_{p-1})| + |q(y_{p-1}) - q(y_{p-2})|$ (see Definition \ref{ANC}).  Using the convexity of $\partial \Omega,$  it follows that  for $p \geq 2,$
\begin{equation} \label{estimate1} \begin{array}{llll}
 |A_{p,\epsilon}^{(0)}(q(y),q(y'))|  \leq \prod_{k=1}^{p-2}  \left( \frac{ |q (y_{k+1}) - q(y_k)|}{|q(y_{k+1})-q(y_k)| + |q(y_k) - q(y_{k-1})| } \right)^{\frac{n-1}{2}}  \times \prod_{k=1}^{p-1}  \left( \frac{1}{ \langle \nu_{q(y_k)}, r(q(y_k), q(y_{k-1})) + r(q(y_k), q(y_{k+1})) \rangle  } \right)  \\ \\ \times  \prod_{k=1}^{p-1}  \langle \nu_{y_{k}}, r(q(y_k),q(y_{k+1}) \rangle^{-1} \times  \left( \frac{1}{|q(y_p)-q(y_{p-1})| + |q(y_{p-1}) - q(y_{p-2})| } \right)^{\frac{n-1}{2}} \\ \\ \leq L(q(y_p),q(y_{p-2}))^{- \frac{n-1}{2}} \prod_{k=1}^{p-1}  \langle \nu_{y_{k}}, r(q(y_k),q(y_{k+1})) \rangle^{-1}  \times \prod_{k=1}^{p-1}   \langle \nu_{y_k}, r(q(y_k), q(y_{k-1})) + r(q(y_k), q(y_{k+1})) \rangle^{-1} \\ \\
 = 2^{-(p-1)} L(q(y_p),q(y_{p-2}))^{- \frac{n-1}{2}} \prod_{k=1}^{p-1}  \langle \nu_{y_{k}}, r(q(y_k),q(y_{k+1})) \rangle^{-2} .  \end{array}\end{equation}
 In the last line of (\ref{estimate1}) we have used that $ \langle \nu_{y_k}, r(q(y_k), q(y_{k-1})) \rangle  = \langle   \nu_{y_k}, r(q(y_k), q(y_{k+1})) \rangle. $

 Similarily, we let $y_0^H = \pi \beta_H (y,\eta),  \, y_{|m-n|+1}^H = \pi \beta_{H}^{-1} \beta^{m-n} \beta_H (y,\eta) $ and $y_k^H = \pi \beta^k \beta_H (y,\eta), k=1,...,|m-n|.$  It follows that with a constant $C_{H,\Omega}>0,$
\begin{equation} \label{estimate2} \begin{array}{lll}
|B_{m-n,\epsilon}^{(0)}(q(y),q(y'))| \leq    2^{- | |m-n|-1 |}C_{H,\Omega} \, L(q(y^H_{|m-n|}),q(y^H_{|m-n|-2}))^{- \frac{n-1}{2}}.  \end{array}\end{equation}

The estimate in (\ref{estimate2}) follows in the same way as in (\ref{estimate1}) noting that  $N^{\partial \Omega}_{\partial \Omega}(\lambda)(q(y),q(y'))$-kernel has the additional $\langle \nu_{y}, r(q(y),q(y') \rangle$-term in the numerator as compared with the Greens function $G_{0}(\lambda)(q(y),q(y')).$ This accounts for the absence of the additional $\langle \nu_{y_k}, r(q(y_k),q(y_{k-1}) \rangle$-terms in the denominator of (\ref{estimate2}).  It then  follows from Theorem \ref{weyl2} (see also the argument in subsection \ref{WKB}), the time-cutoff Lemma \ref{cutofftrace}  and (\ref{indexbound}) that
 \begin{equation} \label{qacupshot3} \begin{array}{llll}
\limsup_{\lambda \rightarrow \infty}  \left| \frac{1}{N(\lambda)} \sum_{\lambda_j \leq \lambda} F_{m,n}(\lambda_j;a,\epsilon) u_{\lambda_j}^{b}, u_{\lambda_j}^{b} \rangle \right|  = \limsup_{\lambda \rightarrow \infty}  \left| \frac{1}{N(\lambda)} \sum_{\lambda_j \leq \lambda} \langle \tilde{F}_{m,n}(\lambda_j;a,\epsilon) u_{\lambda_j}^{b}, u_{\lambda_j}^{b} \rangle \right|  \\ \\
 \leq \limsup_{\lambda \rightarrow \infty} \frac{C}{\lambda^{n-1}} \sum_{p; L(p,\epsilon) \leq L_{H}(|m-n|,\epsilon)+1}  | \int_{Crit ( \Psi_{m-n,p}) }   [ \lambda \, G_{0}^{\epsilon}(\lambda)]^{p} (q(y),q(y'))  \\ \\ \times \tilde{F}_{m,n}(\lambda;a,\epsilon)(q(y'),q(y)) \,\, dq(y) dq(y') | \\ \\   \leq \limsup_{\lambda \rightarrow \infty} C \sum_{p; L(p,\epsilon) \leq L_{H}(|m-n|,\epsilon)+1}  | \int_{Crit ( \Psi_{m-n,p}) }  e^{i \lambda \,  \Psi_{m-n,p}(q(y),q(y')) }  A_{p,\epsilon}^{(0)}(q(y),q(y')) \\ \\ \times B_{m-n,\epsilon}^{(0)}(q(y'),q(y))\, d\sigma(y) \, d\sigma(y')  |  \\ \\
  \leq C_1 \sum_{p; L(p,\epsilon) \leq L_{H}(|m-n|,\epsilon)+1} \int_{Crit (\Psi_{p,m-n}) } | A_{p,\epsilon}^{(0)}(q(y),q(y')) \, B_{m-n,\epsilon}^{(0)}(q(y'),q(y))| \, d\sigma(y) d\sigma(y'). \end{array} \end{equation}
Substitution of the estimates (\ref{estimate1}) and (\ref{estimate2}) in (\ref{qacupshot3}) gives
\begin{equation} \label{qacupshot4} \begin{array}{llll}
\limsup_{\lambda \rightarrow \infty}  \left| \frac{1}{N(\lambda)} \sum_{\lambda_j \leq \lambda} F_{m,n}(\lambda_j;a,\epsilon) u_{\lambda_j}^{b}, u_{\lambda_j}^{b} \rangle \right| \\ \\
 \leq   \sum_{p; L(p,\epsilon) \leq L_{H}(|m-n|,\epsilon)+1 } 2^{ - (p + |m-n| +2)}\int_{Crit (\Psi_{p,m-n})  }  L(q(y^H_{|m-n|}),q(y^H_{|m-n|-2}))^{- \frac{n-1}{2}}  \\ \\ \times L(q(y_p),q(y_{p-2}))^{- \frac{n-1}{2}}   \prod_{k=1}^{p-1} \langle \nu_{y_{k}}, r(q(y_k),q(y_{k+1})) \rangle^{-2}  d\sigma(y) d\sigma(y').
    \end{array} \end{equation}

 We rewrite (\ref{qacupshot3}) by making the change of variables $w:(y,y') \mapsto (y,\pi_{T_y} r(q(y),q(y')) )=: (y,\eta).$ The result is that
\begin{equation} \label{qacupshot5} \begin{array}{lll}
\limsup_{\lambda \rightarrow \infty}  \left| \frac{1}{N(\lambda)} \sum_{\lambda_j
 \leq \lambda} F_{m,n}(\lambda_j;a,\epsilon) u_{\lambda_j}^{b}, u_{\lambda_j}^{b} \rangle \right| \\ \\
 \leq   \sum_{p;L(p,\epsilon) \leq L_{H}(|m-n|,\epsilon)+1} \int_{ \ccal \ocal_{p,m-n}  \cap U_{\epsilon}^{M} } L_{p}(y,\eta)^{- \frac{n-1}{2}}
   L_{m-n}(\beta_H(y,\eta))^{- \frac{n-1}{2}}
   \times \prod_{k=1}^{p} \gamma^{-2}(\beta^{k}(y,\eta)) \, dy d\eta.  \end{array} \end{equation}

This finishes the proof of part (b).

\end{proof}
In view of the quantitative almost commutativity assumption,
 substitution of the bound (\ref{qacupshot}) in the last line of (\ref{CONCLUDEDa}) concludes the proof of the Proposition \ref{qaclemma}.
\end{proof}

\subsection{Completion of the proof of Theorem \ref{maintheorem}: general symbols.} \label{completion}

Given $ a \in S^{0,0}_{cl}(T^*H \times (0,\lambda_0])$ we let $\chi_{\epsilon}^{M} \in C^{\infty}_{0}(T^*\partial \Omega)$ as above and write $ a = [( \tau_{\partial \Omega}^{H})^* \chi_{\epsilon}^{M} ]  a +[ 1- ( \tau_{\partial \Omega}^{H})^*  \chi_{\epsilon}^{M})] a.$
 From Proposition \ref{qaclemma} and the diagonal step, it follows that
 \begin{equation} \label{genargument}
\begin{array}{lll}
\limsup_{\lambda \to \infty} \frac{1}{N(\lambda)} \sum_{\lambda_j \leq \lambda} | \langle Op_{\lambda_j}(a) \phi_{\lambda_j}|_H,
\phi_{\lambda_j}|_{H} \rangle - \int_{B^*H} a \rho_{\partial \Omega}^{H} ds d\tau |^{2}  \\ \\
  \leq  3 \limsup_{\lambda \rightarrow \infty} \frac{1}{N(\lambda)} \sum_{\lambda_j
\leq \lambda} | \langle Op_{\lambda_j}( [ 1- (\tau_{\partial \Omega}^{H})^*  \chi_{\epsilon}^{M} ] a ) \phi_{\lambda_j}|_H,
\phi_{\lambda_j}|_{H} \rangle \\ \\- \int_{B^*H}  (1-\chi_{\epsilon}^{M})(s,\tau) a(s,\tau) \rho_{\partial \Omega}^{H} ds d\tau|^{2}  \\ \\
+ 3 \left| \, \int_{B^*H}   \chi_{\epsilon}^{M}(s,\tau) a(s,\tau) \, \rho_{\partial \Omega}^{H} ds d\tau  \, \right|^{2} \\ \\

+ 3 \limsup_{\lambda \rightarrow \infty} \frac{2}{N(\lambda)} |  \langle \, Op_{\lambda_j}( [ (\tau_{\partial \Omega}^{H})^*  \chi_{\epsilon}^{M} ] a ) \, \phi_{\lambda_j}|_{H}, \phi_{\lambda_j}|_{H} \, \rangle|^{2}
.\end{array} \end{equation}
From Lemma \ref{tanmass} it follows that
\begin{equation} \label{tanmassref} \limsup_{\lambda \rightarrow \infty} \frac{1}{N(\lambda)} |  \langle \, Op_{\lambda_j}( [ (\tau_{\partial \Omega}^{H})^*  \chi_{\epsilon}^{M}]  a ) \, \phi_{\lambda_j}|_{H}, \phi_{\lambda_j}|_{H} \, \rangle|^{2} = {\mathcal O}_{M}(\epsilon), \end{equation} and it is clear that
\begin{equation} \label{easypart}
\left| \, \int_{B^*H}  \chi_{\epsilon}^{M}(s,\tau) a(s,\tau) \, \rho_{\partial \Omega}^{H} ds d\tau  \, \right|^{2}    = {\mathcal O}_{M}(\epsilon). \end{equation}
Finally, since $ \text{dist} ( \, \text{supp} \,  ( \, a - [(\tau_{\partial \Omega}^{H})^*\chi_{\epsilon}^{M}]  a \, ), \, \tau_{\partial \Omega}^{H} ( \gcal^*_{C(M,\epsilon)} \cup \Sigma_{C(M,\epsilon)}^{*} \cup \tcal ) \, ) \geq \epsilon,$ it follows from (\ref{upshotdiagonal}) and Proposition \ref{qaclemma} that
\begin{equation} \label{psdopart} \begin{array}{lll}
\limsup_{\lambda \rightarrow \infty} \frac{2}{N(\lambda)} \sum_{\lambda_j
\leq \lambda} | \langle Op_{\lambda_j}(  a-  [ (\tau_{\partial \Omega}^{H})^* \chi_{\epsilon}^{M}] a ) \phi_{\lambda_j}|_H,
\phi_{\lambda_j}|_{H} \rangle - \int_{B^*H}  (1-\chi_{\epsilon}^{M})(s,\tau) a(s,\tau) \rho_{\partial \Omega}^{H} ds d\tau|^{2} \\ \\
\leq \frac{1}{M^2} \sum_{n=1}^{M}   \left( \limsup_{\lambda \rightarrow \infty}  \frac{1}{N(\lambda)} \sum_{\lambda_j \leq \lambda} \langle Op_{\lambda_j}(  \,  a - [( \tau_{\partial \Omega}^{H})^* \chi_{\epsilon}^{M} ] a )_{n,n}(\lambda_j;a,\epsilon) u_{\lambda_j}^{b}, u_{\lambda_j}^{b} \rangle \right) \\ \\
+ \frac{1}{M} \sum_{|m-n|=1}^{2M} \sum_{p \in {\mathbb Z} } \mu_{p,m-n}^{\epsilon}( \ccal \ocal_{p,m-n} ) \\ \\
 = \frac{1}{M} \int_{B^*H} (1-\chi_{\epsilon}^{M}) a \,\rho_{\partial \Omega}^{H} ds d\tau  + o_{\epsilon}(1) \leq \frac{C}{M}\,  \|a \|_{L^{\infty}} + o_{\epsilon}(1) \end{array} \end{equation}
 as $M \rightarrow \infty.$ By the quantitative almost commutation condition in Definition \ref{ANCQ},  for the last term on the RHS of  (\ref{psdopart}),  $\limsup_{\epsilon \rightarrow 0^+} o_{\epsilon}(1) = o(1)$ as $M \rightarrow \infty.$  The constant $C =C(\Omega)>0$ in the last line of (\ref{psdopart}) is uniform in $\epsilon$ and $M$.
Letting $\epsilon \rightarrow 0^+$ kills the first two terms  in (\ref{tanmassref}) and (\ref{easypart}) and then taking $M \rightarrow \infty$ kills the last term in (\ref{psdopart}).
 This completes the proof of Theorem \ref{maintheorem} \qed.

\begin{rem} As a check on the midline of the stadium, we note that
in that case $|\Sigma_{m, n} | = 1$ for all $m,n$ since $\beta^p$
commutes with the transmission map $\beta_{\partial \Omega}^{H;k}$
(which equals $\sigma$). Thus, the sets in the union in
(\ref{M-N}) are empty except when $m - n = p$ and that one has
full measure. It is simple to check that the limit is non-zero.

\end{rem}

\subsection{Extensions and generalizations} \label{extensions}

We now observe that the proof so of Theorems  \ref{CDTHM}  and \ref{maintheorem} extend to general symbols and to
Dirichlet boundary conditions.

\subsubsection{Eigenfunction mass along tangential and singular sets}

The extension of Theorems \ref{CDTHM}  and \ref{maintheorem} in section  \ref{completion} (see (\ref{tanmassref})) to general symbols $a \in S^{0,0}(T^*H \times (0,\lambda_0^{-1}]$ relies on showing that for a full-density of eigenfunctions, mass does not concentrate along tangential or singular sets.

The point of this subsection is to establish this fact  (see Lemma \ref{tanmass}) by using a  pointwise microlcal Weyl law argument  to show that there is no such mass concentraion on any closed set of measure zero. This method has the advantage of working whether or not $\Omega$ has a boundary. An alternative method  is to use the potential layer $N_{\partial \Omega}^{H}(\lambda)$ to reduce the problem to $\partial  \Omega$ and then apply the argument in \cite{HZ} (see Lemmas 7.1, 9.2 and Appendix 12) which uses the Karamata Tauberian argument for the restriction of the Neumann heat kernel to the boundary diagonal.

In the following we let $(x',x_n) \in {\mathbb R}^{n-1} \times {\mathbb R}$ denote normal coordinates near the interior hypersurface $H \subset \Omega$ defined by $  \Omega  \ni x = q_H(x') + x_n \nu_{x'}$ and  write $(\xi',\xi_n)$ for the corresponding fiber coordinates defined by $T^*_{x}\Omega \ni \xi = \xi^{'} + \xi_n \nu_{x'}$. In the following, we denote the restriction operator to $H$ by $\gamma_H: f \mapsto f|_{H}$ and  continue to write $\zeta: B^*H \rightarrow S_{H}^{*} \Omega$ for the map $\zeta: (x',\xi') \mapsto (x', \xi' + \sqrt{1-|\xi'|^{2}} \nu_{x'}).$

Let ${\zcal} \subset B^*H$ be a closed subset with $| \zcal | = 0.$  Then, by the $C^{\infty}$ Urysohn lemma, we can choose $\chi_{\epsilon,\zcal}(x',\xi)$ smooth and positive homogeneous of degree zero  with the property that  for $(x', \xi) \in S_{H}^{*}\Omega,$
\begin{equation} \label{tancutoff}  \left\{ \begin{array}{ll}  \chi_{\epsilon,\zcal}(x', \xi) = 1 \;\; \mbox{when} \; \mbox{dist} \;((x', \xi),  \zeta \left(  \zcal \right) \, ) \leq \epsilon,
\\ \\  \chi_{\epsilon,\zcal}(x',\xi) = 0\;\; \mbox{when} \; \; \mbox{dist} \; ((x',\xi),  \zeta \left( \, \zcal   \right) \,) \geq 2\epsilon.
\end{array} \right. \end{equation}
We define
$$ {\tilde{\chi}}_{\epsilon,\zcal}(x,\xi) = \chi_{\epsilon,\zcal}(x',\xi) \cdot \chi(x_{n}),$$
where, $\chi(u) \in C^{\infty}_{0}({\mathbb R})$ equal to $1$ near $u=0$ and the corresponding restriction $\chi_{\epsilon, \zcal}^{H} \in C^{\infty}_{0}(T^*H)$ given by
\begin{equation} \label{res-symbol}
 \chi_{\epsilon, \zcal}^{H}(x',\xi'):= \chi_{\epsilon, \zcal}(x'; \xi', \xi_n=0). \end{equation}
 Consider the operator $Op({\tilde{\chi}}_{\epsilon,\zcal}): C^{\infty}(\Omega) \rightarrow C^{\infty}(\Omega)$ and let $U(t) = e^{it \sqrt{\Delta_{\Omega}}}$ denote the Neumann wave operator on $\Omega.$   We wish to compute the small time asymptotics for $ \sum_{j} e^{i\lambda_j t} \| Op_{\lambda_j}( \chi_{\epsilon, \zcal}^{H}) u_{\lambda_j}^{H} \|_{L^{2}(H)}^{2}$ and then apply the usual Fourier Tauberian theorem to get large $\lambda$ asymptotics for the $\lambda$-microlocalized Weyl sum $ \frac{1}{N(\lambda)} \sum_{\lambda_j \leq \lambda} \| Op_{\lambda_j}( \chi_{\epsilon, \zcal}^{H}) u_{\lambda_j}^{H} \|_{L^{2}(H)}^{2}. $

  For  $|t| < \text{dist}( x', \partial \Om), $ we have that
\begin{equation} \label{microwave}  \begin{array}{ll}
U_{\epsilon,\zcal}(t,x',x')
:= \sum_{j} e^{i\lambda_j t} | [ \, Op( \chi_{\epsilon,\zcal}^H)u_{\lambda_j}^{H} \,](x')|^{2} \\ \\ =[ \gamma_H   \, Op( {\tilde{\chi}}_{\epsilon,\zcal} ) \,  U(t) \,   Op( {\tilde{\chi}}_{\epsilon,\zcal})^* \, \gamma_H^* ](t,x',x'), \,\,\,\, x' \in H. \end{array} \end{equation}

Since $H$ is an interior hypersurface,  for $|t| < \text{dist} (x',\partial \Omega), $  bicharacteristics starting from $x' \in H$ do not intersect the boundary,  $\partial \Omega.$ Therefore, by propagation of singularities,  modulo a $C^{\infty}$-error, we can substitute the interior parametrix $U_{0}(t,x,y) = (2\pi)^{-n} \int_{{\mathbb R}^{n}} e^{i [ \langle x-y,\xi \rangle  - t |\xi| ]}   a(x,y,\xi)  d\xi$ for $U(t,x,y)$ in (\ref{microwave}).
 By  the usual composition calculus for Fourier integral and pseudodifferential operators, it follows from (\ref{microwave}) that
\begin{equation} \label{microwave2}
U_{\epsilon,\zcal}(t,x',x') = (2\pi)^{-n} \ \int_{T_{x'}^*\Omega} e^{-it |\xi|} |\chi_{\epsilon,\zcal}(x',\xi)|^{2}  d\xi + \cdots,
\end{equation}
where the dots denote terms with lower-order singularities at $t \rightarrow 0.$  By making a polar coordinates decomposition in the fiber variables in (\ref{microwave2}) and using that $\chi_{\epsilon,\zcal}$ is homogeneous of degree zero, it follows   that as $t \rightarrow 0,$
\begin{equation} \label{microwave3} \begin{array}{ll}
U_{\epsilon,\zcal}(t,x',x') = (2\pi)^{-n} \left( \int_{S_{x'}^*\Omega} |\chi_{\epsilon,\zcal}(x',\omega)|^{2}  d\omega \right) \, (t+i0)^{-n}  + o(t^{-n}) \\ \\
= (2\pi)^{-n} \left( \int_{B_{x'}^*H} |\chi_{\epsilon,\zcal}^H(x',\xi')|^{2} \, \gamma^{-1}(x',\xi') \,  d\xi' \right) \, (t+i0)^{-n}  + o(t^{-n}).  \end{array}
\end{equation}
Finally, an application of the Fourier Tauberian theorem,  integration over $x' \in H$ on both sides of (\ref{microwave3}) and using that $Op(\chi_{\epsilon,\zcal}) = Op_{\lambda_j}(\chi_{\epsilon,\zcal})$ (since $\chi_{\epsilon,\zcal}$ is homogeneous of order zero) proves  the following lemma:
\begin{lem} \label{tanmass} Let $\zcal \subset B^*H$ be closed with $|\zcal |=0$ and for arbitrary $\epsilon >0$ let $\chi_{\epsilon,\zcal}^{H} \in C^{\infty}_{0}(T^*H)$ be  the symbol in (\ref{res-symbol}). Then,
$$ \limsup_{\lambda \rightarrow \infty} \frac{1}{N(\lambda)} \sum_{\lambda_j \leq \lambda} \| Op_{\lambda_j}(\chi_{\epsilon,\zcal}^{H}) u_{\lambda_j}^{H} \|^{2}_{L^2(H)}   = (2\pi)^{-n} \left( \int_{B^*H} |\chi_{\epsilon,\zcal}^H(x',\xi')|^{2}  \, \gamma^{-1}(x',\xi') \,  dx' d\xi' \right)  =  {\mathcal O}(\epsilon).$$
\end{lem}

Thus, for a full-density of eigenfunctions, mass does not concentrate on measure-zero closed subsets $\zcal \subset B^*H$.

 In subsection \ref{completion} (see (\ref{tanmassref})), we apply Lemma \ref{tanmass} with $\zcal = \tau_{\partial \Omega}^{H} ( \Sigma^*_{C(M,\epsilon)} \cup \gcal_{C(M,\epsilon)}^* \cup \tcal)$ to extend Theorem \ref{maintheorem} to arbitrary symbols $a \in S^{0,0}.$

\subsubsection{The Dirichlet Case} \label{Dirichlet} The proof in the Dirichlet case is very simliar
to the Neumann case. There are a few minor changes: the relevant
density is $\rho_{H}(s,\eta)= (1-|\eta|_{g}^{2})^{1/2}$ for
Dirichlet and this follows from the result in \cite{HZ} for
boundary ergodicity. Also, the $N_{\partial \Omega}^{H}$-operator
gets replaced by the operator with kernel $G_{0}(q,q_H;\lambda)$
and this operator has the same microlocal properties as
$N_{\partial \Omega}^{H}$. Finally, of course the boundary traces
of DIrichlet eigenfunctions are $v_{\lambda}^{b}:\partial_{\nu} \phi_{\lambda}.$

\subsection{Nodal intersections with interior curves: Proof of Corollary \ref{nodalthm}} \label{final}

\begin{proof}  Let $\Omega \subset \mathbb{R}^2$ be a piecewise analytic planar domain. The corollary follows from  Theorem 6 in \cite{TZ} provided we  establish the  following "goodness
" condition on $C \subset \Omega$:
\begin{equation} \label{good}
\frac{\| u_{\lambda}^{C,\nu} \|_{L^2(C)} } { \| u_{\lambda}^{C}
\|_{L^{2}(C)} } = {\mathcal O}(e^{a \lambda}).
\end{equation}
Here, we write $u^{C}_{\lambda} = \phi_{\lambda}|_{C}$  and $u^{C,\nu}_{\lambda} = \partial_{\nu_C} \phi_{\lambda}|_{C}$ and $a>0$ is an
arbitrary positive constant. Assume the eigenfunctions
$\phi_\lambda$ are $L^{2}$-normalized in $\Omega$ so that $
\int_{\Omega} |\phi_\lambda|^{2} dx = 1$. Then, by Theorem
\ref{maintheorem} it follows that  for an ergodic sequence of  the
$u_{\lambda}^{C}$'s,
$$ \|u_\lambda^{C} \|_{L^{2}(C)} \sim 1$$
as $\lambda \rightarrow \infty$. By standard sup-estimates, we
also have that
$$ \sup_{x \in \Omega} |\partial_{\nu} \phi_{\lambda}(x)| = {\mathcal O}(\lambda^{3/2}).$$
Corollary \ref{nodalthm} then follows from (\ref{good})  since
clearly $\lambda^{3/2} \ll e^{a \lambda}$ for any $a>0$.

\end{proof}

\section{ Generic interior hypersurfaces $H \subset \Omega$: Proof of Proposition \ref{GENERICintro}} \label{fixedpoint}

In this section, we prove the genericity of the condition of
almost never commutativity stated in Proposition
\ref{GENERICintro}. As mentioned in the introduction, the
non-commuting condition should be generic in a much wider sense.
We present only a simple result here for the sake of brevity; we
hope to give a more general one elsewhere.

There are at least two standard ways to approach generic domains.
Following \cite{CPS,PS1,PS2,S}), we may consider the Frechet space
 $C^{\infty}(S^1, \Omega)$ of smooth embeddings $f: S^1 \to
\Omega^o$ (the interior of $\Omega$);  one could relax the
smoothness to obtain a Banach space $C^{k}(H, \Omega)$ for some $k
\geq 2$. A second approach (see \cite{U,FO}) is to fix a
hypersurface $H$ (here, a curve)  and to parameterize nearby
hypersurfaces by flowing out $H$ under vector fields along $H$,
extended smoothly to a neighborhood of $H$. In either definition,
an infinitesimal variation of $H$ is  given by a smooth vector
field $X$ along $H$. We say that a property is generic if it holds
for a residual subset of the relevant space, i.e. a subset
containing a countable intersection of dense open sets.

The main idea is to prove that the sets $\ccal \ocal_{k, p}$ of
Definition \ref{ANC} are smooth hypersurfaces away from the
corners. To do this we use a well-known transversality theorem.
We recall that if $F: M \to N$ is a $C^1$ map of Banach manifolds,
then $x \in M$ is a regular point if $DF_x: T_x M \to T_{F(x)} N$
is surjective. Also, $y \in N$ is a regular value if every point
of $F^{-1}(y)$ is regular.

\medskip
\noindent{\bf Transversality theorem} \cite{AR,Q,U} {\it Let $\Phi
: \hcal \times B \to E $ be a $C^k$ map of Banach manifolds with
$E, \hcal$ separable. If $0$ is a regular value of $\Phi$, and if
$\Phi_H = \Phi(H, \cdot)$ is a Fredholm map of index $< k$ then
the set $\{H \in \hcal: $0$ \; \mbox{is a regular value of}\;
\Phi_H\}$ is residual in $\hcal$.}

We let $X = S^{1}$ and $Y = \Omega$. The Fredholm condition is
then trivial.  We denote an embedding by $f: S^{1} \to \Omega$ and
let $H = f(S^{1})$. The properties of $f$ which concern us are
properties only of the image $H$. The principal property  is that the measure of
$\ccal \ocal_{p, k}$ is positive. We introduce the
partially defined symplectic correspondence,
$$\Phi_{j, p}: \hcal \times B^* \partial \Omega \to B^* \partial \Omega, \;\;\;\Phi_{k, p}(H, \zeta)
=   \beta^{-p}  \beta_{H}^{k*} \beta^j \beta_H^k (\zeta). $$ Here,
$\beta_H^{k*}$ denotes the (partially defined) inverse to
$\beta_H^k$, which makes sense since $\beta_H^k$ is symplectic on
its domain of definition. More precisely, we restrict $\zeta$ to
the domain of $\beta^{-p} \beta_H^{k*} \beta^j \beta_H^k $, which
depends (slightly) on $H$. Our aim is to prove that, for
 generic (i.e. a residual set of) $H$,  the set $\{\zeta \in B^* \partial
\Omega:
 \beta^{-p}  \beta_H^{k*}  \beta^j \beta_H^k(\zeta) = \zeta \}$ is
  a hypersurface (i.e. curve) away from the corner set  in $B^* \partial \Omega$  and hence of Minkowski content zero.
It suffices to prove that $0$ is a regular value of the map
$\Phi_{j, p}(H, \zeta) - \zeta$,   i.e. that $D(\Phi_{j, p}(H,
\zeta) - \zeta)$ is surjective for each $(H, \zeta)$ such that
$\Phi_{j, p}(H, \zeta) - \zeta = 0$. If $\Phi_{j, p}(H, \zeta) \zeta$ then $D_{\zeta} \Phi_{j, p}(H, \zeta) :  T_{\zeta} B^*
\partial \Omega \to T_{\zeta} B^*
\partial \Omega$. Surjectivity of $$I - D_{\zeta} \Phi_{j,p}(H, \zeta): T_{\zeta} B^* \partial \Omega \to T_{\zeta} B^*
\partial \Omega$$ holds  if and only if   $1$ is not an
eigenvalue of this kind of `Poincar\'e map' for generic $H$ except
for a possible curve  in $B^* \partial \Omega$.
 Note  that if $\ccal \ocal_{p, j}$ had positive measure, then all eigenvalues
 of this Poincar\'e map would equal one  at any point of density $(y, \eta)$ of this
 set.  In \cite{PS1,PS2,S} the
somewhat analogous result is proved that for generic domains, the
spectrum of the Poincar\'e map for every periodic reflecting ray
omits 1 from its spectrum. It is possible that the multi-jet
transversality approach of that paper could be adapted to our
problem. Here we opt for a simpler proof when the billiard flow is
hyperbolic.

\begin{proof}

We tacitly view the equation $\beta_H^k \beta^p(y, \eta) - \beta^j
\beta_H^k(y, \eta) = 0$ as written in local coordinates in $B^*
\partial \Omega$. We henceforth write $\beta_H^k = \beta_H$ and suppress the fact that $\beta_H$ is
double-valued since we can separately consider the four branches
of $\Phi_{j, p}$.   We would like to show that the derivative in
$H$ of the vector valued function $\beta_H \beta^p(y, \eta) -
\beta^j \beta_H(y, \eta)$ spans the tangent space $T_{\beta^j
\beta_H (y, \eta)} B^*
\partial \Omega$ at each point where the equation holds.

The variation $\delta_H \beta_H (y, \eta)$ is defined as follows:
Let $H_{\epsilon}$ be a smooth curve of hypersurfaces through $H_0
= H$ with variational vector field $X$ along $H$. Then,
$$\delta \beta_H (\zeta) : = \frac{d}{d\epsilon} |_{\epsilon = 0}
\beta_{H_{\epsilon}}(\zeta). $$ It is a tangent vector to a curve
through $\beta_H(\zeta)$ depending on the vector field $X$ on $H$
defining the variation, i.e. $\delta \beta_H (\zeta): T_H \hcal
\to T_{\beta_H(\zeta)} B^* \partial \Omega$. To see the dependence
on $X$ we give an explicit formula for $\beta_H$ similar to that
 in \S \ref{BMAPSH}:

\begin{equation}
\beta_H(y, \eta) = E(t^1_H(y, - \eta), y, - \eta) -
(\tilde{t}_H^2(y, \eta) - t_H^1 (y, - \eta)) \zeta_H^j(y, - \eta)
.\end{equation} Here, $\tilde{t}_H^2(y, - \eta)$ is the time at
which the reflected ray hits $\partial \Omega$ and $\zeta_H^j(y,
\eta)$ is defined in (\ref{ZETAHJ}). This vector is smooth in the
unit normal $\nu_{q_H^j(y, \eta)}$. The variation thus consists
of two types of terms: ones where we differentiate the hitting
times  and one where we  vary the normal direction:
\begin{equation} \label{ZETAHJdelta} \delta \zeta_H^j(y, - \eta) = 2 \delta (\zeta \cdot
\nu_{q_H^j(y, - \eta)}) \nu_{q_H^j(y, - \eta)}. \end{equation} It
is clear that both terms depend on (and only on) $X_{q_H^j(y,
-\eta)}$ (we recall that $q_H^j(y, -\eta)$ is the point where the
trajectory from $(y, \eta)$ intersects $H$ (see Definition
\ref{HBILLDEF}). Thus, $\delta \beta_{H_0}(y, -\eta)$ depends on
(and only on) the value of $X$ at the point $q^j_H(y, - \eta)$.

The variation  of the equation has the form $\delta \beta_H
(\beta^p(y, - \eta)) = D \beta^k (\delta \beta_H(y, - \eta))$. As
just noted,  $\delta \beta_{H} (y, - \eta)$ depends on (and only
on) the value of $X$ at $q_H^j(y, -\eta)$. On the other hand,
$\delta \beta_H (\beta^p(y, - \eta))$ depends only on the point
$q_H^j(\beta^p(y,- \eta))$ where the trajectory intersects $H$
after $p$ bounces. We now prove that these points must be
different.

\begin{lem} Let $\Omega$ be a hyperbolic planar billiard table. Let
 $H \subset \Omega$ be a convex curve. Then for almost all $(y, \eta) \in B^* \partial \Omega,$
  the set of $\leq 2$ points where the line segment $\overline{y,
  \eta}$ hits $H$ is disjoint from the set where the line segment
  $\overline{\beta^k \beta_H(y, \eta)}$ hits $H$.

\end{lem}

\begin{proof}

The statement is equivalent to saying that the  `probability' in
$\partial \Omega \times H$ of pairs $(x,q)$ such that the
trajectory defined by $(x, -  \frac{x - q}{|x - q|})$ has the
property that the point $q$ where it first hits $H$ is the same as
the point $q'$ where it hits $H$ after $p$ bounces from $\partial
\Omega$. We may think of the orbit as beginning at $q$ rather than
$x$ with the same initial direction $- \frac{x - q}{|x - q|}$ and
proceeding along the billiard flow of $\Omega$ until it has
bounced off $\partial \Omega$ $p$ times. We then terminate the
trajectory when it hits $H$ for the first time after hitting at
the $p$th bounce point $x_p$. If $q' = q$ for a set of positive
measure in $x$, then $q$ is a self-conjugate point in the sense
that there exists a normal Jacobi field along the broken geodesic
which vanishes at $t = 0$ (at the point $q$) and also at a later
time when it returns to $q$. In fact, it is a self-conjugate point
in the stronger  sense that there exists a positive measure set of
directions in $S^*_q \Omega$ for which the billiard trajectory
returns to $q$ (and additionally so that there exist exactly $p$
bounces before the return to $q$). This is trivially impossible
for hyperbolic billiards, since the (broken) Jacobi fields defined
by the billiard trajectories have no zeros and hence there are no
self-conjugate points (see \cite{W}).
\end{proof}

Although we did not state it,  the argument is valid in all
dimensions. The assumption of hyperbolicity is more than
necessary; for the Proposition, we only use  that the set $(q,
\eta)$ of directions of tangential conjugate points is a set of
content zero in $B^*
\partial \Omega$.

Henceforth we  assume $y \not= y'$. Since $X$ ranges over all
vector fields, we may assume that $X$ is zero at one of these
points and non-zero at the other. Since $D \beta^{-p}$ is
symplectic (hence an isomorphism), the following Lemma is
sufficient to prove surjectivity of the differential away from
curve of $(y, \eta) \in B^* \partial \Omega$ where $\beta_H(y, \eta) (y, \eta)$.

\begin{lem}

If $\dim \Omega = 2, $ then for all $\zeta \in B^* \partial
\Omega, $  $\{\delta \beta_H(q, \eta)\}$ spans $T_{r_H(q, \eta)}
B^* \Omega$ as the variation runs over vector fields on $H$,
except in the case where $\beta_H(q, \eta) = (q, \eta)$, i.e.
where  $q + t \zeta(q, \eta) $ intersects $H$ orthogonally.

\end{lem}

\begin{proof}

We choose the variations so that the variation  vector fields $X$
 vanish at $q_H^j(\beta^p \beta_{H_0}(\zeta)) $   and span
$T_{q_H^j(y,\eta)} \Omega$. We want to prove that $D_H \beta^k
\delta_H \beta_H(y, \eta)$ spans $T_{\beta^k \beta_H(y, \eta)} B^*
\partial \Omega$.
  It is helpful to use
the following notation:  a tangent vector $v$ to $B^*
\partial \Omega$ is said to be vertical ($v \in V B^* \Omega$ if it is tangent to the fibers of
the natural projection to $\Omega$;  it is horizontal ($v \in H
B^* \Omega$) if the momentum variable is fixed while the position
variable changes.

We  first observe that $q(y) + t \zeta(y, - \eta)$ is a fixed ray
independent of $H$. We note
 that the reflection from $H$ depends only on the tangent line to $H$  at
 the point of impact. In varying $H$, we  can vary (and only vary) the tangent line at the
 point of impact  and the  distance
 along the ray to the tangent line.  These parameters correspond
 to two variations of $H$: in the first, we fix the point of
 impact and deform $H$ so that the tangent line at the point of
 impact varies; while in the second, we hold the tangent line
 fixed and move the point of impact along the ray $q(y) + t \zeta(y,
 - \eta)$.

The possible reflected rays thus have the form
$$Y(t, \nu_{\theta}) = q(y) + t_2 \zeta(y, -\eta) - 2 (t_2 - t) \langle \nu_{\theta},
\zeta(y, -\eta) \rangle \nu_{\theta}, $$ where $t_2 = t_2(t,
\theta)$ is the time at which the reflected ray hits $\partial
\Omega$ and where $\nu_{\theta} = (\cos \theta, \sin \theta).$ The
reflected direction is $ =  \zeta(y, -\eta) - 2 \langle
\nu_{\theta}, \zeta(y, -\eta) \rangle \nu_{\theta}, $ so
$$\beta_{H(t, \theta)}(y, -\eta) = (Y, \zeta(y, -\eta) - 2 \langle \nu_{\theta},
\zeta(y, -\eta) \rangle \nu_{\theta}). $$ We note that the
direction is independent of $t$. So for fixed $\theta$,
$$\frac{\partial \beta_H}{\partial t} = \frac{\partial Y}{\partial t}
\oplus 0 \in H_{\beta_H(y, \eta)} B^* \partial \Omega. $$ Here,
$$\frac{\partial Y}{\partial t} = \frac{\partial t_2}{\partial t} \zeta - 2 \langle \nu_{\theta}, \zeta \rangle
\nu_{\theta}. $$ Further,
$$\frac{\partial \beta_H}{\partial \theta} = \frac{\partial Y}{\partial
\theta}\oplus \frac{\partial }{\partial \theta } (- 2 \langle
\nu_{\theta}, \zeta \rangle \nu_{\theta}). $$ We now check that
these two vectors span $T_{\beta_H(y, \eta)} B^* \partial \Omega.$
We have,  $\frac{\partial }{\partial \theta } ( \langle
\nu_{\theta}, \zeta \rangle \nu_{\theta}) =  \langle
\frac{\partial }{\partial \theta }  \nu_{\theta}, \zeta \rangle
\nu_{\theta} +  \langle \nu_{\theta}, \zeta \rangle \frac{\partial
}{\partial \theta } \nu_{\theta}. $ The two terms are orthogonal,
so the vector cannot vanish unless both coefficients vanish, but
that is also impossible by orthogonality. This tangent vector is
vertical while the $t$ derivative is horizontal so to complete the
proof it suffices to consider the case where $\frac{\partial
t_2}{\partial t} \zeta - 2 \langle \nu_{\theta}, \zeta \rangle
\nu_{\theta} = 0$. The sum vanishes at $\theta = 0$ if and only
$\zeta = \nu$, the normal to $H$ at the point of impact. But in
that case, $\beta_H(\zeta) = \zeta$.

\end{proof}

\section{\label{NOTATION} Notational Appendix}

\begin{enumerate}
\item To notationally disinguish points of $\partial \Omega$ and
points of $H$, points of $\partial \Omega$ are usually denoted by
$y$ or by $q$, while points of $H$ are denoted by $q_H$. By a
slight abuse of notation, we also write  $y \mapsto q(y) \in
\partial \Omega$ for the parameter of a local parameterization, and   $s \mapsto q_H(s) \in H$
for a local parameterization of $H$;  we  sometimes use $s$ to
denote a point of $H$.  The dual coordinates are $\eta \in
T_y^*(\partial \Omega)$ and $\tau \in T_s^*(H).$ \medskip

\item $\pi_q^T(\xi) $ is the projection onto the tangent space at
$q$ of $\xi \in S^*_q \Omega$; we use the same notation for
$\partial \Omega$ and for $H$. \medskip

.\item $r(q_H,q) = \frac{q_H-q}{|q_H-q|}$: the direction vector
joining $q \in \partial \Omega$ and $q_H \in H.$ \medskip

 \item $\gamma_H$ resp. $\gamma_{\partial \Omega}$ are the
restriction operators to $H$, resp. $\partial \Omega$. \medskip

  \item $u_{\lambda}^b$: boundary traces of eigenfunctions (see table). Traces on $H$:   $u_{\lambda}^H = \phi_{\lambda}|_H;$ $  u^{H,\nu}_{\lambda} = \frac{1}{\lambda} \partial_{\nu_H} \phi_{\lambda}|_{H}.$
\medskip

\item $\gamma(y, \eta) = \sqrt{1 - |\eta|^2} $ (see
(\ref{a-defn})). \medskip

\item The length of successive double links is
$L_{p}(y,\eta):= |\pi \beta^{p}(y,\eta) - \pi \beta^{p-1}(y,\eta)| + |\pi \beta^{p-1}(y,\eta) - \pi \beta^{p-2}(y,\eta)|.$
\medskip

  \item $\overline{(y,\eta)}$ is the single-link billiard trajectory starting
   at $y \in \partial \Omega$ with direction $\zeta(y,\eta)= \eta + \sqrt{1-|\eta|^{2}} \nu_y$ and endpoint on $\partial \Omega.$
   \medskip

\item $\tau_{\partial \Omega}^H$: the transfer map from
$B^*\partial \Omega$ to $B^*H$  (Definition \ref{TRANSFERMAP}).
\medskip

\item $\tau_H^{\partial \Omega}$: the transfer map from $B^* H$ to
$B^* \partial \Omega$ (\ref{TRANSFERH}). \medskip

\item $\ccal$:   The domain of  $\tau_{\partial \Omega}^{H}$
(Definition \ref{HBILLDEF}); \medskip

\item  $\gcal$: The grazing set. Generalized grazing sets:
$\gcal^*:= \cup_{k \in {\mathbb Z}} \beta^{k}(S^* \partial
\Omega)$ and $\gcal^{*}_{N}: = \cup_{|k| \leq N} \beta^{k}(S^*
\partial \Omega).$

\item  $\Sigma$: The singular set. Generalized singular sets:
$\Sigma^* = \cup_{k \in {\mathbb Z}} \beta^{k} (B_\Sigma^* \Omega)
$ and $ \Sigma_{N}^*= \cup_{|k| \leq N} \beta^{k}(B_\Sigma^*
\Omega).$

\item From Definition \ref{TRANSFERMAP}, Definition
\ref{SECONDTRANSFER}  and (\ref{GAMMAS}): The graph of the
transfer map is denoted by  $ \Gamma_{ \tau_{\partial \Omega}^{H}}
= \Gamma_{\tau_{\partial \Omega}^{H; 1 } } \cup
\Gamma_{\tau_{\partial \Omega}^{H; 2 } }. $
\medskip

\item   The $\epsilon$-regular set: $U_{\epsilon}:= \{ (y,\eta)
\in B^* \partial \Omega; \text{dist} ( \tau_{\partial \Omega}^{H}
( B_{\Sigma}^{*} \partial \Omega \cup \gcal \cup \tcal) > \epsilon
\}.$

\item
  $ N_{\partial \Omega}^{H}(\lambda)$ is the quantum transfer
  operator of (\ref{boundarylayer}). \medskip

  \item $ N^{\partial \Omega}_{\partial \Omega}(\lambda)$ is the
  boundary integral operator of (\ref{boundarylayerb}).  \medskip

\item $\beta_H$ (\S \ref{betas} and (\ref{TRANS})): the once
broken transmission map (or correspondence) on $B^* \partial
\Omega \to B^* \partial \Omega$. It has two branches $\beta_H^k,
(k = 1,2$)
\medskip

\item  $F(\lambda): = N_{\partial \Omega}^H(\lambda)^*
Op_{\lambda} (a) N_{\partial \Omega}^H(\lambda)$ (see
(\ref{FLAMBDA}); \medskip

\item From Lemma \ref{GAMMABETADEF}:   The   canonical relation of
$ N_{\partial \Omega}^H(\lambda)^* Op (a) N_{\partial
\Omega}^H(\lambda)$ equals
$$\Gamma_{\tau_{\partial \Omega}^{H * } } \circ \Gamma_{\tau_{\partial \Omega}^{H  } } =  \Delta_{\ccal
\times \ccal}  \cup  \Gamma_{\beta_H^{ 1}}  \cup \Gamma_{\beta_H^{
2}}.  $$

\item $F_{i j}^{\nu} (\lambda)$: terms in the decomposition of
$F(\lambda)$ (see Proposition \ref{Egorov});
\medskip

\end{enumerate}

\end{proof}

\end{proof}


\begin{thebibliography}{HHHH}

\bibitem[AR]{AR} R. Abraham and J. Robbin,
Transversal mappings and flows. An appendix by Al Kelley W. A.
Benjamin, Inc., New York-Amsterdam 1967.





\bibitem[BR]{BR}  J. Bourgain and Z.  Rudnick,
 Restriction of toral eigenfunctions to hypersurfaces, C. R. Math. Acad. Sci. Paris 347 (2009), no. 21-22, 1249--1253.
 (arXiv:0907.4824.)

\bibitem[BU]{BU} L. A. Bunimovich, On the ergodic properties of nowhere dispersing billiards. Comm. Math.
Phys. 65 (1979), no. 3, 295-312.

\bibitem[Bu]{Bu} N. Burq, Quantum ergodicity of boundary values of eigenfunctions: a control theory approach. Canad. Math. Bull. 48 (2005), no. 1, 3--15.

\bibitem[BGT]{BGT}  N. Burq, P. Gerard, N. Tzvetkov,
Restrictions of the Laplace-Beltrami eigenfunctions to submanifolds. Duke Math. J. 138(3) (2007), 445-486.

\bibitem[CPS]{CPS} F.  Cardoso, V. Petkov, and L.  Stoyanov,
Singularities of the scattering kernel for generic obstacles.
Hyperbolic behaviour of dynamical systems (Paris, 1990). Ann.
Inst. H. Poincaré Phys. Théor. 53 (1990), no. 4, 445--466.

\bibitem[CHTZ]{CHTZ} H. Christianson, H. Hezari, J. Toth and S. Zelditch,
Quantum restriction theorems, II: Cauchy data, Calderon
projections and Dirichlet to Neumann maps along hypersurfaces (in
preparation).

\bibitem[CHTZ2]{CHTZ2} H. Christianson, H. Hezari, J. Toth and S. Zelditch,
Quantum restriction theorems, III: Phase space cross-sections and
monodromy operators (in preparation).

\bibitem[CFS]{CFS} I. P. Cornfeld, S. V. Fomin, Ya. G. Sinai, Ergodic theory, Grundelehren Math. Wiss. 245,
Springer, Berlin, 1982.

\bibitem[DS]{DS} M. Dimassi and J. Sjoestrand, Spectral Asymptotics in the Semi-classical Limit, London
Mathematical Society Lecture Note Series 268, Cambridge University Press, 1999.

\bibitem[F0]{FO}
D. Fujiwara and S. Ozawa,  The Hadamard variational formula for
the Green functions of some normal elliptic boundary value
problems. Proc. Japan Acad. Ser. A Math. Sci. 54 (1978), no. 8,
215--220.




\bibitem[GL]{GL} P. G\'erard and E.  Leichtnam,  Ergodic properties of eigenfunctions for the Dirichlet problem. Duke Math. J. 71 (1993), no. 2, 559--607.

\bibitem[GG]{GG} M. Golubitsky and V. Guillemin, {\it Stable mappings and their singularities}
. Graduate Texts in Mathematics, Vol. 14. Springer-Verlag, New York-Heidelberg, 1973.

\bibitem[GM]{GM} V. Guillemin and R. B. Melrose, The  Poisson summation formula for manifolds with boundary. Advances in Math. 32 (1979), 294-232.



\bibitem[HZ]{HZ} A. Hassell and S. Zelditch,  Ergodicity of boundary values of eigenfunctions, Comm.Math.Phys. Vol-
ume 248, Number 1 (2004) 119 - 168.

\bibitem[HR]{HR} D. A. Hejhal and B. Rackner,
On the topography of Maass waveforms for ${\rm PSL}(2,Z)$.
 Experiment. Math. 1 (1992), no. 4, 275--305.

\bibitem[KTZ]{KTZ} H.  Koch, D.  Tataru, and M. Zworski,
Semiclassical $L\sp p$ estimates. Ann. Henri Poincar\'e 8 (2007),
no. 5, 885--916.

\bibitem[MO]{MO} J. Marklof, S. O'Keefe,  Weyl's law and quantum
ergodicity for maps with divided phase space. With an appendix
``Converse quantum ergodicity'' by Steve Zelditch. Nonlinearity 18
(2005), no. 1, 277--304.

\bibitem[M]{M} R. B. Melrose, Calculus of conormal distributions on manifolds with corners. Internat. Math. Res. Notices 1992, no. 3, 51--61.
(See also: Manifolds with corners and Lie algebras of vector
fields,
http://www.cirm.univ-mrs.fr/videos/2008/exposes/257/Melrose.pdf.)

\bibitem[MS]{MS} R.B. Melrose and J. Sjoestrand, Singularities of boundary value problems. Comm. Pure
Appl. Math., 31 (1978), 593-617.

\bibitem[P]{P} K. Petersen, Ergodic Theory, Cambridge studies in advanced mathematics 2, Cambridge
University Press, 1983.

\bibitem[PS]{PS} V.  Petkov and L. Stoyanov, {\it Geometry of reflecting rays and inverse spectral problems}. Pure and Applied Mathematics (New York). John Wiley $\&$ Sons, Ltd., Chichester, 1992.

\bibitem[PS1]{PS1} V.  Petkov and L.  Stojanov,  On the number of
periodic reflecting rays in generic domains. Ergodic Theory Dynam.
Systems 8 (1988), no. 1, 81--91.

\bibitem[PS2]{PS2} V.  Petkov and L.  Stojanov,  Spectrum of the
Poincar\'e map for periodic reflecting rays in generic domains.
Math. Z. 194 (1987), no. 4, 505--518



\bibitem[Q]{Q} F. Quinn,  Transversal approximation on Banach manifolds. 1970 Global Analysis (Proc. Sympos. Pure Math., Vol. XV, Berkeley, Calif., 1968) pp. 213--222 Amer. Math. Soc., Providence, R.I.


\bibitem[R]{R} A.
Reznikov,  Norms of geodesic restrictions for eigenfunctions on
hyperbolic surfaces and representation theory, arXiv:math/0403437.

\bibitem[SV]{SV} Yu. Safarov and D. Vassiliev, {\it The asymptotic distribution of eigenvalues of partial differential operators.}
  Translations of Mathematical Monographs, 155. American Mathematical Society, Providence, RI, 1997.

\bibitem[So]{So}  C. D. Sogge, {\em Kakeya-Nikodym averages and $L^p$-norms of eigenfunctions}, arXiv:0907.4827.

\bibitem[STZ]{STZ} C. D. Sogge, J. A. Toth and S. Zelditch, About the blowup of quasimodes  on Riemannian
manifolds, to appear in J. Geom. Anal. (arXiv:0908.0688)



\bibitem[S]{S} L. Stojanov,
A bumpy metric theorem and the Poisson relation for generic
strictly convex domains. Math. Ann. 287 (1990), no. 4, 675--696.

\bibitem[T]{T} M. Tacy, Semiclassical $L^p$  Estimates of Quasimodes on
Submanifolds, Arxiv preprint arXiv:0905.2240, 2009.

\bibitem[Ta]{Ta} M. Taylor, Partial differential equations II, Applied Mathematical Sciences, 116.
Springer-Verlag, New York, 1996.


\bibitem[To]{To} J.Toth, $L\sp 2$-restriction bounds for eigenfunctions along curves in the quantum completely integrable case. Comm. Math. Phys. 288 (2009), no. 1, 379--401



\bibitem[TZ]{TZ} J.Toth and S. Zelditch, Counting nodal lines which touch the boundary of an analytic domain, J. Differential Geom. 81 (2009), no. 3, 649--686
(arXiv:0710.0101).

\bibitem[TZ2]{TZ2} J. Toth and S. Zelditch, Quantum ergodic
restriction thoerems, II: manifolds without boundary (in
preparation).

\bibitem[TZ3]{TZ3} J. Toth and S. Zelditch, Matrix elements of
Fourier integral operators, in preparation.

\bibitem[Tr]{Tr} S. Troubetzkoy, Recurrence and periodic billiard orbits in polygons, Regul.
Chaotic Dyn. 9 (2004) 1-12.

\bibitem[U]{U} K. Uhlenbeck,
Generic properties of eigenfunctions. Amer. J. Math. 98 (1976),
no. 4, 1059--1078.

\bibitem[W]{W} M. P.  Wojtkowski,  Design of hyperbolic billiards. Comm. Math. Phys. 273 (2007), no. 2, 283--304.


\bibitem[Z]{Z}  S. Zelditch,
Kuznecov sum formulae and Szego limit formulae on manifolds.
Comm. Partial Differential Equations 17 (1992), no. 1-2, 221--260.

\bibitem[Z2]{Z2} S. Zelditch, New results in quantum ergodicity,
Current Developments in Mathematics 2009 (arXiv:0911.4312; to
appear).

\bibitem[ZZ]{ZZ} S. Zelditch and M. Zworski, Ergodicity of eigenfunctions for ergodic billiards. Comm.
Math. Phys. 175 (1996), 673Ð682.
\end{thebibliography}
\end{document}